\documentclass[11pt, reqno]{amsart}
\pdfoutput=1
\usepackage{etex}
\usepackage{amsmath,amsthm,amsfonts,amssymb,amscd,flafter,pinlabel, booktabs}
\usepackage[mathscr]{eucal}
\usepackage{graphics}
 \usepackage[all,knot,arc]{xy}
  \usepackage[normalem]{ulem}

\usepackage[usenames,dvipsnames]{color}
\usepackage{subfigure}
\usepackage{graphicx}
 \usepackage{tabu}
 \usepackage{float}

\usepackage{bbm}
\usepackage{mathtools}

\usepackage{pgf}
\usepackage{tikz}
\usetikzlibrary{arrows,automata}
\usepackage[latin1]{inputenc}
\usetikzlibrary{matrix,arrows}

\normalfont\upshape

\graphicspath{{./}}
\usepackage[colorlinks=true,linkcolor=blue,citecolor=blue,urlcolor=blue]{hyperref}

\usepackage[lmargin=1.05in,rmargin=1.05in,tmargin=1in,bmargin=1in]{geometry}

\usepackage{enumitem}
\makeatletter
\newcommand{\mylabel}[2]{#2\def\@currentlabel{#2}\label{#1}}
\makeatother

\theoremstyle{definition}

\newtheorem{defn}[equation]{Definition}

\newtheorem*{rmk}{Remark}
\theoremstyle{plain}

\newtheorem{lem}[equation]{Lemma}
\newtheorem{prop}[equation]{Proposition}
\newtheorem{thm}[equation]{Theorem}
\numberwithin{equation}{section}
\numberwithin{figure}{section}
\numberwithin{table}{section}
\newtheorem{ourtheorem}{Theorem}

\newtheorem{ourconjecture}[ourtheorem]{Conjecture}

\theoremstyle{definition}
\newtheorem*{prf}{Proof}
\newtheorem*{note}{Note}

\newcommand{\bdy}{\partial}
\newcommand{\e}{\epsilon}
\newcommand{\alg}{\mathcal A}
\newcommand{\STA}{S_{\alg_P}^A}
\newcommand{\STD}{S_{\alg_P}^D}

\newcommand{\STDp}{S_{\alg_{P'}}^D}

\newcommand{\HD}{\mathcal H}
\newcommand{\F}{\mathbb F}
\newcommand{\Z}{\mathbb Z}
\newcommand{\s}{\mathbf{s}}
\newcommand{\gens}{\mathcal{G}}

\newcommand{\cfahat}{\widehat{\mathit{CFA}}}
\newcommand{\cfdhat}{\widehat{\mathit{CFD}}}
\newcommand{\mza}{m}

\newcommand{\GnA}{\mathcal{G}_{n,\alg_P}}
\newcommand{\GmA}{\mathcal{G}_{m,\alg_P}}
\newcommand{\FnA}{\mathcal{F}_{n,\alg_P}}
\newcommand{\FmA}{\mathcal{F}_{m,\alg_P}}

\newcommand{\GnAp}{\mathcal{G}_{n,\alg_{P'}}}
\newcommand{\FnAp}{\mathcal{F}_{n,\alg_{P'}}}

\newcommand{\fmnpp}{\mathcal F_{m,n,P}}
\newcommand{\fmn}{\mathcal F_{m,n}}
\newcommand{\gmn}{\mathcal G_{m,n}}
\newcommand{\gmnpp}{\mathcal G_{m,n,P}}
\newcommand{\sign}{\mathrm{sign}}
\newcommand{\Ainf}{\mathcal{A}_\infty}
\newcommand{\fmnp}{\mathcal F_{m,n}'}
\newcommand{\fempty}{{\mathcal F_{m,n}^{\circ}}'}
\newcommand{\thin}{\mathrm{t}\mathcal F_{m,n}'}
\newcommand{\plthin}{\mathrm{pt}{\mathcal F_{m,n}'}}
\newcommand{\plempty}{\mathrm{p}{\mathcal F_{m,n}^{\circ}}'}

\newcommand{\cf}{\mathit{CF}}
\newcommand{\hf}{\mathit{HF}}
\newcommand{\hfhat}{\widehat{\hf}}
\newcommand{\cfhat}{\widehat{\cf}}
\newcommand{\hft}{\widetilde{\rm{HF}}}

\newcommand\alphas{\mbox{\boldmath$\alpha$}}
\newcommand\betas{\mbox{\boldmath$\beta$}}

\newcommand{\gen}[1]{\mathbf{#1}}
\newcommand{\x}{\gen{x}}
\newcommand{\y}{\gen{y}}
\newcommand{\z}{\gen{z}}
\newcommand{\gt}{\gen{t}}
\newcommand{\gs}{\gen{s}}
\newcommand{\w}{\gen{w}}
\newcommand{\ga}{\gen{a}}
\newcommand{\gb}{\gen{b}}

\newcommand{\ZZ}{\mathcal Z}
\newcommand{\algz}{\alg_{\Z}}
\newcommand{\gr}{\mathrm{gr}}

\DeclareMathOperator{\id}{id}

\begin{document}

 \title[{Bordered Floer homology with integral coefficients for manifolds with torus boundary}]{Bordered Floer homology with integral coefficients for manifolds with torus boundary}

\author[Douglas Knowles]{Douglas Knowles}
\address{Learning Strategies Center, Cornell University\\ Ithaca, NY 14850}
\email{ddk62@cornell.edu}
\urladdr{\href{https://math.cornell.edu/doug-knowles}{https://math.cornell.edu/doug-knowles}}
\author{Ina Petkova}
\address {Department of Mathematics, Dartmouth College\\ Hanover, NH 03755}
\email {ina.petkova@dartmouth.edu}
\urladdr{\href{http://math.dartmouth.edu/~ina}{http://math.dartmouth.edu/~ina}}

\begin{abstract}
We provide a combinatorial definition of a bordered Floer theory with $\Z$ coefficients for manifolds with torus boundary. Our bordered Floer structures recover the combinatorial Heegaard Floer homology from \cite{hfz}.
\end{abstract}

\maketitle



\section{Introduction}

In \cite{osz8, osz14}, Ozsv\'ath and Szab\'o introduced a package of invariants for closed, oriented 3-manifolds,  called Heegaard Floer homology. Using a Heegaard diagram presentation of the 3-manifold $Y$, one constructs a chain complex by counting certain holomorphic disks in a symmetric product of the Heegaard surface. The homology of this complex is an invariant of the 3-manifold. The simplest version  $\hfhat(Y; \Z/2\Z)$ is a finitely generated vector space over $\F_2 = \Z/2\Z$, and is obtained by a mod $2$ count of holomorphic disks. By using coherent orientation systems, one can also count with sign, resulting in a finitely generated abelian group $\hfhat(Y; \Z)$. Many 3-dimensional applications only need the theory over $\F_2$. Still, the theory over $\Z$ is generally much more powerful \cite{osz-tri-symp, pm, jm2}. 

In \cite{swnice}, Sarkar and Wang gave a combinatorial description of the Heegaard Floer chain complex over $\F_2$ using a special class of Heegaard diagrams called \emph{nice}. Ozsv\'ath, Stipsicz, and Szab\'o then gave a combinatorial proof of the invariance of a stabilized version $\hft$ of the homology \cite{nice}, and later provided an integral lift of this combinatorial definition \cite{hfz}. While these combinatorial reformulations are significant advances in resolving the computational difficulties arising from the analytical nature of the original construction, modifying a Heegaard diagram to a nice one typically comes at the cost of increasing the size of the complex immensely. Thus, combining combinatorial methods with cut-and-paste techniques is the natural next step.

In \cite{bfh2}, Lipshitz, Ozsv\'ath, and Thurston generalized Heegaard Floer homology to a theory for 3-manifolds with parametrized boundary. To a $3$-manifold $Y$ with parametrized boundary, one can associate an $\Ainf$-module, also called a \emph{type $A$ structure},  $\cfahat(Y)$ over a differential graded algebra (DGA) $\alg(\bdy Y)$, or equivalently a \emph{type $D$ structure} (roughly,  a differential graded module) $\cfdhat(Y)$ over $\alg(-\bdy Y)$. If $Y$ is a closed 3-manifold obtained by a gluing $Y_1\cup_{\bdy}Y_2$, then the homology of the derived tensor product $\cfahat(Y_1)\widetilde{\otimes}_{\alg(\bdy Y_1)} \cfdhat(Y_2)$  (which often has a smaller model) is isomorphic to $\hfhat(Y; \F_2)$. These structures can be defined  analytically, as well as combinatorially using nice bordered Heegaard diagrams. 

In this paper we extend the definition of bordered Floer homology to $\mathbb{Z}$-coefficients for torus boundary. We follow the combinatorial approach from \cite{hfz} and define \emph{formal generators} and \emph{formal flows} which retain certain combinatorial data of actual generators and flows (the objects counted by the Heegaard Floer differential) for a nice bordered Heegaard diagram. We then define \emph{bordered sign assignments}, that is, maps from formal flows to $\{\pm 1\}$ which satisfy certain properties. We then show that these sign assignments are unique in a certain sense, and use them to count flows on an actual nice bordered Heegaard diagram with sign, to obtain the following.

\begin{ourtheorem}\label{thm:ind}
Given a nice bordered Heegaard diagram $\HD$ for a bordered 3-manifold with torus boundary, along with some additional choices, we define a right type $A$ structure $\cfahat(\HD; \Z)$ over $\mathcal A_{\Z}(\bdy \HD)$, and a left type $D$ structure $\cfdhat(\HD; \Z)$ over $\mathcal A_{\Z}(-\bdy \HD)$, where $\mathcal A_{\Z}(\bdy \HD)\cong \mathcal A_{\Z}(-\bdy \HD)$ is a graded algebra over $\Z$. Moreover, these structures do not depend (in an appropriate sense)\footnote{See Theorems~\ref{thm:indA} and \ref{thm:indD} for a precise statement of this result.}
 on the choices made in their definitions. 
\end{ourtheorem}

We expect that these structures are invariant under the choice of nice diagram.

\begin{ourconjecture}
The structures  $\cfahat(\HD; \Z)$ and  $\cfdhat(\HD; \Z)$ are invariants of the underlying bordered 3-manifold. 
\end{ourconjecture}

We  show that these bordered structures recover the combinatorial theory  $\hft$ over $\Z$ from \cite{hfz}. To state our pairing result, we remark that the bordered sign assignments used to define $\cfahat(\HD; \Z)$ fall into four equivalence classes, and so do those used to define $\cfdhat(\HD; \Z)$. We require a certain \emph{compatibility} property in order to pair sign assignments, and show that for any type $A$ equivalence class, there is a compatible type $D$ equivalence class. We then have the following pairing theorem.

\begin{ourtheorem}\label{thm:pair}
Suppose $\HD_1$ and $\HD_2$ are nice bordered Heegaard diagrams for 3-manifolds with torus boundary with $\partial\HD_1=-\partial\HD_2$. Further suppose that $\cfahat(\HD_1;\Z)$ and $\cfdhat(\HD_2;\Z)$ are defined using compatible sign assignments. Let $\HD=\HD_1 \cup_{\partial\HD_1} \HD_2$. Then 
\[H_*(\cfahat(\HD_1;\Z) \boxtimes \cfdhat(\HD_2;\Z)) \cong  \widetilde{\rm{HF}}(\HD; \Z).\]
\end{ourtheorem}

As a small application, one may hope to prove the existence of a surgery exact triangle for the integral version of the combinatorially defined theory $\widetilde{\rm{HF}}$ from \cite{hfz}. In Section~\ref{sec:tri}, we describe three nice bordered Heegaard diagrams $\HD_{\infty}$, $\HD_{-1}$, and $\HD_0$ for the three solid tori with slopes $\infty$, $-1$, and $0$, respectively.  We make the following conjecture (c.f.\ \cite[Section 11.2]{bfh2} for the proof of the $\F_2$ analogue.)

\begin{ourconjecture}\label{conj:tri}
Suppose $Y$ is a closed, oriented 3-manifold and  $K$ is a framed knot in $Y$. Let $\HD$ be a nice bordered Heegaard diagram for the infinity-framed complement $Y\setminus \mathrm{Nbhd} (K)$,  so that $\HD_{i}'\coloneqq \HD\cup \HD_{i}$ is a nice Heegaard diagram for the 3-manifold $Y_i$ obtained from $Y$ by $i$-surgery on $K$, for $i\in \{\infty, -1,0\}$.   Then we have an exact sequence 
\[\cdots \to  \widetilde{\rm{HF}}(\HD_{\infty}'; \Z)\to  \widetilde{\rm{HF}}(\HD_{-1}'; \Z)\to  \widetilde{\rm{HF}}(\HD_0'; \Z) \to \cdots\]
\end{ourconjecture}

However, using Theorem~\ref{thm:pair} to prove this conjecture in a way analogous to \cite[Section 11.2]{bfh2} does not quite work as hoped. We discuss this at the end of the paper.

\subsection*{Organization} In Section~\ref{sec:background}, we review the necessary background, including $\Ainf$ structures with signs, bordered Floer homology with characteristic two, and sign assignments for closed nice diagrams. In Section~\ref{sec:alg}, we extend the algebra associated to the torus to an algebra over $\Z$. In Section~\ref{sec:bord-sa}, we generalize the notions of formal flows, generators, and sign assignments to the bordered setting. In Section~\ref{sec:cfa}, we discuss the existence and uniqueness of sign assignments of ``type $A$", and define the type $A$ structure $\cfahat$ with signs. In Section~\ref{sec:cfd} we carry out the analogous work to define $\cfdhat$ with signs, and complete the proof of Theorem~\ref{thm:ind}. In Section~\ref{sec:pairing} we prove Theorem~\ref{thm:pair}. Finally, in Section~\ref{sec:tri}, we include a short discussion related to Conjecture~\ref{conj:tri}.

\subsection*{Acknowledgments} We thank Robert Lipshitz for suggesting this problem, as well as  Paolo Ghiggini, Adam Levine, and Peter Ozsv\'ath for helpful conversations. I.P.\ received support from NSF grant DMS-1711100.


\section{Background}\label{sec:background}

\subsection{$\Ainf$ structures over $\Z$.}

Since definitions for the relevant algebraic structures appear only over $\Z/2$ in the bordered Floer literature, 
we begin our background discussion by reminding the reader the more general setup over $\Z$. Our summary closely follows \cite[Section 12]{osz-bord2}. Fix a ring $R$, not necessarily of characteristic two. In the later sections, $R$ will be the ring $\Z^2$ of idempotents of the torus algebra $\alg_P$ defined in Section~\ref{sec:alg}. For the discussion below, we will assume the spaces are graded by $\Z/2$. There are, of course, more general definitions in the literature.

\begin{defn}
An $\Ainf$-algebra $\alg$ is a $\Z/2$-graded $R$-bimodule $A$, equipped with degree 0 $R$-linear multiplication maps for $i\geq1$
\[\mu_i: A^{\otimes i} \to A[2-i] \]
that satisfy the following conditions for $n\geq1$
\[\sum_{r+s+t=n} (-1)^{r+st} \mu_{r+1+t}(\text{Id}^{\otimes r} \otimes \mu_s \otimes \text{Id}^{\otimes t})=0.\]
We say that such an $\Ainf$-algebra is \emph{operationally bounded} if $\mu_i=0$ for sufficiently large $i$.
\end{defn}

\begin{note}
When a tensor product of graded maps is applied to specific elements, signs may appear based on the gradings of the maps and the elements themselves. In general, if we have $\Z/2$-graded abelian groups $G,G',H,H'$ and $\Z/2$-graded homomorphisms $g:G \to G'$ and $h:H \to H'$, then their tensor product $g \otimes h: G \otimes H \to G' \otimes H'$ is defined as
\[(g \otimes h)(x \otimes y)=(-1)^{|h||x|}(g(x) \otimes h(y)),\]
if $x$ is homogeneous of degree $|x|$, and $|h|$ denotes the degree of $h$.
\end{note}

\begin{note}
The algebras that arise in bordered Heegaard Floer homology are all just differential graded algebras (DGAs), i.e. for $i \geq 3$ we have that $\mu_i=0$. This also implies that they are operationally bounded. For the remainder of this section, we assume that $\alg$ is a DGA. We also assume that $\alg$ is \emph{unital}, i.e.\ there exists a unit $I_A\in A$ such that  $\mu_2(a,I_A) = a = \mu_2(I_A, a)$ for all $a\in A$.
\end{note}

Given these two notes, our only (non-trivial) $\Ainf$-algebra relations are
\begin{align*}
d_A^2(a)&=0  \\ 
d_A(a \cdot b)&=d_A(a) \cdot b + (-1)^{|a|}a \cdot d_A(b) \\
a\cdot (b\cdot c) &= (a\cdot b) \cdot c
\end{align*}
for all $a,b,c \in A$. Note we have used the more common notation of $d_A$ for $\mu_1$ and $\cdot$ for $\mu_2$.

\begin{defn}
A right type $A$ structure $\mathcal{M}$ over a DGA $\alg$ is a $\Z/2$-graded $R$-module $M$ with degree 0 $R$-linear maps for $j \geq 0$
\[m_{j+1}: M \otimes A[1]^{\otimes j} \to M[1] \]
that satisfy the following conditions for $n \geq 1$ and arbitrary homogeneous elements $x \in M$ and $a_1,\dots,a_{n-1} \in A$:
\begin{equation}\label{eqn:Arelns}
\begin{split}
0 =  &\sum_{i=0}^{n-1} (-1)^{i(n-1)} m_{i+1}(m_{n-i}(x,a_1,\dots,a_{n-i-1}),a_{n-i},\dots,a_{n-1}) \\
	&+ \sum_{i=1}^{n-2} (-1)^{i} m_{n-1}(x,a_1,\dots, a_{i-1}, \mu_2(a_i,a_{i+1}),a_{i+2},\dots,a_{n-1}) \\
	&+ (-1)^{n-1} \sum_{i=1}^{n-1} (-1)^{|x|+|a_1|+\dots+|a_{i-1}|}m_{n}(x,a_1,\dots, a_{i-1}, \mu_1(a_i),a_{i+1},\dots,a_{n-1})
\end{split}
\end{equation}
We say that $\mathcal{M}_{\alg, \Z}$ is \emph{bounded} if $m_i=0$ for sufficiently large $i$. \footnote{We caution the reader comparing Equation~\ref{eqn:Arelns} with \cite[Equation (12.7)]{osz-bord2}  that there is a typo in \cite[Equation (12.7)]{osz-bord2} -- the first occurrence of $(-1)$ should be raised to the power $i(n-i)$ instead of the power $i$.}
\end{defn}

\begin{defn}
A left type $D$ structure  $\mathcal{N}$ over a DGA $\alg$ is a $\Z/2$-graded $R$-module $N$ equipped with a degree 0 $R$-linear map 
\[\delta^1: N \to (A \otimes N) [1] \]
with the condition that
\[0=(\mu_2 \otimes \id_N) \circ (\id_A \otimes \delta^1) \circ \delta^1 + (\mu_1 \otimes \id_N) \circ \delta^1. \]
Let $\delta^k$ be defined inductively as $\delta^0=\id_N$ and $\delta^j=(\id_{A^{\otimes j-1}} \otimes \delta^1)\circ \delta^{j-1}$. We say $\mathcal{N}$ is \emph{bounded} if $\delta^k=0$ for sufficiently large $k$.
\end{defn}

Assuming that at least one of the two structures is bounded, we can pair a type $A$ structure $\mathcal{M}$ and a type $D$ structure $\mathcal{N}$  using the $\alg_{\infty}$ box tensor product in the following way. Define a complex $\mathcal{M}\boxtimes \mathcal{N}$, whose underlying space is $\mathcal{M}\otimes_{R} \mathcal{N}$, and whose differential is given by 
\[\partial^{\boxtimes}(x \otimes y):= \displaystyle\sum_{k=0}^{\infty} (m_{k+1} \otimes \id_N)\circ (\id_M \otimes \delta^k)(x\otimes y). \]
Note that this sum is finite, since one of these structures is assumed to be bounded.

In this paper we define type $A$ and $D$ structures for a particular kind of bordered Heegaard diagram called a \emph{nice} diagram. The type $A$ structure associated to a nice diagram
has  no higher multiplication maps, i.e.\ $m_i$ for $i \geq 3$ is identically zero; note that this implies that the type $A$ structure is bounded. In this case, the type $A$ relations become
\begin{align*}
d_M^2(m)&=0 \\
d_M(m \cdot a)&=d_M(m) \cdot a + (-1)^{|m|}(m \cdot d_A(a))\\
m\cdot (a \cdot b) &= (m\cdot a) \cdot b
\end{align*}
for all homogeneous $m \in M$ and $a, b \in A$, where we have once again used the more common notations of $d_M, d_A$, and $\cdot$, instead of $m_1, \mu_1$, and $m_2$ respectively.
The type $D$ structure associated to a nice diagram is also bounded; this follows directly from \cite[Lemma 6.5]{bfh2}, which says that admissible diagrams  result in bounded type $D$ structures, and \cite[Lemma 8.3]{bfh2}, which says that nice diagrams are admissible, since the  finiteness argument in the proof of \cite[Lemma 6.5]{bfh2} is valid if one counts holomorphic disks with signs as well.

In the case of nice diagrams with torus boundary, the box tensor differential becomes
\[\partial^{\boxtimes}(x \otimes y)= m_1(x) \otimes y + (-1)^{|x|}\sum_{j}m_2(x,a_{j}) \otimes y_j,\]
where $x$ is homogeneous, and $a_j$ are the algebra elements that appear if we write out $\delta^1(y)$ in terms of generators as $\delta^1(y) = \displaystyle\sum_j a_j \otimes y_j$.

\subsection{Bordered Heegaard Floer homology over $\F_2$}

Bordered Floer homology is an extension of Heegaard Floer homology to manifolds with boundary \cite{bfh2}. To a parametrized surface $F$, one associates a differential algebra $A(F)$ over the ground ring $\F_2$, and to a manifold whose boundary is identified with $F$, a right type $A$ structure $\cfahat(Y)$ over $A(F)$, or a left type $D$ structure over $A(-F)$. These structures are invariants of the manifold up to homotopy equivalence, and their ``box tensor product" recovers Heegaard Floer homology. The structures are defined using bordered Heegaard diagrams. When the diagrams are ``nice", i.e. all regions away from the basepoint are bigons and rectangles, the differential is combinatorial, given by counting empty embedded bigons and rectangles with only convex corners.

\subsubsection{The torus algebra over $\F_2$}\label{sssec:algz2}

Recall the following general definition from \cite[Section 3.2]{bfh2}:
\begin{defn}\label{def:pmc}
A  \emph{pointed matched circle} is a quadruple  $\ZZ = (Z, {\bf a}, M, z)$ consisting of
\begin{itemize}
\item[-]  an oriented circle $Z$ with a basepoint $z$
\item[-] a set  of $4k$ points ${\bf a}=\{a_1, \ldots, a_{4k}\}$ on $Z\setminus z$ 
\item[-] a matching $M: {\bf a}\to [2k]$ (where $[2k] = \{1, \ldots, 2k\}$) such that surgery along the matched pairs yields a single circle.
\end{itemize}
\end{defn}

A pointed matched circle $\ZZ$ as above specifies a surface $F(\ZZ)$ of genus $k$. We omit the discussion of the construction, and refer the reader to \cite[Section 3.2]{bfh2}. We  write $-\ZZ$ to denote the pointed matched circle with the oppositely oriented circle $-Z$ and otherwise the same data.

Given a pointed matched circle $\ZZ$, one can define an algebra $\alg(\ZZ)$ in terms of Reeb chords in $(Z, {\bf a})$. We omit the general description, and focus on the torus case ($k=1$) below.

There is a unique pointed matched circle $\ZZ$ that represents the surface of genus $1$. The corresponding algebra $\alg(\ZZ)$ has three non-trivial summands. The interesting one, and the only one that acts non-trivially on modules coming from Heegaard diagrams, is the summand $\alg$, denoted $\alg(\ZZ,0)$ in \cite{bfh2}, defined as follows.

The algebra $\alg$ is generated over $\F_2$ by two idempotents denoted $\iota_1$ and $\iota_2$, and six non-trivial elements denoted $\rho_1$, $\rho_2$, $\rho_3$, $\rho_{12}$, $\rho_{23}$, and $\rho_{123}$. The differential is zero,  the non-zero products are
\[\rho_1\rho_2 = \rho_{12}\qquad \rho_2\rho_3=\rho_{23} \qquad \rho_1\rho_{23} = \rho_{123} \qquad \rho_{12}\rho_3 = \rho_{123}\]
and the compatibility with the idempotents is given by 
\begin{align*}
\rho_1 &= \iota_1 \rho_1\iota_2 \quad &\rho_2 &= \iota_2 \rho_2\iota_1 \quad &\rho_3 &= \iota_1 \rho_3\iota_2\\
\rho_{12} &= \iota_1 \rho_{12}\iota_1 \quad &\rho_{23} &= \iota_2 \rho_{23}\iota_2 \quad &\rho_{123} &= \iota_1 \rho_{123}\iota_2.
\end{align*}
Observe that $\iota_1$ and $\iota_2$ generate a subalgebra of $\alg$, which we denote $\mathcal I$; in \cite{bfh2}, this subalgebra is denoted $\mathcal I(\ZZ,0)$.

We now provide a graphical description of $\alg$. With the notation $\ZZ = (Z, {\bf a}, M, z)$  in mind, write ${\bf a} = \{a_1,a_2,a_3,a_4\}$, with points indexed in order of appearance along $Z\setminus z$, and write $M(a_1)=M(a_3)=1, M(a_2)=M(a_4)=2$. The six non-trivial elements of $\alg$ correspond to the six oriented subarcs of $Z\setminus z$ with endpoints in ${\bf a}$, as follows.  When examining $Z \setminus \mathbf{a}$, label the component containing $z$ with 0, then label the remaining components in order (induced by the orientation of $Z$) with 1, 2, and 3. Then the oriented subarc of $Z\setminus z$ that traverses region 1 corresponds to  the algebra element $\rho_1$. Similarly the subarc that traverses both region 2 and 3 corresponds to  the algebra element $\rho_{23}$, and so on.
 The two idempotents $\iota_1$ and $\iota_2$  are given by the pairs $M^{-1}(1) = \{a_1, a_3\}$ and $M^{-1} (2)= \{a_2, a_4\}$, respectively. See  Figure~\ref{fig:alg-def}.

\begin{figure}[h]
\centering
  \labellist
 	\pinlabel $\iota_1$ at 433 569
	\pinlabel $\iota_1$ at 433 638
	\pinlabel $\rho_1$ at 22 583
	\pinlabel $\rho_{2}$ at 90 620
	\pinlabel $\rho_{3}$ at 158 657
	\pinlabel $\rho_{12}$ at 226 602
	\pinlabel $\rho_{23}$ at 294 638
	\pinlabel $\rho_{123}$ at 362 620
	\pinlabel $\iota_2$ at 501 602
	\pinlabel $\iota_2$ at 501 674
		\endlabellist
 \includegraphics[scale=.7]{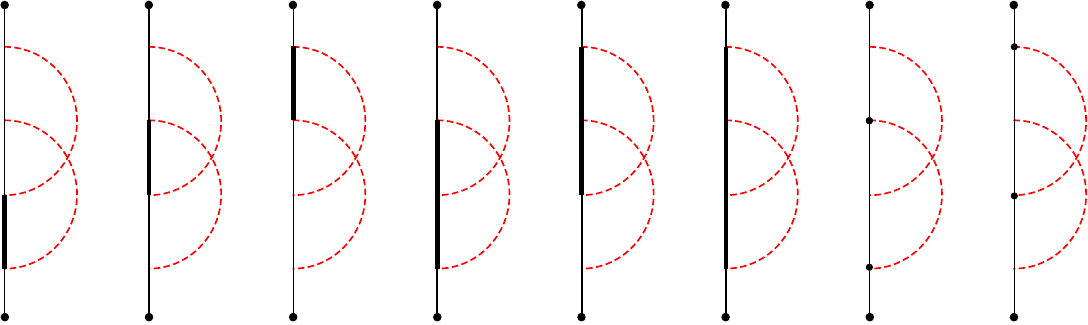}
       \caption{The arcs and pairs of points on $\ZZ$ associated to the generators of $\alg = \alg(\ZZ,0)$. In each of the eight diagrams, the top and bottom point are identified, and represent the basepoint $z$. The circle $Z$ is oriented from bottom to top.}\label{fig:alg-def}
\end{figure}

The multiplication of two non-trivial generators corresponds to concatenation of the corresponding arcs, which only gives a nontrivial result if the endpoint of the first arc is the same as the starting point of the second.   Left (resp.~right) compatibility with an idempotent encodes whether the starting point (resp.~endpoint) of the subarc contains a point of that idempotent; in other words, for a non-trivial element with starting point $a_i$ and ending point $a_j$, compatibility on the left with an idempotent is determined by the value of $M(a_i)$, and compatibility on the right by $M(a_j)$.

\subsubsection{Bordered Heegaard diagrams}\label{sssec:bhd}
Here we lay out the requisite background for bordered Heegaard diagrams and their associated type $A$ and type $D$ structures over $\F_2$. Note the definition of these kinds of structures over $\F_2$ has the same definition as those over $\mathbb{Z}$ if one ignores the signs. For more details we refer the reader to \cite{bfh2}.
\begin{defn}
A \emph{bordered Heegaard diagram} is a tuple $\HD=(\Sigma,\alphas,\betas,z)$ where
\begin{itemize}
\item[-] $\Sigma$ is a compact, oriented surface with one boundary component of genus $g$;
\item[-] $\betas=\{\beta_1,\dots,\beta_g\}$ is a $g$-tuple of pairwise disjoint circles in the interior of $\Sigma$;
\item[-] $\alphas=\{\alpha_1^c,\dots,\alpha_{g-k}^c,\alpha_1^a,\dots,\alpha_{2k}^a\}$ is a $(g+k)$-tuple of pairwise disjoint curves consisting of $g-k$ circles $\alphas^c=\{\alpha_1^c,\dots,\alpha_{g-k}^c\}$ in the interior of $\Sigma$ and $2k$ arcs $\alphas^a=\{\alpha_1^a,\dots,\alpha_{2k}^a\}$ properly embedded on $\Sigma$ (so $\partial\alpha_i^a \in \partial\Sigma$ for all $i$);
\item[-] $z$ is a point on $\partial\Sigma \setminus (\alphas \cap \partial\Sigma)$.
\end{itemize}
\end{defn}

There is a pointed matched circle $\ZZ=(Z,\mathbf{a},M)$ naturally associated to a bordered Heegaard diagram $\HD$ with $Z=\partial\Sigma$, $\mathbf{a}=\alphas \cap \partial\Sigma$, and the matching $M$ given by $M(p)=i$ if and only if $p \in \alpha_i^a \cap \partial\Sigma$. We will sometimes write $\ZZ$ as $\partial\HD$ in this context.

The construction to recover the $3$-manifold with boundary associated to a bordered Heegaard diagram $\HD=(\Sigma,\alphas,\betas,z)$ with $\ZZ=\partial\HD$ goes as follows. Begin by taking $[-\epsilon, 0] \times Z$ to be a closed collar neighborhood of $\partial\Sigma$ such that $0 \times Z$ is identified with $\partial\Sigma$, and also take a tubular neighborhood $Z \times [0,1]$ of $Z$ in $F(\ZZ)$. Now glue $\Sigma \times [0,1]$ to $[-\epsilon,0] \times F(\ZZ)$ by identifying $([-\epsilon,0] \times Z) \times [0,1] \subset \Sigma \times [0,1]$ with $[-\epsilon,0] \times (Z \times [0,1]) \subset [-\epsilon,0] \times F(\ZZ)$. Denote the result of this gluing by $Y_0$.

Next, attach a 3-dimensional 2-handle to each $\beta_i \times 1 \subset \Sigma \times [0,1] \subset Y_0$ and $\alpha_i^c \times 0 \subset \Sigma \times [0,1] \subset Y_0$. We will denote the resulting manifold by $Y_1$. Note we now have three boundary components: a sphere $S^2$ meeting $\Sigma \times 1$, a surface $\Sigma'$ of genus $2k$ meeting $\Sigma \times 0$, and our desired final boundary component given by a surface identified with $0 \times F(\ZZ) \subset [-\epsilon,0] \times F(\ZZ)$. 

Glue a 3-ball to the $S^2$ boundary component, and join each $\alpha_i^a \times 0 \subset \Sigma'$ to the co-core of the corresponding handle in $-\epsilon \times F(\ZZ)$ to form a closed curve. Attach a 3-dimensional 2-handle to each of these new circles, and denote the resulting manifold by $Y_2$. Then $Y_2$ still has a boundary component identified with $F(\ZZ)$, and a $S^2$ boundary component meeting $\Sigma \times 0 \subset Y_0$. Glue a 3-ball to this $S^2$ component and the resulting manifold $Y$ is our desired 3-manifold with boundary identified with $F(\ZZ)$. We will occasionally use $\phi$ to notate the identification of $F(\ZZ)$ with $\bdy Y$ and write $(Y, \ZZ, \phi)$. 

For the purposes of this paper, we restrict the rest of the background discussion to the case of manifolds with torus boundary. We will further restrict to a special  class of bordered Heegaard diagrams, called \emph{nice diagrams}. Nice diagrams are diagrams where each region not containing the basepoint $z$ is either a rectangle or a bigon. As a result, the structure maps for the associated type $A$ or type $D$ structure are encoded in the combinatorics of the diagram. One can show that any bordered 3-manifold admits a nice diagram. Let $\HD$ be a nice bordered Heegaard diagram for a bordered  manifold $Y$ with torus boundary, and let $\partial\HD$ denote the pointed matched circle associated to $\HD$. We define the  generators for $\HD$  to be the subsets of $\alphas \cap \betas$ with:
\begin{itemize}
\item[-] $g$ total elements
\item[-] exactly one element on each closed circle $\beta_i$ and exactly one on each $\alpha_i^c$
\item[-] at most one element on each arc $\alpha_i^a$.
\end{itemize}
We denote the set of generators for $\HD$ by $\mathcal{G}(\HD)$.

The authors of \cite{bfh2} associate to each generator a $\textrm{spin}^c$ structure of the manifold $Y$, and we will write $\mathcal{G}(\HD,\mathfrak{s})$ to denote the set of generators with a given associated $\textrm{spin}^c$ structure $\mathfrak{s} \in \textrm{Spin}^c(Y)$. Let $X(\HD)$ and $X(\HD,\mathfrak{s})$ be the $\mathbb{F}_2$-vector spaces spanned by $\mathcal{G}(\HD)$ and $\mathcal{G}(\HD,\mathfrak{s})$ respectively (i.e. $X(\HD)=\bigoplus_{\mathfrak{s} \in \textrm{Spin}^c(Y)} X(\HD,\mathfrak{s})$). Let $o(\x)\in \{1,2\}$ be the index of the $\alpha$-arc occupied by $\x$ and $\overline{o(\x)}$ be the index of the unoccupied arc, i.e.\ the complement of $o(\x)$ in $\{1,2\}$. We assume the $\alpha$-arcs are ordered as they appear along $\bdy \HD$ starting from $z$, so that after identifying $\bdy \HD$ (or similarly $-\bdy \HD$) with $\ZZ$, two points $p,q$ are the endpoints of $\alpha_i$ if and only if they are matched in $\ZZ$ as $M(p)=i = M(q)$.
 We are now ready to define a right type $A$ structure $\cfahat(\HD)$ over $\alg(\bdy \HD)$ and a left type $D$ structure $\cfdhat(\HD)$ over $-\bdy\HD$.

Given a diagram $\HD$, if we identify $\bdy \HD$ with $\ZZ$ we can define a right type $A$ structure $\cfahat(\HD,\mathfrak{s})$ over $\alg$ as follows. 
Let $\cfahat(\HD,\mathfrak{s})$ be the module generated over $\mathbb{F}_2$ by $\mathcal{G}(\HD,\mathfrak{s})$. 

Define a right action of $\mathcal{I}$ on $\cfahat(\HD,\mathfrak{s})$ by 
\[\x\cdot \iota_t = \begin{cases}
\x \quad \textrm{ if } t = o(\x) \\
0 \quad \textrm{otherwise.}
\end{cases}
\]

It remains to define the $\alg_\infty$ structure maps $m_{n+1}$. Because we are working in the context of a nice diagram, $m_{n+1}$ is trivial for $n \geq 2$. The map $m_1: \cfahat(\HD,\mathfrak{s}) \to \cfahat(\HD,\mathfrak{s})$ counts empty rectangles and bigons in the interior of the surface of $\HD$. The map $m_2: \cfahat(\HD,\mathfrak{s}) \otimes_{\mathcal{I}} \alg \to \cfahat(\HD,\mathfrak{s})$ instead counts empty rectangles with a single edge on the boundary of $\HD$ (these are called ``half strips" in \cite[Section 8]{bfh2}). Specifically $m_2(\x, \rho_I) = \y$ if and only if there is an empty rectangle from $\x$ to $\y$ with an edge on $\bdy\HD$ representing $\rho_I$; in addition, $m_2(\x, \iota_{o(\x)}) = \x$.

Let $\cfahat(\HD)=\bigoplus_{\mathfrak{s}\in\textrm{Spin}^c(Y)}\cfahat(\HD,\mathfrak{s})$.

Given a diagram $\HD$, if we identify $-\bdy \HD$ with $\ZZ$ we can define a left type $D$ structure $\cfdhat(\HD,\mathfrak{s})$ over $\alg = \alg(-\partial\HD)$ as follows. Let $\cfdhat(\HD,\mathfrak{s})$ be the module generated over $\mathbb{F}_2$ by $\mathcal{G}(\HD,\mathfrak{s})$. 

Define a left action of $\mathcal{I}$ on $\cfdhat(\HD,\mathfrak{s})$ by 
\[\iota_t \cdot \x= \begin{cases}
\x \quad \textrm{ if } t = \overline{o(\x)} \\
0 \quad \textrm{otherwise.}
\end{cases}
\]

It remains to define the type $D$ structure map $\delta^1:\cfdhat(\HD,\mathfrak{s}) \to \alg \otimes_{\mathcal{I}} \cfdhat(\HD,\mathfrak{s})$. Given $\x, \y\in \mathcal{G}(\HD,\mathfrak{s})$, let $\pi_2^{\mathrm{int}}(\x,\y)$ be the set of empty bigons and rectangles from $\x$ to $\y$ in the interior of the surface of $\HD$, and let $\pi_2^{\bdy}(\x,\y)$ be the set of empty rectangles with a single edge on the boundary of $\HD$. The structure map is defined on generators as  
\[\delta^1(\x):=\displaystyle\sum_{\y \in \mathcal{G}(\HD,\mathfrak{s})}\displaystyle\sum_{\phi \in \pi_2^{\mathrm{int}}(\x,\y)}\iota_{\overline{o(\y)}} \otimes \y + \sum_{\y \in \mathcal{G}(\HD,\mathfrak{s})}\displaystyle\sum_{\phi \in \pi_2^{\bdy}(\x,\y)}\rho^{\phi} \otimes \y,\]
where $\rho^{\phi}$ is the algebra element corresponding to the edge of $\phi$ that lies on  $\bdy \HD$.

Let $\cfdhat(\HD)=\bigoplus_{\mathfrak{s}\in\textrm{Spin}^c(Y)}\cfdhat(\HD,\mathfrak{s})$.

\subsubsection{Gluing}\label{sssec:glue}

Let $\HD_1=(\Sigma_1, \alphas_1, \betas_1, z_1)$ and $\HD_2=(\Sigma_2, \alphas_2, \betas_2, z_2)$ be nice bordered Heegaard diagrams  for bordered manifolds $(Y_1, \ZZ,\phi_1)$ and $(Y_2, -\ZZ, \phi_2)$ with torus boundary, respectively. We can create a closed Heegaard diagram by gluing $\HD_1$ and $\HD_2$ along their boundaries according to the marked points of their common pointed matched circle. This diagram 
\[ \HD=\HD_1 \cup_\bdy \HD_2=(\Sigma_1 \cup_\bdy \Sigma_2, \alphas_1 \cup_\bdy \alphas_2, \betas_1 \cup \betas_2, z_1=z_2)\] is a closed Heegaard diagram for $Y_1 \cup_\bdy Y_2= Y_1 \cup_{\phi_2 \circ \phi_1^{-1}} Y_2=Y$.  Now let $\cfahat(\HD_1)$ be the left type $A$ structure associated to $\HD_1$ and $\cfdhat(\HD_2)$ be the right type $D$ structure associated to $\HD_2$. Since $\HD_1$ and $\HD_2$ are nice diagrams, these structures are bounded. As proven in \cite{bfh2}, we then have 
\[\cfhat(\HD)\cong \cfahat(\HD_1) \boxtimes \cfdhat(\HD_2) .\] More specifically, the generators of $\cfahat(\HD_1) \boxtimes \cfdhat(\HD_2)$ take the form $x:=x_1 \otimes x_2$ where $x_1$ is a generator of $\HD_1$, $x_2$ is  a generator of $\HD_2$, and $o(x_1) \neq o(x_2)$. Then from our previous discussion, we know that on a nice diagram with torus boundary, $\partial^\boxtimes(x_1 \otimes x_2)=m_1(x_1) \otimes x_2 + (-1)^{|x_1|}\sum_{j}m_2(x_1,a_{j}) \otimes y_j$ where $\delta^1(x_2)=\sum_j a_j \otimes y_j$. The first term is counting rectangles and bigons in $\HD_1$, and the pieces of the second term where $a_j=\iota_{x_2}$ are counting rectangles and bigons in $\HD_2$. The remaining terms are counting pairs of half-strips, one in each diagram, with segments on the boundary that each correspond to $a_j$. One should think of these half-strips as gluing together in the closed diagram to form a rectangle that crosses the joined boundaries. We refer the reader to \cite{bfh2} for the more general case of diagrams that aren't nice and boundaries other than the torus.

Finally, one can conclude that because these are isomorphic chain complexes, we have that \[\hfhat(Y)\cong H_\ast (\cfahat(Y_1) \boxtimes \cfdhat(Y_2)) .\]

\subsubsection{Gradings}\label{sssec:gr}
 In \cite{bfh2}, the algebra is graded by a non-abelian group $G$, and so are domains on a bordered Heegaard diagram, whereas a right (resp.~left) structure associated to a bordered Heegaard diagram is graded by left (resp.~right) cosets in $G$ of the subgroup of gradings of periodic domains. The tensor product is graded by double cosets in $G$, from which one can recover the usual differential grading on Heegaard Floer homology. 

We will not need the above gradings in this paper, but only a simpler grading by $\Z/2$. Gradings by $\Z/2$ on bordered Floer homology have appeared in \cite{dec, z2gr, hlw}. We lay out definitions here directly, specializing to the case of torus boundary, and encourage the reader to compare with earlier literature.

The discussion in \cite[Section 3]{dec} shows that a $\Z/2$-grading on the torus algebra $\alg$ that is a reduction of the $G$-grading and plays well with gradings on modules arising from Heegaard diagrams must be $1$ for both $\rho_{12}$ and $\rho_{23}$.
By inspection, one sees that there are two  possible such gradings on $\alg$, which we will denote by $\gr_1$ and $\gr_2$; see Table~\ref{tab:az-gr}. 

\begin{table}[h]
\begin{center}
  \begin{tabular}{ |    c    ||    l      |     l     |   l   | l | l | l |}
    \hline
     & $\rho_1$ & $\rho_2$ & $\rho_3$ & $\rho_{12}$ & $\rho_{23}$ & $\rho_{123}$ \\ 
    \hline
    \hline
    $\gr_1$ & $0$ & $1$  & $0$ & $1$ & $1$ & $1$ \\ \hline
    $\gr_2$ & $1$ & $0$  & $1$ & $1$ & $1$ & $0$ \\ \hline
  \end{tabular}
  \vspace{.2cm}
\end{center}
  \caption{The two $\Z/2$-grading functions $\gr_1$ and $\gr_2$ on $\alg$.}
  \label{tab:az-gr}
\end{table}

Recall that the algebra $\alg(\ZZ)$ (thought of as a bimodule over itself) can be seen as the left-right type $AA$ bimodule associated to a certain $\beta$-$\alpha$-bordered Heegaard diagram $\mathsf{AZ}(\ZZ)$; see for example \cite[Section 4]{hfmor}. Figure~\ref{fig:az-torus-oriented} shows this diagram when $\ZZ$ is the pointed matched circle for a torus. Each generator of  $\alg = \alg(\ZZ,0)$ corresponds to an intersection point of an $\alpha$-curve and a $\beta$-curve, as indicated on Figure~\ref{fig:az-torus-oriented}. For example, the idempotent $\iota_1$ ($\iota_2$ resp.) corresponds to the intersection point in $\alpha_1\cap \beta_1$ (resp.~$\alpha_2\cap \beta_2$) that lies on the diagonal ``edge" of the diagram where 1-handles are attached. The generator $\rho_1$ corresponds to the lowest intersection point in $\alpha_2\cap \beta_1$, as seen on the diagram, and so on. 
Right and left multiplications correspond to counting $\alpha$-bordered and $\beta$-bordered rectangles, respectively; we omit the complete description here.

 Order the $\alpha$-arcs in order of appearance along $\bdy(\mathsf{AZ}(\ZZ))\setminus z$, and order the $\beta$-curves in the same way.  
Choose an orientation on the $\alpha$-arcs, and orient the $\beta$-arcs identically, so that the intersection signs at points corresponding to idempotents are all positive. The sign function 
\[s(\x) = \textrm{sign}(\sigma_{\x})\prod_{i=1}^g s(x_i)\]
on generators of $\mathsf{AZ}(\ZZ)$ induces a $\Z/2$ differential grading $m$ on the algebra $\alg(\ZZ)$ given by  $s = (-1)^m$. Here, $\sigma_{\x}$ is the permutation for $\x$ coming from the induced order of the occupied arcs. When $\ZZ$ is the pointed matched circle for the torus, there are four possible orientations on the curves. When both $\alpha$-arcs are oriented pointing into (or both pointing out of) the Heegaard surface as they first appear along $\bdy(\mathsf{AZ}(\ZZ))\setminus z$, it is easy to see that $m = \gr_1$. In the other two cases, we have $m = \gr_2$. 

\begin{figure}[h]
\centering
  \labellist
 	\pinlabel $\iota_1$ at 115 20
	\pinlabel $\rho_1$ at 115 52
	\pinlabel $\rho_{12}$ at 115 82
	\pinlabel $\rho_{123}$ at 115 112
	\pinlabel $\iota_2$ at 81 51
	\pinlabel $\rho_{2}$ at 81 82
	\pinlabel $\rho_{23}$ at 81 112
	\pinlabel $\rho_{3}$ at 52 112
  	\pinlabel \textcolor{red}{$\alpha_1$} at 180 60
	\pinlabel \textcolor{blue}{$\beta_1$} at 90 187
	\pinlabel \textcolor{red}{$\alpha_2$} at 180 90
	\pinlabel \textcolor{blue}{$\beta_2$} at 60 187
		\endlabellist
 \includegraphics[scale=.7]{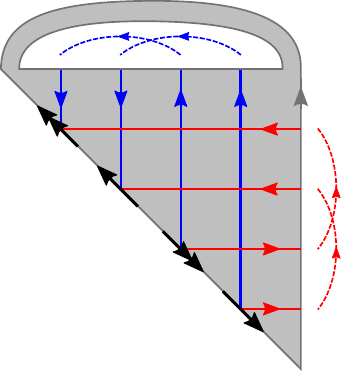} 
       \caption{The diagram $\mathsf{AZ}(\ZZ)$ when $\ZZ$ is the pointed matched circle for a torus. Arcs are further oriented, yielding a sign-induced grading $m$ that agrees with $\gr_1$. }\label{fig:az-torus-oriented}
\end{figure}
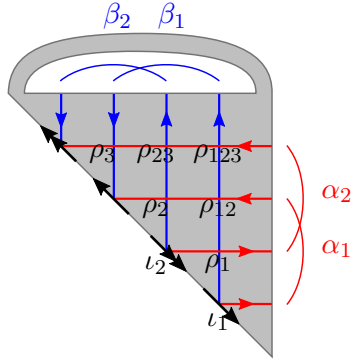

Given a bordered Heegaard diagram $\HD$ of genus $g$ for a manifold with torus boundary, order the $\alpha$-arcs so that $\alpha_1^a$ is the first arc we encounter starting at $z$ and following the orientation on $\bdy \HD$, and $\alpha_2^a$ is the second arc. Also order and orient the $\alpha$ and $\beta$ circles, and define a complete ordering on all $\alpha$-curves by $\alpha^c_1, \ldots, \alpha^c_{g-1}, \alpha_1^a,  \alpha_2^a$. Write each generator as an ordered tuple $\x = (x_1, \ldots, x_g)$ so that 
\begin{align*}
x_1 &\in \alpha^c_{1}\cap \beta_{\sigma_{\x}(1)}\\
& \hspace{7pt} \vdots\\
x_{g-1} &\in \alpha^c_{g-1}\cap \beta_{\sigma_{\x}(g-1)}\\
x_{g} &\in \alpha_{o(\x)}\cap \beta_{\sigma_{\x}(g)}
\end{align*}
for some permutation $\sigma_{\x}$.
For each $x_i$, define $s(x_i)$ to be the intersection sign of $\alphas$ and $\betas$ at $x_i$. Last, define the sign of $\x$ by 
\[s(\x) = \textrm{sign}(\sigma_{\x})\prod_{i=1}^g s(x_i).\]
We define a $\Z/2$-valued function $m$ on generators of $\HD$ by the rule $s = (-1)^m$. 

\begin{prop}
\label{prop:agrading}
If we choose the  grading on $\alg(\bdy\HD)$ corresponding to the prescribed orientations of the $\alpha$-arcs near $\bdy\HD$, the function $m$ is a $\Z/2$ differential grading on $\cfahat(\HD)$. 
\end{prop}

Now, given a bordered Heegaard diagram $\HD$ of genus $g$ for a manifold with torus boundary, order the $\alpha$-arcs so that $\alpha_2^a$ is the first arc we encounter starting at $z$ and following the orientation on $\bdy \HD$, and $\alpha_1^a$ is the second arc. Also order and orient the $\alpha$ and $\beta$ circles, and define a complete ordering on all $\alpha$-curves by $\alpha_1^a, \alpha_2^a, \alpha^c_1, \ldots, \alpha^c_{g-1}$. Write each generator as an ordered tuple $\x = (x_1, \ldots, x_g)$ so that 
\begin{align*}
x_1 &\in \alpha_{o(\x)}\cap \beta_{\sigma_{\x}(1)}\\
x_2 &\in \alpha^c_{1}\cap \beta_{\sigma_{\x}(2)}\\
& \hspace{7pt} \vdots\\
x_{g} &\in \alpha^c_{g-1}\cap \beta_{\sigma_{\x}(g)}
\end{align*}
for some permutation $\sigma_{\x}$.
For each $x_i$, define $s(x_i)$ to be the intersection sign of $\alphas$ and $\betas$ at $x_i$. Last, define the sign of $\x$ by 
\[s(\x) = (-1)^{o(\x)} \textrm{sign}(\sigma_{\x})\prod_{i=1}^g s(x_i).\]
We define a $\Z/2$-valued function $m$ on generators of $\HD$ by the rule $s = (-1)^m$. 

\begin{prop}
\label{prop:dgrading}
If we choose the  grading on $\alg(-\bdy\HD)$ corresponding to the prescribed orientations of the $\alpha$-arcs near $-\bdy\HD$, the function $m$ is a $\Z/2$ differential grading on $\cfdhat(\HD)$. 
\end{prop}

For brevity, we sometimes write $|\x|$ for $m(\x)$.

\subsection{Combinatorial sign assignments for closed nice diagrams}

In this section, we briefly review the relevant material from \cite{hfz}.

\begin{defn}
A \emph{formal generator} $\x$ is a one-to-one correspondence $\rho$ between two $n$-element sets $\alphas$ and $\betas$ (which can be represented as a subset of the Cartesian product $\alphas \times \betas$), together with a function $\zeta$ from $\rho$ to $\{\pm1\}$. The set of all such generators will be denoted as $\mathcal{G}_n$.
\end{defn}

After fixing an ordering of the elements of $\alphas$ and $\betas$, $\alphas=\{\alpha_1,\dots,\alpha_n\}$ and  $\betas=\{\beta_1,\dots,\beta_n\}$, we can think of the correspondence $\rho$ as the permutation $\sigma \in S_n$ for which $\rho = \{(\alpha_i, \beta_{\sigma(i)}) \mid i=1, \ldots, n \}\subset \alphas\times \betas$. We can then think of $\zeta$ as the $n$-tuple $\e = (\e_1, \dots, \e_n) \in \{\pm 1\}^n$ given by  $\e_i = \zeta(\alpha_i, \beta_{\sigma(i)})$. We call $\sigma$ and $\e$ the \emph{associated permutation} and \emph{sign profile} for the formal generator, and write $\x = (\sigma, \e)$. 

Graphically, a formal generator can be thought of as $n$ disjoint crosses between oriented red and blue arcs, with the red arcs decorated with $\alpha_1, \dots, \alpha_n$ and the blue arcs with $\beta_{\sigma(1)},\dots,\beta_{\sigma(n)}$, so that $\alpha_i$ and $\beta_{\sigma(i)}$ intersect. The sign profile $\e$ is given by the sign of each cross (with the convention that the $\alpha$-curve comes first). See Figure~\ref{fig:gen} for example.

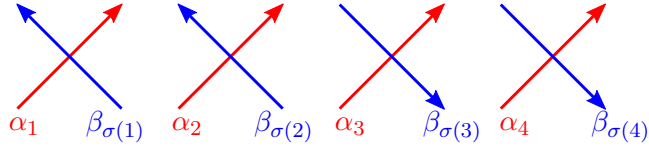
\begin{figure}[h]
\centering
  \labellist
  	\pinlabel \textcolor{red}{$\alpha_1$} at 73 745
	\pinlabel \textcolor{blue}{$\beta_{\sigma(1)}$} at 123 745
	\pinlabel \textcolor{red}{$\alpha_2$} at 162 745
	\pinlabel \textcolor{blue}{$\beta_{\sigma(2)}$} at 215 745
	\pinlabel \textcolor{red}{$\alpha_3$} at 251 745
	\pinlabel \textcolor{blue}{$\beta_{\sigma(3)}$} at 307 745
	\pinlabel \textcolor{red}{$\alpha_4$} at 341 745
	\pinlabel \textcolor{blue}{$\beta_{\sigma(4)}$} at 399 745
		\endlabellist
 \includegraphics[scale=.7]{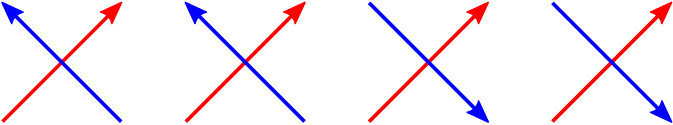} 
      \vskip .2 cm
       \caption{A generator with associated permutation $\sigma$ and sign profile $(+1,+1,-1,-1)$.}\label{fig:gen}
\end{figure}

\begin{defn}
\label{def:oss-bigon}
Fix a positive integer $n$ and sets $\alphas = \{\alpha_1,\dots,\alpha_n\}$ and $\betas = \{\beta_1,\dots,\beta_n\}$. Consider $n-1$ disjoint crosses of oriented arcs, along with another pair of oriented arcs, intersecting each other twice and disjoint from all the crosses. The complement of the last two arcs has two components (one compact and one non-compact); the $n-1$ disjoint crosses and the ends of the two arcs that intersect twice are all required to be in the non-compact component. Decorate one of the arcs in each pair of intersecting arcs with an $\alpha_i$ and the other one with a $\beta_j$ so that each element of $\alphas$ and of $\betas$ is used exactly once. Two such configurations are considered to be equivalent if there is an orientation preserving diffeomorphism of the plane  that maps one configuration into the other and respects the orientations and the decorations of the arcs. An equivalence class of such configurations is called a \emph{formal bigon}. See Figure~\ref{fig:bi} for example. 
\end{defn}

\begin{figure}[h]
\centering
\labellist
  \pinlabel $\textcolor{red}{\alpha_1}$ at 3 -2
	\pinlabel $\textcolor{blue}{\beta_2}$ at 60 -2
	\pinlabel $\textcolor{red}{\alpha_2}$ at 96 -10
	\pinlabel $\textcolor{blue}{\beta_3}$ at 142 -10
	\pinlabel $\textcolor{red}{\alpha_3}$ at 179 -2
	\pinlabel $\textcolor{blue}{\beta_1}$ at 239 -2
	\pinlabel $\textcolor{black}{xy}$ at 29 22
	\pinlabel $\textcolor{black}{x}$ at 118 -1
	\pinlabel $\textcolor{black}{y}$ at 118 83
	\pinlabel $\textcolor{black}{xy}$ at 206 22
	\endlabellist
  \includegraphics[scale=.7]{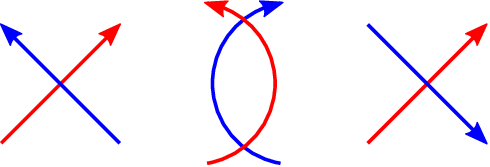} 
      \vskip .2 cm
       \caption{An example of a formal bigon. Visually, we can see this bigon goes from the generator $\x=\{(1 \ 2 \ 3), (+1, +1, -1)\}$ to the generator $\y=\{(1 \ 2 \ 3), (+1, -1, -1)\}$. See the paragraph after Definition~\ref{def:oss-flow} for a more precise definition of \emph{from} and \emph{to}.} \label{fig:bi}
\end{figure}
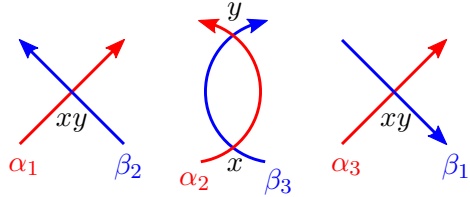

\begin{defn}
\label{def:oss-rect}
For a fixed positive integer $n$ and sets $\alphas = \{\alpha_1,\dots,\alpha_n\}$ and $\betas = \{\beta_1,\dots,\beta_n\}$, consider $n-2$ disjoint oriented crosses. Consider furthermore two pairs of oriented closed arcs ($a_1, b_1$) and ($a_2, b_2$) such that $a_1$ and $a_2$ (and likewise $b_1$ and $b_2$) are disjoint, while both $a_i$ intersect both $b_j$ exactly once in their interiors. One of the two components of the complement of the last four arcs is compact, and we require its interior to be disjoint from all arcs and the endpoints of the four arcs to lie in the non-compact region. Decorate one of the arcs in each pair with an $\alpha_i$ and the other one with a $\beta_j$ so that each element of $\alphas$ and of $\betas$ is used exactly once, and so that the $a_i$ arcs in the rectangle are decorated by elements of $\alphas$ while the $b_j$ arcs are decorated with elements of $\betas$. Two such configurations are considered to be equivalent if there is an orientation preserving diffeomorphism of the plane mapping one into the other, while respecting both the orientations and the decorations of the arcs. An equivalence class of such objects is called a \emph{formal rectangle}. See Figure~\ref{fig:rect} for example. 
\end{defn}

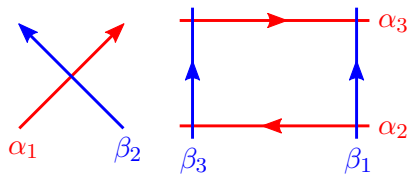
\begin{figure}[h]
\centering
\labellist
  \pinlabel $\textcolor{red}{\alpha_1}$ at 3 -6
	\pinlabel $\textcolor{blue}{\beta_2}$ at 60 -6
	\pinlabel $\textcolor{red}{\alpha_2}$ at 205 5
	\pinlabel $\textcolor{blue}{\beta_3}$ at 96 -12
	\pinlabel $\textcolor{red}{\alpha_3}$ at 205 63
	\pinlabel $\textcolor{blue}{\beta_1}$ at 186 -12
	\endlabellist
  \includegraphics[scale=.7]{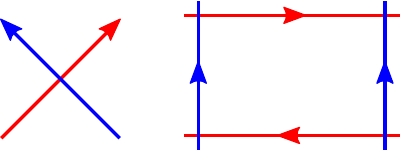} 
      \vskip .2 cm
       \caption{An example of a formal rectangle.}\label{fig:rect}
\end{figure}

\begin{defn}
\label{def:oss-flow}
A \emph{formal flow} is either a formal bigon or formal rectangle. For a given integer $n$, the set of all flows moving between elements of $\mathcal{G}_n$ will be denoted $\mathcal{F}_n$.
\end{defn}

The compact region described above for a given formal flow is called the \emph{domain} for that flow. If we wish to consider the region with its full set of data of both labels and orientations, we refer to it as a \emph{decorated domain}. A formal flow $\phi$ determines two formal generators $\x$ and $\y$, obtained by adding  to the disjoint crosses a small neighborhood of the point(s) where the oriented boundary of the domain switches from a $\beta$-arc to an $\alpha$-arc, or from an $\alpha$-arc to a $\beta$-arc, respectively. We say that $\phi$ is a flow \emph{from $\x$ to $\y$}, and write $\phi:\x\to \y$. We refer to $\x$ and $\y$ as the  \emph{starting generator}  and  \emph{ending generator}, respectively. We call the $\alpha$-indices at which $\x$ and $\y$ differ the \emph{moving coordinates}. A bigon has one moving coordinate, and the associated generators have identical permutations and sign profiles that differ exactly at that coordinate. A rectangle has two moving coordinates, the generators' permutations differ by the transposition corresponding to the moving coordinates, and their sign profiles may differ at either both or neither of the moving coordinates and agree elsewhere.

Suppose that the bigons $\phi_1$ and $\phi_2$ agree at all non-moving coordinates. Furthermore, suppose that at the moving coordinates the $\beta$-arcs both agree or both disagree with the orientation induced by the boundary of the domain, whereas one $\alpha$-arc agrees and the other disagrees with the orientation of the boundary of the domain. Graphically, this means the two domains can be glued along their respective $\alpha$-edges to obtain a disk, so that the labels and orientations of the $\alpha$-arcs match along the gluing, and the $\beta$-arcs result in a consistently oriented and labeled boundary for the disk. We say that $\phi_1$ and $\phi_2$ are a disk-like \emph{$\beta$-type boundary degeneration}. See the left diagram in Figure~\ref{fig:beta-degens}.

Similarly, suppose that the rectangles $\phi_1$ and $\phi_2$ agree at all non-moving coordinates, and their domains can be glued along their $\alpha$-arcs to obtain an annulus, so that labels and orientations match along the gluing, and also result in consistently labeled and oriented boundary components for the annulus. 
We say that $\phi_1$ and $\phi_2$ are an annular \emph{$\beta$-type boundary degeneration}. See the right diagram in Figure~\ref{fig:beta-degens}.

\begin{figure}[h]
\centering
 \includegraphics[scale=.7]{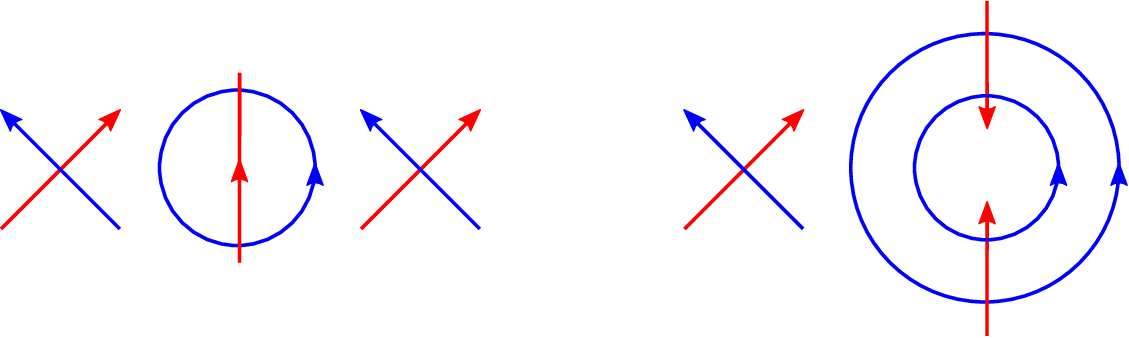} 
       \caption{Left: A disk-like $\beta$-degeneration in $\mathcal{F}_3$. Right: An annular $\beta$-degeneration in $\mathcal{F}_3$. Labels on the curves are omitted. }\label{fig:beta-degens}
\end{figure}

If in the above description the roles of $\alpha$-curves and $\beta$-curves  are interchanged, we say that $\phi_1$ and $\phi_2$ are an \emph{$\alpha$-type boundary degeneration}. 

Observe that in a boundary degeneration, the starting generator of each flow is the ending generator of the other flow.

Now consider a pair of flows $\phi_1:\x\to \y$, $\phi_2:\y\to \z$, where $\x\neq \z$. Suppose $\phi_1$ and $\phi_2$ have no moving coordinates in common. This pair uniquely determines another pair of flows $\phi_3:\x\to \y'$, $\phi_4:\y'\to \z$ by requiring that $\phi_3$ has the same starting generator as $\phi_1$ and the same domain (considered with labelings and orientations on the boundary arcs) as $\phi_2$, and $\phi_4$ has the same ending generator as $\phi_2$ and the same domain as $\phi_1$. We say that the pairs $(\phi_1,\phi_2)$ and $(\phi_3,\phi_4)$  form a \emph{square}. Graphically, this description means that we can embed both domains disjointly in the plane, and complete with crosses in the non-compact region so that all four generators and flows can be seen. 

Next suppose $\phi_1$ and $\phi_2$ share a moving coordinate. We can embed both flows simultaneously on the plane so that the domains have disjoint interiors, there are exactly $n$ red arcs and $n$ blue arcs, and the arcs' endpoints are all disjoint from the two domains. In this graphical representation, there is a unique way to decompose the union of the domains to see two other flows $\phi_3:\x\to \y'$, $\phi_4:\y'\to \z$, so that the union of their domains is the same as that for $\phi_1$ and $\phi_2$. Suppose further the union of the domains looks like one of the diagrams in Figure~\ref{fig:z-int-sq1}. In this case again we say $(\phi_1,\phi_2)$ and $(\phi_3,\phi_4)$  form a \emph{square}.

\begin{figure}[h]
\centering
 \includegraphics[scale=.7]{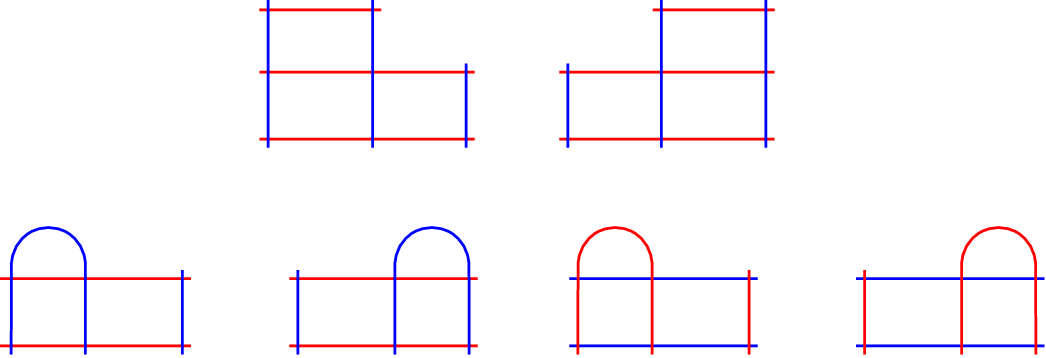} 
       \caption{Squares. Unlike in Figure~\ref{fig:beta-degens}, orientations, as well as crosses at non-moving coordinates are omitted from the diagram. The pairs of flows in the first row examples share one moving coordinate, and there are $n-3$ crosses omitted from each example. The pairs of flows in the second row examples share one  moving coordinate, and there are $n-2$ crosses omitted from each example.}\label{fig:z-int-sq1}
\end{figure}

We are now ready to recall the definition of a sign assignment.

\begin{defn}\label{def:s}
Fix an integer $n$. A \emph{sign assignment} of power $n$ is a map $S$ from the set of formal flows $\mathcal{F}_n$ to $\{\pm 1\}$ with the following properties:
\begin{itemize}
\item[\mylabel{itm:s1}{(S-1)}] \quad if $(\phi_1,\phi_2)$ is an $\alpha$-type boundary degeneration, then
\[S(\phi_1) \cdot S(\phi_2) = 1;\]

\item[\mylabel{itm:s2}{(S-2)}] \quad if $(\phi_1,\phi_2)$ is an $\beta$-type boundary degeneration, then
\[S(\phi_1) \cdot S(\phi_2) = -1;\]

\item[\mylabel{itm:s3}{(S-3)}] \quad if the two pairs $(\phi_1,\phi_2)$ and $(\phi_3,\phi_4)$ form a square, then
\[S(\phi_1) \cdot S(\phi_2) = -S(\phi_3) \cdot S(\phi_4).\]

\end{itemize}
\end{defn}

Note that (S-3) can be restated as $\Pi_{i=1}^4 S(\phi_i)= -1$.

In \cite{hfz}, the authors discuss some additional types of squares $\{(\phi_1,\phi_2),(\phi_3,\phi_4)\}$, namely, the ones given by Figure~\ref{fig:z-int-sq2}. 
\begin{figure}[h]
\centering
 \includegraphics[scale=.7]{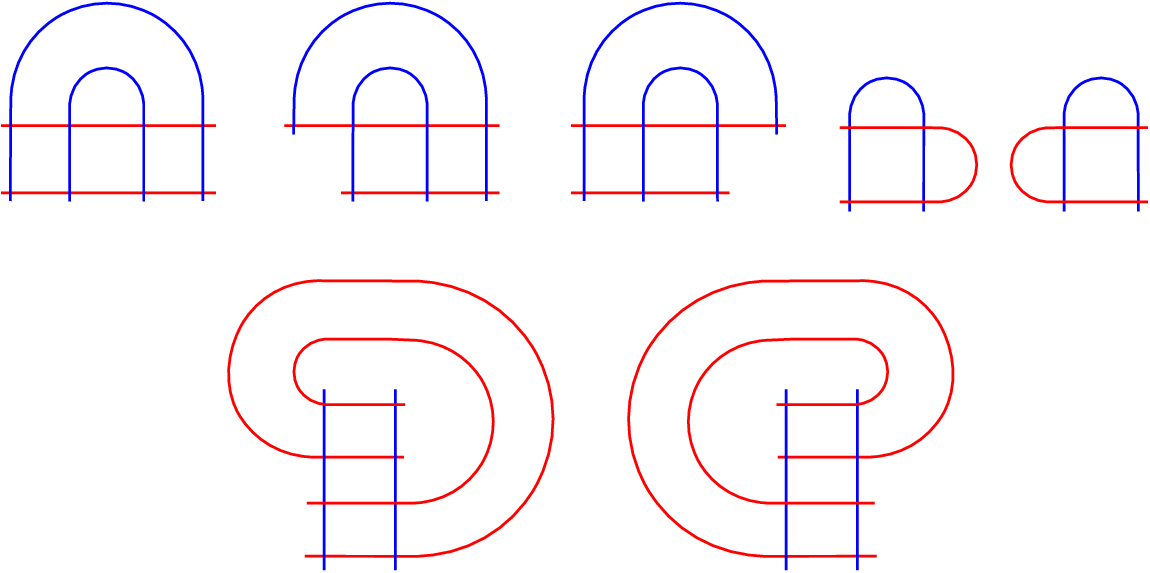} 
       \caption{Further types of squares.}\label{fig:z-int-sq2}
\end{figure}
In \cite[Lemma 2.8]{hfz}, they show that if $S$ is a sign assignment and $\{(\phi_1,\phi_2),(\phi_3,\phi_4)\}$ is one of these further types of squares, then $\Pi_{i=1}^4 S(\phi_i)= -1$.

In upcoming arguments, we will occasionally need to compare two rectangles with the same domain but different non-moving coordinates. Two such rectangles are related by a sequence of operations called ``simple flips", defined below. 

\begin{defn}\label{def:sflip}
Let $\phi: \x \to \y$ be a formal rectangle and let $i$ be a non-moving coordinate of $\phi$. Consider the new flow $\phi':\x' \to \y'$ that has the same domain as $\phi$, and such that $\x'$ and $\y'$ are obtained from $\x$ and $\y$ by flipping the sign profile at the $i^{\mathrm{th}}$ coordinate. We say that $\phi$ and $\phi'$ are related by a \emph{simple flip} and note that there exists some pair of bigons $b, b'$ such that the pairs $(b,\phi)$ and $(\phi',b')$ form a square.
\end{defn}


\section{The torus algebra $\algz$.}
\label{sec:alg}

When $\ZZ$ is the pointed matched circle for the torus, 
we extend the definition of the summand $\alg = \alg(\ZZ,0)$ of $\alg(\ZZ)$ from Section~\ref{sssec:algz2} to an algebra $\algz(\ZZ,0)$.
\begin{defn}
The (zero-summand)   \emph{$\Z$-algebra for the pointed matched circle for a torus} $\algz$ is defined as $\algz \coloneqq \Z\otimes_{\F_2}\alg(\ZZ,0)$.
\end{defn}
In other words, we again have eight generators, but over $\Z$, with the same products as for $\alg(\ZZ,0)$. 

The two functions from Table~\ref{tab:az-gr} defined on generators extend to $\Z/2$ gradings on $\algz$. Thus, there are two distinct graded algebras over $\Z$ that we will consider in this paper. To fit with the diagrammatic nature of sign assignments, we endow a pointed matched circle with additional data, to which a $\Z/2$-graded $\Z$-algebra can naturally be associated. This additional data views each pair of matched points on a pointed matched circle as the boundary of an oriented 1-manifold:

\begin{defn}\label{def:spmc}
A \emph{signed pointed matched circle} is a pointed matched circle $\ZZ = (Z, {\bf a}, M, z)$ along with the following data. For each pair of points $M^{-1}(i)$, one point is labelled  with $``-"$ and the other is labelled with ``$+$". 
\end{defn}
Given a pointed matched circle $\ZZ$, order the $2k$ pairs of points in order of appearance along $\ZZ\setminus z$. Let $P\in \{-,+\}^{2k}$ be a sign sequence. Define $\ZZ_P$ to be the signed pointed matched circle for which the first point (as encountered along $\ZZ\setminus z$) of the $i^{\mathrm{th}}$ matched pair is labelled with a $+$ if and only if   the $i^{\mathrm{th}}$ coordinate of $P$ is a $+$. For ease of notation, we will write $\ZZ_{+-}$ for $\ZZ_{(+,-)}$, and so on. 

A signed pointed matched circle $\ZZ_P$ induces an orientation on the $\alpha$- and $\beta$-arcs on $\mathsf{AZ}(\ZZ)$, as follows. Identify each boundary component of  $\mathsf{AZ}(\ZZ)$  with $\ZZ_P$, and orient each arc so it starts at a $-$ and ends at a $+$. We write $\mathsf{AZ}(\ZZ_P)$ for the resulting Heegaard diagram along with the prescribed orientations. 

When $|{\bf a}|=4$, i.e. in the torus case, there are four distinct signed pointed matched circles -- $\ZZ_{++}$, $\ZZ_{-+}$,  $\ZZ_{+-}$, and $\ZZ_{--}$; see Figure~\ref{fig:s-pmc}. Observe that the oriented diagram in Figure~\ref{fig:az-torus-oriented} represents  $\mathsf{AZ}(\ZZ_{++})$.

\begin{figure}[h]
\centering
\labellist
  	\pinlabel $\textcolor{red}{+}$ at -9 24
	\pinlabel $\textcolor{red}{+}$ at -9 59
	\pinlabel $\textcolor{red}{-}$ at -9 97
	\pinlabel $\textcolor{red}{-}$ at -9 132 
	\pinlabel $\textcolor{red}{1}$ at 50 55
	\pinlabel $\textcolor{red}{2}$ at 50 100
	\pinlabel $\textcolor{red}{-}$ at 100 24
	\pinlabel $\textcolor{red}{+}$ at 100 59
	\pinlabel $\textcolor{red}{+}$ at 100 97
	\pinlabel $\textcolor{red}{-}$ at 100 132
	\pinlabel $\textcolor{red}{1}$ at 160 55
	\pinlabel $\textcolor{red}{2}$ at 160 100
	\pinlabel $\textcolor{red}{+}$ at 215 24
	\pinlabel $\textcolor{red}{-}$ at 215 59
	\pinlabel $\textcolor{red}{-}$ at 215 97
	\pinlabel $\textcolor{red}{+}$ at 215 132 
	\pinlabel $\textcolor{red}{1}$ at 272 55
	\pinlabel $\textcolor{red}{2}$ at 272 100
	\pinlabel $\textcolor{red}{-}$ at 320 24
	\pinlabel $\textcolor{red}{-}$ at 320 59
	\pinlabel $\textcolor{red}{+}$ at 320 97
	\pinlabel $\textcolor{red}{+}$ at 320 132 
	\pinlabel $\textcolor{red}{1}$ at 385 55
	\pinlabel $\textcolor{red}{2}$ at 385 100
\endlabellist
 \includegraphics[scale=.6]{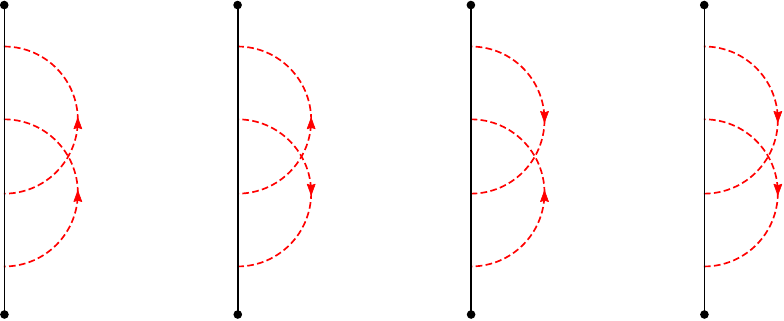} 
       \caption{The four signed pointed matched circles   $\ZZ_{++}$, $\ZZ_{-+}$,  $\ZZ_{+-}$, and $\ZZ_{--}$ (from left to right)  for the torus. Here, the top and bottom point on each arc are identified, and represent the basepoint $z$. The circle $Z$ is oriented from bottom to top. The dashed curves represent the matching data, rather than the $\alpha$-arcs from $\mathsf{AZ}(\ZZ_P)$; as in Figure~\ref{fig:az-torus-oriented}, they are oriented in a manner that would complete the $\alpha$-arcs to oriented closed circles.}\label{fig:s-pmc}
\end{figure}

\begin{defn}\label{def:alg-spmc}
Let   $\ZZ_P$ be a signed pointed matched circle for the torus. The \emph{$\Z/2$-graded $\Z$-algebra $\alg_P$} is the algebra $\algz$ with the $\Z/2$ grading associated to $\mathsf{AZ}(\ZZ_P)$ from Section~\ref{sssec:gr}.
\end{defn}

We can see the generators of $\alg_P$ on $\ZZ_P$ in the same way as in the $\F_2$ case in Section~\ref{sssec:algz2}. The signed pointed matched circle $\ZZ_P$ and the (unsigned) pointed matched circle $\ZZ$ have the same underlying data $(Z, {\bf a}, M, z)$. An oriented subarc of $Z\setminus z$ with endpoints in ${\bf a}$ corresponds to an element $\rho_I$ in $\alg_P$ if and only if it corresponds to $\rho_I$ in $\alg$. The \emph{ring of idempotents associated to $\alg_P$} is the algebra  $\mathcal I_P$ generated by the two idempotents of $\alg_P$. Note that $\mathcal I_P\cong \Z^2$.

If we consider subarcs with the signs on the endpoints inherited from the signed pointed matched circle, one can determine the gradings on the algebra elements in the following way.

\begin{lem}
The product of the inherited signs on the endpoints for a subarc $\rho_I$ is equal to $(-1)^{\mid\rho_I\mid}$.
\end{lem}

\begin{proof}
Consider  $\mathsf{AZ}(\ZZ_P)$. Recall that idempotents always have grading zero, and the grading for $\rho_I$ can be determined by comparing the intersection sign at the vertex labelled $\rho_I$ with the intersection sign at the idempotent vertex vertically below it: if they agree, $\rho_I$ has grading zero; if they differ, it has grading one. By inspection we see that the above signs differ precisely when the two horizontal segments that   connect the respective intersection points to the right boundary edge of the diagram have opposing orientations as seen on the diagram. However, we can identify this right boundary component of $\mathsf{AZ}(\ZZ_P)$ that intersects the $\alpha$ curves with $\ZZ_P$, and correspondingly view $\rho_I$ as an oriented segment on $\ZZ_P$ that connects endpoints of the two horizontal segments discussed above. This in turn tells us that the orientations of these horizontal segments are the same as the orientations (locally speaking) of the $\alpha$-arcs where they intersect the endpoints of $\rho_I$ in $\ZZ_P$. Finally, if the orientations of these two $\alpha$-arcs differ, then the inherited signs of the endpoints of $\rho_I$ must differ (since those signs are precisely how the orientation is determined in the first place). 
\end{proof}


\section{Bordered Sign Assignments}\label{sec:bord-sa}

We now discuss how to extend the construction from \cite{hfz} to bordered Floer homology in the case of torus boundary.

For Definitions~\ref{def:bf}--\ref{def:sd}, fix a positive integer $n$ and a sign sequence  $P\in \{-,+\}^{2}$. 

\begin{defn}
A \emph{formal right (resp.\ left) bordered generator} is a one-to-one correspondence $\rho$ between two ordered $n$-element sets $\alphas$ and $\betas$ (which can be represented as a subset of the Cartesian product $\alphas \times \betas$), together with a function $\e$ from $\rho$ to $\{\pm1\}$ and a choice $\s\in\{1,2\}$. The choice $\s$ can be thought of as a special label associated to the last (resp.~first) element in $\alphas$, as described below. The set of all such generators will be denoted  $\GnA^R$ (resp.~$\GnA^L$). 
\end{defn}

 A bordered right (resp.\ left) generator can be more clearly thought of as $n$ disjoint crosses of oriented arcs. One arc of each cross is labelled by an element of the set $\{\alpha_1,\dots,\alpha_{n-1},\alpha_1^a,\alpha_2^a\}$ (resp.~$\{\alpha_1^a,\alpha_2^a,\alpha_2,\dots,\alpha_n\}$) so that only one of the two labels $\alpha_1^a$ and $\alpha_2^a$ is used, and each $\alpha_i$ label is used exactly once. The remaining arcs are labelled by elements of the set $\betas$ so that each element is used once. The choice of $\s$ is the index of the used label $\alpha_{\s}^a$. The correspondence $\rho$ can be thought of as the permutation $\sigma \in S_n$ represented by the labelling, thinking of the $\alpha_{\s}^a$ label as the $n^{\text{th}}$ (resp.\ first) coordinate of the permutation, and $\e=(\e_1,\dots,\e_n) \in \{\pm 1\}^n$ is determined by the sign of the crosses, with the convention that the $\alpha$-arcs come first. Once again we call $\sigma$ and $\e$ the \emph{associated permutation} and \emph{sign profile} for $\x$, and we call $\s$ the \emph{idempotent} for $\x$ (see the note below for an explanation). We write $\x = (\sigma, \e, \s)$. For clarity, we will frequently write $\sigma_\x$ for the permutation $\sigma$ associated to a generator $\x$. Similarly, we will often write $\e(\x)$ and $\s(\x)$.

 We denote the sign profile that is identically $1$ in each factor by ${\bf 1}$.

\begin{note}
Given a specific bordered Heegaard diagram $\HD$ and a sign sequence  $P\in \{-,+\}^{2}$, one can order and orient all of the $\alpha$-circles, $\beta$-circles, and $\alpha$-arcs so that if we identify $\bdy \HD$ with $\ZZ_P$, each $\alpha$-arc is oriented from a $+$ to a $-$. With this choice, each generator of the diagram has two associated formal bordered generators, a left one and a right one. For a left (resp.~right) generator, we think of the occupied $\alpha$-arc as the ``last" (resp.~``first") of all the occupied $\alpha$-curves. The generator is then specified by the associated permutation, the signs of intersections, and the index $\s$ of the occupied $\alpha^a$-arc (with the convention that $\alpha_1^a$ is the first arc we see as we follow $\bdy \HD$ along the  orientation induced as boundary of $\HD$, starting at the basepoint $z$). We remark that in the forthcoming definition of the type $A$ structure $\cfahat(\HD, \Z)$, the choice $\s$ will correspond to the idempotent of $\alg_P$ which will act on the given generator by the identity.
\end{note}

 \begin{defn}\label{def:bf}
A \emph{formal right (resp.\ left) internal flow $\phi$ between formal bordered generators $\x$ and $\y$  in $\GnA^R$ (resp.\ in $\GnA^L$)}, is defined analogously to a formal flow from \cite{hfz} (also see Definitions~\ref{def:oss-bigon}, \ref{def:oss-rect}, \ref{def:oss-flow})---geometrically it appears the same as in the closed case (a \emph{bigon} or a \emph{rectangle}), but the labels now come from $\{\alpha_1,\dots,\alpha_{n-1},\alpha_1^a,\alpha_2^a\}$ (resp.\ $\{\alpha_1^a,\alpha_2^a,\alpha_2,\dots,\alpha_n\}$) and $\betas$ instead of $\alphas$ and $\betas$. 

A \emph{formal right (resp.~left) bordered flow} $\phi$ between formal bordered generators in $\GnA^R$ (resp.\ $\GnA^L$) is graphically similar to a formal rectangle, but with $n-1$ disjoint crosses and a rectangle. For a right (resp.\ left) bordered flow, the edges of the rectangle carry the following data when traversed counterclockwise. One edge, hereafter referred to as the \emph{bottom edge}, has label $\alpha_i^a$; the next edge, called the \emph{right edge}, is unlabeled (resp.\ labeled $\beta_{\sigma(n)}$); the third edge, called the \emph{top edge}, is labeled $\alpha_j^a$, with $j$ not necessarily distinct from $i$; the last edge, called the \emph{left edge},  is labeled $\beta_{\sigma(n)}$ (resp.\ unlabeled).  Label the intersection point of each $\alpha$ arc with the unlabeled edge with a $+$ if the $\alpha$ arc is oriented towards (resp.\ away from) that point, and with a $-$ if the $\alpha$ arc is oriented away from  (resp.\ towards) that point, and orient the unlabeled edge from bottom to top. We further require that the unlabeled edge with the $\pm$ decorations on it can be identified with an oriented segment of $\ZZ_P\setminus z$ so that the $\pm$ labels match. See Figure~\ref{fig:rho-flows}, for example.  

The set of  flows between elements of $\GnA^R$ (resp.\ $\GnA^L$) is denoted $\FnA^R$ (resp.\ $\FnA^L$). We call the flows in each set \emph{right} (resp.\ \emph{left}) formal flows.
\end{defn}

\begin{note}
Observe that, by definition of formal bordered generator, in a formal internal flow only one of the two $\alpha_i^a$ labels may be used. Also observe that formal internal flows may carry $\alpha^a$-labels at the moving coordinates.  Last, note that a right (resp.~left) bordered  flow has only one moving coordinate, which is the $n^\text{th}$ (resp.~first) one. 
\end{note}

We can further divide right (resp.\ left) bordered flows into six classes, corresponding to the six non-trivial generators of $\alg_P$. We say that a bordered flow is of \emph{type $\rho_i$} if the decorated unlabeled edge, when oriented from bottom to top, agrees with the segment of $\ZZ_P$ representing $\rho_i$.  In figures, we will often label the ``unlabeled" edge with the corresponding $\rho_i$.

\vskip .2 cm

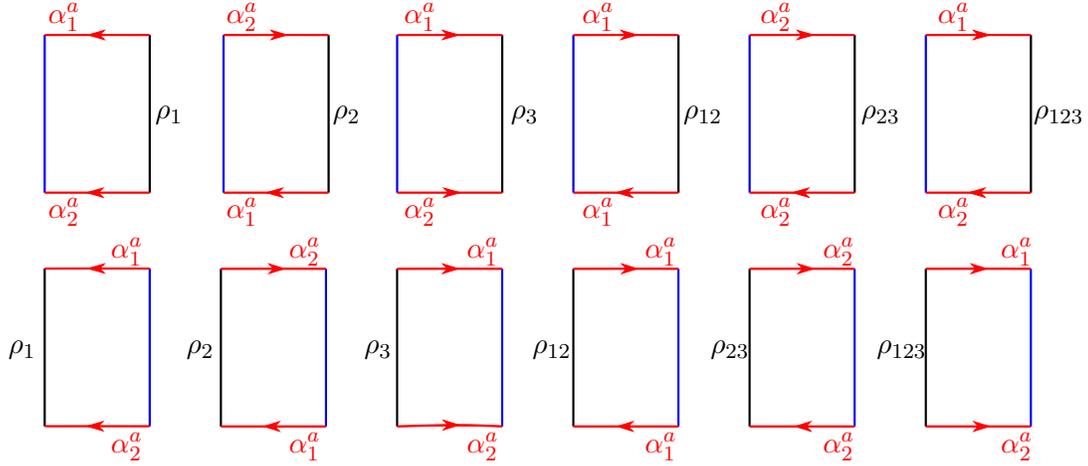
\begin{figure}[h]
\centering
\labellist
  	\pinlabel $\textcolor{red}{\alpha_1^a}$ at 65 855
	\pinlabel $\textcolor{red}{\alpha_2^a}$ at 65 762
	\pinlabel $\textcolor{red}{\alpha_2^a}$ at 150 855
	\pinlabel $\textcolor{red}{\alpha_1^a}$ at 150 762
	\pinlabel $\textcolor{red}{\alpha_1^a}$ at 235 855
	\pinlabel $\textcolor{red}{\alpha_2^a}$ at 235 762
	\pinlabel $\textcolor{red}{\alpha_1^a}$ at 320 855
	\pinlabel $\textcolor{red}{\alpha_1^a}$ at 320 762
	\pinlabel $\textcolor{red}{\alpha_2^a}$ at 405 855
	\pinlabel $\textcolor{red}{\alpha_2^a}$ at 405 762
	\pinlabel $\textcolor{red}{\alpha_1^a}$ at 490 855
	\pinlabel $\textcolor{red}{\alpha_2^a}$ at 490 762
	\pinlabel $\textcolor{black}{\rho_1}$ at 115 808
	\pinlabel $\textcolor{black}{\rho_2}$ at 200 808
	\pinlabel $\textcolor{black}{\rho_3}$ at 285 808
	\pinlabel $\textcolor{black}{\rho_{12}}$ at 370 808
	\pinlabel $\textcolor{black}{\rho_{23}}$ at 455 808
	\pinlabel $\textcolor{black}{\rho_{123}}$ at 540 808
	\pinlabel $\textcolor{red}{\alpha_1^a}$ at 95 743
	\pinlabel $\textcolor{red}{\alpha_2^a}$ at 95 650
	\pinlabel $\textcolor{red}{\alpha_2^a}$ at 180 743
	\pinlabel $\textcolor{red}{\alpha_1^a}$ at 180 650
	\pinlabel $\textcolor{red}{\alpha_1^a}$ at 265 743
	\pinlabel $\textcolor{red}{\alpha_2^a}$ at 265 650
	\pinlabel $\textcolor{red}{\alpha_1^a}$ at 350 743
	\pinlabel $\textcolor{red}{\alpha_1^a}$ at 350 650
	\pinlabel $\textcolor{red}{\alpha_2^a}$ at 435 743
	\pinlabel $\textcolor{red}{\alpha_2^a}$ at 435 650
	\pinlabel $\textcolor{red}{\alpha_1^a}$ at 520 743
	\pinlabel $\textcolor{red}{\alpha_2^a}$ at 520 650
	\pinlabel $\textcolor{black}{\rho_1}$ at 45 696
	\pinlabel $\textcolor{black}{\rho_2}$ at 130 696
	\pinlabel $\textcolor{black}{\rho_3}$ at 215 696
	\pinlabel $\textcolor{black}{\rho_{12}}$ at 298 696
	\pinlabel $\textcolor{black}{\rho_{23}}$ at 383 696
	\pinlabel $\textcolor{black}{\rho_{123}}$ at 465 696
		\endlabellist
  \includegraphics[scale=.795]{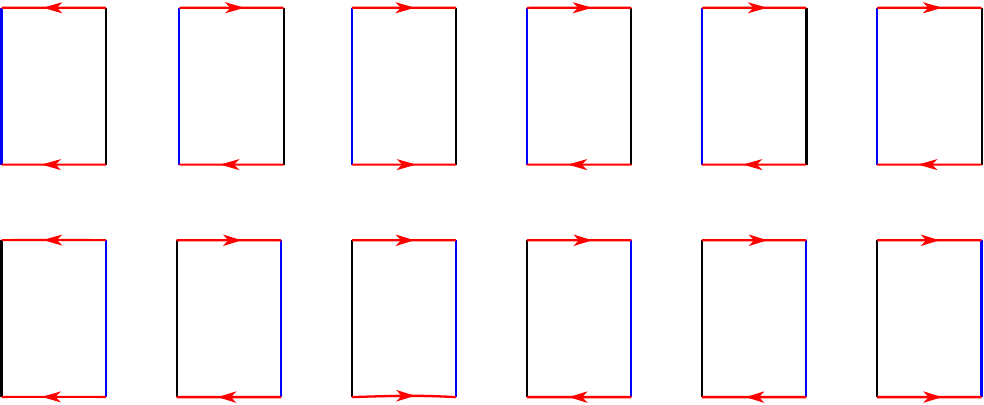} 
  \vskip .2 cm
       \caption{Top: The six classes of right bordered flows for the signed pointed matched circle $\ZZ_{--}$. Bottom: The six classes of left bordered flows for $\ZZ_{--}$. The non-moving coordinates, as well as the label and orientation of the (blue in this figure) $\beta$-arc at the moving coordinate, are omitted from the figure.}\label{fig:rho-flows}
\end{figure}

 Boundary degenerations and squares for flows in $\FnA^R$ or $\FnA^L$  are defined analogously to those for the closed case. The only difference is that for squares, some of the geometric configurations in Figure~\ref{fig:z-int-sq1} (but none of those in Figure~\ref{fig:z-int-sq2}) may now have unlabeled edges instead. See Figure~\ref{fig:z-bord-sq} for the additional cases.

\begin{defn}\label{def:bs}
Suppose that $(\phi_1,\phi_2)$ and $(\phi_3,\phi_4)$ are distinct pairs of flows, all in $\FnA^R$ or all in $\FnA^L$, that form a square. If all four flows are internal, we say that $(\phi_1,\phi_2)$ and $(\phi_3,\phi_4)$ form an \emph{internal square}. Otherwise, we say that $(\phi_1,\phi_2)$ and $(\phi_3,\phi_4)$ form a \emph{bordered square} (see Figure~\ref{fig:z-bord-sq}).
\end{defn}

Note that if  $(\phi_1,\phi_2)$ and $(\phi_3,\phi_4)$ form a bordered square, then exactly the first flow in one pair and the second flow in the other pair must be bordered.

\begin{figure}[h]
\centering
 \includegraphics[scale=.7]{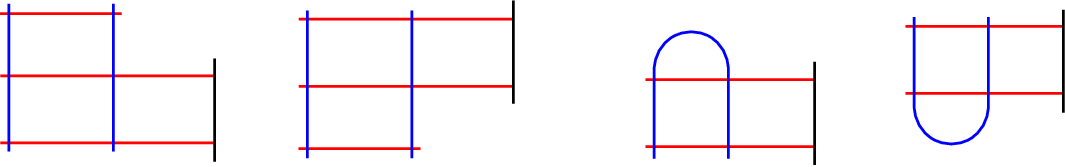} 
       \caption{Right bordered squares. Left bordered squares can be visualized by reflecting along a vertical line.}\label{fig:z-bord-sq}
\end{figure}

\begin{defn}\label{def:tri}
Suppose $\phi_1$, $\phi_2$, and $\phi_3$ are right bordered flows with the same label and orientation on the $\beta$-edge of the underlying rectangle, of  type $\rho_i, \rho_j$, and $\rho_k$, respectively. If $\phi_2(\phi_1(\x))=\phi_3(\x)$ and $\rho_i \rho_j = \rho_k\in \alg$, we say that the pair $(\phi_1,\phi_2)$ and the flow $\phi_3$ together form a \emph{bordered triangle}.
\end{defn}

\begin{defn}\label{def:sa}
A \emph{bordered sign assignment of  type $A$, compatible with $P$, and of power $n$} is a map $\STA$ from the set of formal flows $\FnA^R$ to $\{\pm 1\}$ with the following properties:
\begin{itemize}
\item[\mylabel{itm:s1a}{(A-1)}] \quad if $(\phi_1,\phi_2)$ is an $\alpha$-type boundary degeneration, then
\[\STA(\phi_1) \cdot \STA(\phi_2) = 1;\]

\item[\mylabel{itm:s2a}{(A-2)}]   \quad if $(\phi_1,\phi_2)$ is a $\beta$-type boundary degeneration, then
\[\STA(\phi_1) \cdot \STA(\phi_2) = -1;\]

\item [\mylabel{itm:s3a}{(A-3)}]  \quad if the two pairs $(\phi_1,\phi_2)$ and $(\phi_3,\phi_4)$ form an internal square, then 
\[\STA(\phi_1) \cdot \STA(\phi_2) = -\STA(\phi_3) \cdot \STA(\phi_4);\]

\item[\mylabel{itm:a1}{(A-4)}]  \quad if the two pairs $(\phi_1, \phi_2)$ and  $(\phi_3, \phi_4)$ form a bordered square, then 
\[\STA(\phi_1) \cdot \STA(\phi_2) = \STA(\phi_3) \cdot \STA(\phi_4);\]

\item[\mylabel{itm:a2}{(A-5)}]  \quad if the pair  $(\phi_1, \phi_2)$ and  the flow $\phi_3$ form a bordered triangle, then
\[\STA(\phi_1) \cdot \STA(\phi_2) = \STA(\phi_3).\]
\end{itemize}
\end{defn}

\begin{note}
The first three properties are analogous to those required of a sign assignment in \cite{hfz}, while the remaining two correspond to desired properties of a type $A$ structure. Properties~\ref{itm:s3a} and \ref{itm:a1} say that internal squares anti-commute and bordered squares commute, respectively.
\end{note}

\begin{defn}\label{def:sd}
A \emph{bordered sign assignment of type $D$, compatible with $P$,  and of power $n$} is a map $\STD$ from the set of formal flows $\FnA^L$ to $\{\pm 1\}$ with the following properties:
\begin{itemize}
\item[\mylabel{itm:s1d}{(D-1)}] \quad if $(\phi_1,\phi_2)$ is an $\alpha$-type boundary degeneration, then
\[\STD(\phi_1) \cdot \STD(\phi_2) = 1;\]

\item[\mylabel{itm:s2d}{(D-2)}]   \quad if $(\phi_1,\phi_2)$ is a $\beta$-type boundary degeneration, then
\[\STD(\phi_1) \cdot \STD(\phi_2) = -1;\]

\item [\mylabel{itm:s3d}{(D-3)}]  \quad if the two pairs $(\phi_1,\phi_2)$ and $(\phi_3,\phi_4)$ form an internal square, then
\[\STD(\phi_1) \cdot \STA(\phi_2) = -\STD(\phi_3) \cdot \STA(\phi_4);\]

\item[\mylabel{itm:d1}{(D-4)}]  \quad if the two pairs $(\phi_1, \phi_2)$ and  $(\phi_3, \phi_4)$ form a bordered square and the bordered flows are of type $\rho$, then
\[(-1)^{\mid \rho \mid}\STD(\phi_1) \cdot \STD(\phi_2) \cdot \STD(\phi_3) \cdot \STD(\phi_4)=-1;\]
\end{itemize}
\end{defn}

\begin{rmk}
Recall that the grading $|\rho|$ depends on the sign sequence $P$.
\end{rmk}

\begin{defn}\label{def:gauge}

Let $S$ and $S'$ be two bordered sign assignments (both of type $A$ or both of type $D$). We say that $S$ and $S'$ are \emph{gauge equivalent} if there exists a map $u: \GnA \to \{\pm 1\}$ such that for any $\phi \in \FnA$ a formal flow $\phi : \x \to \y$ we have that $S(\phi)= u(\x)S'(\phi)u(\y)$. We also call $u$ a \emph{gauge equivalence}.
\end{defn}

Several distinct gauge equivalence classes of type $A$, as well as of type $D$, bordered sign assignments exist (see Section~\ref{sec:cfa}). Further, not every ``pairing" of a type $A$ and a type $D$ sign assignment results in a sign assignment in the closed sense. We discuss pairing further in Section~\ref{sec:pairing}.


\section{Type $A$ sign assignments}
\label{sec:cfa}

\subsection{The existence of a type $A$ sign assignment}
\label{ssec:a-exist}

 Our first  step in defining a bordered sign assignment  $\STA$ of  type $A$ is to define $\STA$ on all formal internal flows. For $\s\in \{1,2\}$, let $\GnA^{\s, R}$ be the set of formal right bordered generators of form $(\sigma, \e, \s)$.  Define two injections $f_{\s}:\GnA^{\s, R}\to \gens_{n+1}$ as follows.

 If $\x = (\sigma, \e, \s)$, define $f_{\s}(\sigma)$
 to be the permutation that fixes $n+1$ and maps $i$ to $\sigma(i)$ for $i = 1, \ldots, n$. Define $f_{\s}(\e)$ as the $(n+1)$-tuple in $\{\pm 1\}^{n+1}$ whose first $n$ coordinates agree with those of $\e$, and whose last coordinate is $1$. Define $f_{\s}(\x) = (f_{\s}(\sigma), f_{\s}(\e))$.

Graphically, for $\x\in \GnA^{\s, R}$ we think of $f_\s(\x)$ as the formal generator obtained from $\x$ by replacing the label $\alpha_\s^a$ with $\alpha_n$, and adding one more cross with arcs labeled $\alpha_{n+1}$ and $\beta_{n+1}$ and oriented so their intersection is positive (with the convention that the $\alpha$-arc comes first). See for example Figure~\ref{fig:f_gen}.

\begin{figure}[h]
\centering
  \labellist
  	\pinlabel \textcolor{red}{$\alpha_1$} at 2 100
	\pinlabel \textcolor{blue}{$\beta_{\sigma(1)}$} at 55 100
	\pinlabel \textcolor{red}{$\alpha_2$} at 90 100
	\pinlabel \textcolor{blue}{$\beta_{\sigma(2)}$} at 145 100
	\pinlabel $\dots$ at 172 140
	\pinlabel \textcolor{red}{$\alpha_{n-1}$} at 195 100
	\pinlabel \textcolor{blue}{$\beta_{\sigma(n-1)}$} at 250 100
	\pinlabel \textcolor{red}{$\alpha_1^a$} at 285 100
	\pinlabel \textcolor{blue}{$\beta_{\sigma(n)}$} at 340 100
	\pinlabel \textcolor{red}{$\alpha_1$} at 2 -10
	\pinlabel \textcolor{blue}{$\beta_{\sigma(1)}$} at 55 -10
	\pinlabel \textcolor{red}{$\alpha_2$} at 90 -10
	\pinlabel \textcolor{blue}{$\beta_{\sigma(2)}$} at 145 -10
	\pinlabel $\dots$ at 172 27
	\pinlabel \textcolor{red}{$\alpha_{n-1}$} at 195 -10
	\pinlabel \textcolor{blue}{$\beta_{\sigma(n-1)}$} at 250 -10
	\pinlabel \textcolor{red}{$\alpha_n$} at 290 -10
	\pinlabel \textcolor{blue}{$\beta_{\sigma(n)}$} at 340 -10
	\pinlabel \textcolor{red}{$\alpha_{n+1}$} at 380 -10
	\pinlabel \textcolor{blue}{$\beta_{n+1}$} at 440 -10
   \endlabellist
 \includegraphics[scale=.7]{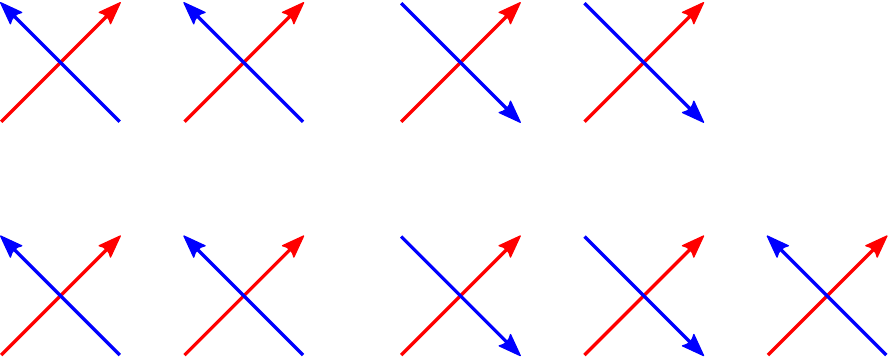} 
      \vskip .3 cm
       \caption{Top: A formal bordered generator $\x\in \GnA^{1, R}$. Bottom: The formal generator $f_1(\x)\in \gens_{n+1}$.}\label{fig:f_gen}
\end{figure}

 Next, we extend the maps $f_\s$ to internal flows. 
Note that if  $\phi$ is an internal flow from $\x = (\sigma_1, \e_1, \s_1)$ to $\y= (\sigma_2, \e_2, \s_2)$, then $\s_1 = \s_2$, so simply write $\s$ for $\s_1$ or $\s_2$. We define $f(\phi)$ as the formal flow from $f_{\s}(\x)$ to $f_{\s}(\y)$ that is obtained from $\phi$ graphically by relabeling $\alpha_{\s}$ to $\alpha_n$ and adding a cross with arcs labeled $\alpha_{n+1}$ and $\beta_{n+1}$ and oriented so their intersection is positive. See for example Figure~\ref{fig:f_flow}.

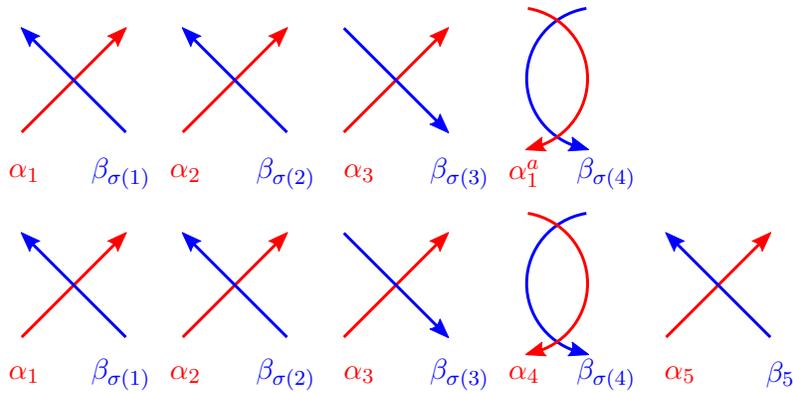
\begin{figure}[h]
\centering
  \labellist
  	\pinlabel \textcolor{red}{$\alpha_1$} at 2 100
	\pinlabel \textcolor{blue}{$\beta_{\sigma(1)}$} at 55 100
	\pinlabel \textcolor{red}{$\alpha_2$} at 90 100
	\pinlabel \textcolor{blue}{$\beta_{\sigma(2)}$} at 145 100
	\pinlabel \textcolor{red}{$\alpha_{3}$} at 185 100
	\pinlabel \textcolor{blue}{$\beta_{\sigma(3)}$} at 240 100
	\pinlabel \textcolor{red}{$\alpha_1^a$} at 275 100
	\pinlabel \textcolor{blue}{$\beta_{\sigma(4)}$} at 320 100
	\pinlabel \textcolor{red}{$\alpha_1$} at 2 -10
	\pinlabel \textcolor{blue}{$\beta_{\sigma(1)}$} at 55 -10
	\pinlabel \textcolor{red}{$\alpha_2$} at 90 -10
	\pinlabel \textcolor{blue}{$\beta_{\sigma(2)}$} at 145 -10
	\pinlabel \textcolor{red}{$\alpha_{3}$} at 185 -10
	\pinlabel \textcolor{blue}{$\beta_{\sigma(3)}$} at 240 -10
	\pinlabel \textcolor{red}{$\alpha_4$} at 275 -10
	\pinlabel \textcolor{blue}{$\beta_{\sigma(4)}$} at 320 -10
	\pinlabel \textcolor{red}{$\alpha_{5}$} at 360 -10
	\pinlabel \textcolor{blue}{$\beta_{5}$} at 415 -10
   \endlabellist
 \includegraphics[scale=.7]{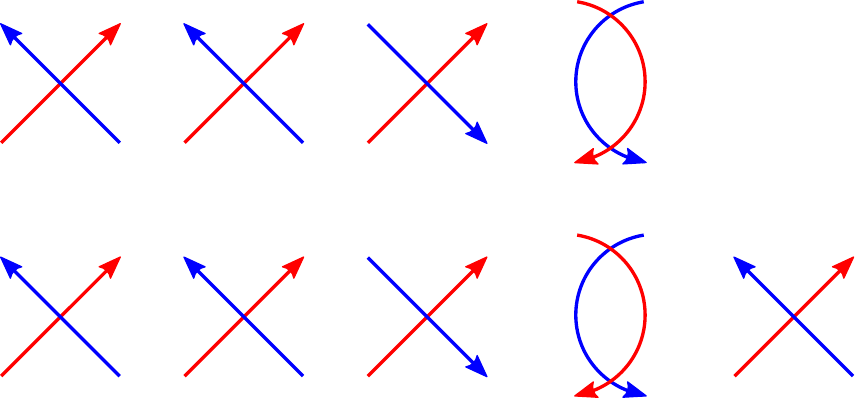} 
      \vskip .3 cm
       \caption{Top: A formal internal flow $\phi\in\FnA^R$. Bottom: The formal flow $f(\phi)\in \mathcal{F}_{n+1}$. (Here $n=4$.)}\label{fig:f_flow}
\end{figure}

\begin{defn}\label{def:a-formula}
Fix $P = (++)$. Let $\STA$ be the map $\FnA^R \to \{\pm 1\}$ defined as follows.
\begin{enumerate}
\item if $\phi$ is an internal flow, then $\STA(\phi)=S(f(\phi))$, where $S$ is the sign assignment defined in \cite{hfz};
\item if $\psi$ is a bordered flow where the starting and ending generator have the same sign at the moving coordinate, then $\STA(\psi)=1$;
\item if $\psi$ is a bordered flow with starting generator $\x$ and the starting and ending generator have different signs at the moving coordinate, then $\STA(\psi)=\sign(\sigma_{\x})  \e_n(\x)$.
\end{enumerate}
\end{defn}

\begin{note}
We can think of $\STA$ as two functions---$S_{\mathrm{int}}$ defined on internal flows, and $S_{\partial}$ defined on bordered flows. This motivates the following definition.
\end{note}

\begin{defn}\label{def:int-sa}
An \emph{internal sign assignment} is a map $S_{\mathrm{int}}$ defined only on internal flows by $S_{\mathrm{int}}(\phi)=S(f(\phi))$ where $S$ is any closed sign assignment.
\end{defn}

\begin{thm}\label{thm:sta_exist}
The function $\STA$ is a bordered sign assignment of type $A$.
\end{thm}

\begin{proof}
We first show that $\STA$ satisfies properties \ref{itm:s1a} through \ref{itm:s3a}. For example, suppose $(\phi_1, \phi_2)$ is an $\alpha$-type boundary degeneration. Then $(f(\phi_1), f(\phi_2))$ is an $\alpha$-type boundary degeneration in  $\mathcal{F}_{n+1}$, so $\STA(\phi_1) \cdot \STA(\phi_2) = S(f(\phi_1)) \cdot S(f(\phi_2)) = 1$, i.e.\ \ref{itm:s1a} holds. Analogous arguments show that \ref{itm:s2a} and \ref{itm:s3a} hold.

Next, we show that $\STA$ satisfies \ref{itm:a1}. Suppose  $(\phi_1: \x \to \x', \phi_2: \x' \to \y')$ and $(\phi_3: \x \to \y,\phi_4: \y \to \y')$ form a bordered square. Schematically, we have 
\begin{center}
    \begin{tikzpicture}
      \node at (0,0) (x) {$\x$};
      \node at (0,-2) (y) {$\y$};
      \node at (2,-0) (x') {$\x'$};
      \node at (2,-2) (y') {$\y'$};
      \draw[->] (x) -- (x') node[midway,above] {$\phi_1$};
      \draw[->]  (y) -- (y') node[midway,below] {$\phi_4$};
      \draw[->]  (x) -- (y) node[midway,left] {$\phi_3$};
      \draw[->]  (x') -- (y') node[midway,right] {$\phi_2$};
    \end{tikzpicture}
    \end{center}
Without loss of generality, assume that $\phi_1$ and $\phi_4$ are bordered flows, and $\phi_2$ and $\phi_3$ are internal flows. 

{\bf{Case 1: $\phi_2$ is a bigon.}} Then $\phi_3$ is a bigon too, and both bigons have the same moving coordinate $i$. We have 
\begin{align*}
\STA(\phi_2)&= S(f(\phi_2)) = S_0(f(\phi_2))\prod_{k=1}^{i-1} \e_k(f_{\s(\x')}(\x')) = S_0(f(\phi_2))\prod_{k=1}^{i-1} \e_k(\x')\\
\STA(\phi_3)&= S(f(\phi_3)) = S_0(f(\phi_3))\prod_{k=1}^{i-1} \e_k(f_{\s(\x)}(\x)) = S_0(f(\phi_3))\prod_{k=1}^{i-1} \e_k(\x)
\end{align*}
where the second equality in each line is the sign formula from \cite{hfz} \footnote{We caution the reader comparing  with \cite{hfz}  that there is a typo at  the end of the proof of \cite[Proposition 5.3]{hfz} -- the formula at the end should be the same as the one at the start of the proof, i.e.\ there should be no $sigma$ in the formula.}, and $S_0$ is a function that only depends on the decorated domain. Since $\phi_1$ is a bordered flow, $\x$ and $\x'$ agree at the first $n-1$ coordinates. In particular $\e_k(\x) = \e_k(\x')$ for $k<n$, so 
\[\STA(\phi_2)\STA(\phi_3) = S_0(f(\phi_2))S_0(f(\phi_3)).\]

{\bf{Case 1.1.}}  Suppose the moving coordinate of $\phi_2$, and hence of $\phi_3$, is $i<n$. Then $\phi_1$ and $\phi_2$ do not share moving coordinates, and neither do $\phi_3$ and $\phi_4$. It follows that $\phi_2$ and $\phi_3$ have the same decorated domain, so $S_0(f(\phi_2)) = S_0(f(\phi_3))$, hence $\STA(\phi_2)\STA(\phi_3) = 1$.

If $\e_n(\x) = e_n(\x')$, by definition we have $\STA(\phi_1)=1$. Since $\phi_1$ and $\phi_4$ have the same decorated domain, we also have  $\e_n(\y) = \e_n(\y')$, and hence $\STA(\phi_4) = 1$. Property~\ref{itm:a1} then holds for the given square. 

If, on the other hand, $\e_n(\x) = -\e_n(\x')$, then we also have $\e_n(\y) = - \e_n(\y')$, so 
\[\STA(\phi_1)\STA(\phi_4) = \sign(\sigma_{\x})  \e_n(\x)\sign(\sigma_{\y})  \e_n(\y).\]
Since $\phi_3: \x \to \y$ is a bigon, we have $\sigma_{\x} = \sigma_{\y}$, and since the bigon's moving coordinate is not $n$, we have $\e_n(\x) = \e_n(\y)$.  Thus, $\STA(\phi_1)\STA(\phi_4) = 1$, and \ref{itm:a1} holds in this case too. 

{\bf{Case 1.2.}}  Now suppose  the moving coordinate of $\phi_2$, and hence of $\phi_3$, is $n$. Then $\phi_1$ and $\phi_2$ share a moving coordinate, and so do $\phi_3$ and $\phi_4$. The square locally looks like one of the last two diagrams in Figure~\ref{fig:z-bord-sq}.

If $\e_n(\x) = \e_n(\x')$, it is easy to see that the decorated domains for $\phi_2$ and $\phi_3$ may only differ at the label of the $\alpha$-arc ($\alpha_1^a$ or $\alpha_2^a$), so $f(\phi_2)$ and $f(\phi_3)$ have the same decorated domain. Then  $S_0(f(\phi_2)) = S_0(f(\phi_3))$, so once again we have  $\STA(\phi_2)\STA(\phi_3) = 1$. Since $\phi_3: \x \to \y$ is a bigon with moving coordinate $n$, we have $\e_n(\y) = -\e_n(\x)$. Similarly, $\e_n(\y') = -\e_n(\x')$. So $\e_n(\y) = \e_n(\y')$. It follows that $\STA(\phi_1)= 1 = \STA(\phi_4)$. Thus, \ref{itm:a1} holds. 

If $\e_n(\x) = -\e_n(\x')$, the decorated domains for $f(\phi_2)$ and $f(\phi_3)$ have the same $\beta$-orientation but different $\alpha$-orientations.  By \cite{hfz}, $S_0(f(\phi_2)) = -S_0(f(\phi_3))$, so this time  $\STA(\phi_2)\STA(\phi_3) = -1$. This time  $\e_n(\y) = -\e_n(\x)$ and $\e_n(\y') = -\e_n(\x')$ imply $\e_n(\y) = - \e_n(\y')$, so 
\[\STA(\phi_1)\STA(\phi_4) = \sign(\sigma_{\x})  \e_n(\x)\sign(\sigma_{\y})  \e_n(\y)=-1.\]
Thus, \ref{itm:a1} holds again. 

{\bf{Case 2: $\phi_2$ is a rectangle.}} 
 Then $\phi_3$ is a rectangle, and both rectangles have the same moving coordinates $i,j$. However, since $\phi_3$ occurs before the bordered flow $\phi_4$ in the composition, it is important to note that $\sigma_{\x}=\sigma_{\y} \cdot (i \ j)$ and as such $\sign(\sigma_{\x})=-\sign(\sigma_{\y})$.

{\bf{Case 2.1.}} Suppose the moving coordinates of $\phi_2$, and hence of $\phi_3$, are $i,j \neq n$. Note that this implies $\e_n(\x)=\e_n(\y)$. Then $\phi_1$ and $\phi_2$ do not share moving coordinates, and neither do $\phi_3$ and $\phi_4$. It follows that $\phi_2$ and $\phi_3$ have the same decorated domain. 

If $\e_n(\x)=\e_n(\x')$, then we also have $\e_n(\y)=\e_n(\y')$. This in turn implies that $f_\s(\x)=f_{\s'}(\x')$ and $f_\s(\y)=f_{\s'}(\y')$, and therefore $f(\phi_2)=f(\phi_3)$. So
\[\STA(\phi_2)=S(f(\phi_2))=S(f(\phi_3))=\STA(\phi_3).\]
Finally, since $\e_n(\x)=\e_n(\x')$ and $\e_n(\y)=\e_n(\y')$ we have $\STA(\phi_1)=1=\STA(\phi_4)$. Therefore we have $\STA(\phi_1) \cdots \STA(\phi_4)$=1 and it follows that Property~\ref{itm:a1} holds for this square. 

If $\e_n(\x)= - \e_n(\x')$, then we also have $\e_n(\y)= - \e_n(\y')$. This in turn implies that while $\phi_2$ and $\phi_3$ have the same decorated domain, if we examine $f(\phi_2)$ and $f(\phi_3)$ we see that they are related by a simple flip at the $n^{th}$ coordinate. The two bigons $b:f_{\s(\x)}(\x) \to f_{\s(\x')}(\x'), b':f_{\s(\y)}(\y) \to f_{\s(\y')}(\y')$ in the square for said simple flip share the same decorated domain, which implies $S_0(b)=S_0(b')$. Then when examining $\e_k(f_{\s(\x)}(\x))$ versus $\e_k(f_{\s(\y)}(\y))$ for all k, we see that they either agree at all places, or differ at both $i$ and $j$ and agree everywhere else. In either scenario, we have 
\[S(b)=S_0(b)\prod_{k=1}^{n-1}\e_k(f_{\s(\x)}(\x))=S_0(b')\prod_{k=1}^{n-1}\e_k(f_{\s(\y)}(\y))=S(b').\]
Looking back at the definition of a simple flip, this tells us 
\[\STA(\phi_2)=S(f(\phi_2))= - S(b)S(f(\phi_3))S(b')= - \STA(\phi_3).\]
However, since $\e_n(\x)= - \e_n(\x')$ and $\e_n(\y)= - \e_n(\y')$, we note that $\STA(\phi_1)=\e_n(\x) \sign(\sigma_{\x})=\e_n(\y)( - \sign(\sigma_{\y}))= - \STA(\phi_4)$. Therefore we have
\[\STA(\phi_1)\STA(\phi_2)\STA(\phi_3)\STA(\phi_4)=\STA(\phi_1)\STA(\phi_2)( - \STA(\phi_2))( - \STA(\phi_1))=1\]
and we see that Property~\ref{itm:a1} still holds.

{\bf{Case 2.2.}} Suppose one of the moving coordinates of $\phi_2$, and hence of $\phi_3$, is the $n^{th}$ coordinate. Similarly to  Case 2.1, we note that $\sigma_{\x}=\sigma_{\y} \cdot (i \ n)$ and as such $\sign(\sigma_{\x})=-\sign(\sigma_{\y})$. Because one of the moving coordinates for $\phi_2$ and $\phi_3$ is $n$, this implies that we are considering a square that comes from one of the first two cases of Figure~\ref{fig:z-bord-sq}.

If $\e_n(\x)=\e_n(\x')$, then we also have $\e_n(\y)=\e_n(\y')$. This in turn implies that $f_\s(\x)=f_{\s'}(\x')$ and $f_\s(\y)=f_{\s'}(\y')$, and therefore $f(\phi_2)=f(\phi_3)$. So
\[\STA(\phi_2)=S(f(\phi_2))=S(f(\phi_3))=\STA(\phi_3).\]
Finally, since $\e_n(\x)=\e_n(\x')$ and $\e_n(\y)=\e_n(\y')$ we have $\STA(\phi_1)=1=\STA(\phi_4)$. Therefore we have $\STA(\phi_1) \cdots \STA(\phi_4)$=1 and it follows that Property~\ref{itm:a1} holds for this square. 

If $\e_n(\x)= - \e_n(\x')$, then we also have $\e_n(\y)= - \e_n(\y')$. Then there exist formal bigons $b:f_{\s(\x)}(\x) \to f_{\s(\x')}(\x'),b':f_{\s(\y)}(\y) \to f_{\s(\y')}(\y')$ from $f_{\s(\x)}(\x)$ to $f_{\s(\x')}(\x')$ and from $f_{\s(\y)}(\y)$ to $f_{\s(\y')}(\y')$ respectively that together with $f(\phi_2)$ and $f(\phi_3)$ (who share a geometric domain) form a square. Note that $b,b'$ share a decorated domain, and as such $S_0(b)=S_0(b')$.

{\bf{Case 2.2.1.}} Suppose that the two edges of the (geometric) domain for $f(\phi_2)$ and $f(\phi_3)$ with $\beta$ labels have the same orientation. Note that this implies $\e_n(\x)=\e_n(\y)$. Then $\e_k(f_{\s(\x)}(\x))=\e_k(f_{\s(\y)}(\y))$ for all $k$, and thus $S(b)=S(b')$. However, this implies that 
\[\STA(\phi_2)=S(f(\phi_2))= - S(f(\phi_3))=- \STA(\phi_3)\]
 since the four flows made a square and the product of their signs must give -1. But we also have that 
\[\STA(\phi_1)=\e_n(\x)\sign(\sigma_{\x})=\e_n(\y)( - \sign(\sigma_{\y}))= - \STA(\phi_4).\]
Putting it altogether, we see that
\[\STA(\phi_1)\STA(\phi_2)\STA(\phi_3)\STA(\phi_4)=\STA(\phi_1)\STA(\phi_2)( - \STA(\phi_2))( - \STA(\phi_1))=1\]
and therefore Property~\ref{itm:a1} still holds.

{\bf{Case 2.2.2}} Suppose that the two edges of the (geometric) domain for $f(\phi_2)$ and $f(\phi_3)$ with $\beta$ labels differ in orientation. Note that this implies $\e_n(\x)= - \e_n(\y)$. Then $\e_i(f_{\s(\x)}(\x))= - \e_i(f_{\s(\y)}(\y))$, and $\e_k(f_{\s(\x)}(\x))=\e_k(f_{\s(\y)}(\y))$ for all $k \neq i,n$. Thus, 
\[S(b)=S_0(b)\prod_{k=1}^{n-1}\e_k(f_{\s(\x)}(\x))=S_0(b')( - \prod_{k=1}^{n-1}\e_k(f_{\s(\y)}(\y)))= - S(b').\]
As before, since these two bigons make a square with $f(\phi_2)$ and $f(\phi_3)$, the product of the signs of these four flows must give -1. This then implies that $\STA(\phi_2)=S(f(\phi_2))=S(f(\phi_3))=\STA(\phi_3)$. We also have that 
\[\STA(\phi_1)=\e_n(\x)\sign(\sigma_{\x})=( - \e_n(\y))( - \sign(\sigma_{\y}))=\STA(\phi_4).\]
As such, we see that $\STA(\phi_1) \cdots \STA(\phi_4)= 1$ and therefore Property~\ref{itm:a1} holds.

Finally, we show that the sign assignment $\STA$ satisfies  \ref{itm:a2}. Suppose the pair  $(\phi_1:\x \to \y, \phi_2: \y\to \z)$ and the flow $\phi_3: \x\to \z$ form a bordered triangle.

Suppose  $\e_n(\x) = \e_n(\y)$. Orient the unlabeled edge in the geometric representation of the bordered triangle so it agrees with the orientation on the boundary of the domains of the flows. Observing that the neighborhood of that edge embeds in the signed pointed matched circle  $\ZZ_P$, 
 we see that at most two of the intersections of the edge with the $\alpha$-arcs have the same sign. Equivalently, at most two of $\e_n(\x), \e_n(\y), \e_n(\z)$ are the same, so $\e_n(\z) = -\e_n(\x) = -\e_n(\y)$. Then 
\begin{align*}
\STA(\phi_1) &= 1,\\
\STA(\phi_2) &= \sign(\sigma_{\y})\e_n(\y),\\
\STA(\phi_3) &= \sign(\sigma_{\x})\e_n(\x).
\end{align*} 
Since $\phi_1$ is bordered, we have $\sigma_{\x} = \sigma_{\y}$, so $\STA(\phi_2) = \STA(\phi_3)$ and  \ref{itm:a2} holds. 

Now suppose  $\e_n(\x) = - \e_n(\y)$. When traversing the signed pointed matched circle we have chosen, signs may only flip once. Equivalently, when traversing the unlabeled edge, its intersection signs with the $\alpha$-arcs may only flip once. Hence the sign sequence $\e_n(\x), \e_n(\y), \e_n(\z)$ may only have one flip, so $\e_n(\z) = \e_n(\y) = -\e_n(\x)$. Then 
\begin{align*}
\STA(\phi_1) &= \sign(\sigma_{\x})\e_n(\x),\\
\STA(\phi_2) &= 1,\\
\STA(\phi_3) &= \sign(\sigma_{\x})\e_n(\x),
\end{align*} 
and \ref{itm:a2} holds again.
\end{proof}

While this proves there exists a type $A$ sign assignment for the signed matched pointed circle given by $P=(++)$, the astute reader may ask about the other sign sequences.

\begin{lem}\label{lem:existA}
There exists a bordered sign assignment of type $A$ for any sign sequence $P' \in \{+,-\}^2$.
\end{lem}

\begin{proof}
Suppose $P'\neq (++)$. Then we construct a bijection $p'$ on the set of formal right generators and flows for $P'$ to the set of formal right generators and flows for $P$. in the following way. Let $\x$ be a formal generator with $\s(\x)=i$. If the $i^{th}$ coordinate of $P'$ is not $+$, map $\x$ to the generator that only differs from $\x$ in the sign of $\epsilon_n(\x)$; otherwise map $\x$ to itself. We can think of this as merely flipping the orientation of $\alpha_i^a$ in the graphical presentation of $\x$. In the same vein, let $\phi$ be a formal flow. For each coordinate of $P'$ that is not $+$, flip the orientation of the corresponding $\alpha_i^a$ in the graphical presentation of $\phi$, i.e.\ if the first coordinate of $P'$ is not $+$, flip the orientation of $\alpha_1^a$ and so on. Let the resulting flow be the image of $\phi$ under the bijection. 

Define a function $S_{P'}$ by the composition $\STA \circ p'$, where $\STA$ is the sign assignment defined above. We observe that, by construction, if a set of flows forms an $\alpha$-degeneration, a $\beta$-degeneration, an internal square, a bordered square, or a bordered triangle, then so does the image of that set of flows under $p'$. For example, if $(\phi_1,\phi_2), (\phi_3,\phi_4)$ form an internal square, then $(p'(\phi_1),p'(\phi_2)),(p'(\phi_3),p'(\phi_4))$ also form an internal square. Then since $\STA$ satisfies \ref{itm:s1a} through \ref{itm:a2}, it follows that $S_{P'}$ satisfies these properties as well. Thus $S_{P'}$ is a bordered sign assignment of type $A$ for the sign sequence $P'$.
 \end{proof}

\subsection{Equivalence classes of type $A$ sign assignments}
\label{ssec:a-unique}

In this section we discuss how the behavior of any given bordered sign assignment can be extrapolated from a small set of choices, and then show how the set of possible choices breaks down into different equivalence classes.

\begin{prop} \label{prop:choices}
A bordered sign assignment $S$ of type $A$ is determined by three choices on three specific bordered flows, and a choice of internal sign assignment on each idempotent.
\end{prop}

\begin{proof}
By definition, a choice of internal sign assignment determines the sign of all internal flows on that idempotent, so it remains to show that the sign for all bordered flows can be deduced from the choices made. Let $c_1(S) \in \{\pm 1\}$ be the sign of the bordered flow $\phi_1$ of type $\rho_1$ with starting generator $\x=\{\id, \mathbf{1}, 1\}$, $c_2(S)$ be the sign of the bordered flow $\phi_2$ of type $\rho_2$ with starting generator $\x=\{\id, \mathbf{1}, 2\}$, and $c_3(S)$ be the sign of the bordered flow $\phi_3$ of type $\rho_3$ with starting generator $\x=\{\id, \mathbf{1}, 1\}$. Then for any bordered flow $\phi_i'$ of type $\rho_i$ whose starting generator differs by a single transposition in the permutation or by a single coordinate in the sign profile, there exists a bordered square involving $\phi_i,\phi_i',$ and two internal flows (rectangles in the case of the transposition, or bigons in the case of the sign profile change). Then by using property \ref{itm:a1} of a bordered sign assignment and the fact that $S(\phi_i)$ and the signs of the internal flows are all known, we can deduce the value of $S(\phi_i')$. By allowing further transpositions and sign profile changes, we can inductively deduce the sign of any bordered flow of type $\rho_1,\rho_2,$ or $\rho_3$. Then any flow of type $\rho_{12}$ can be built as a composition of flows of type $\rho_1$ and $\rho_2$ using a bordered triangle. Then property \ref{itm:a2} and the fact that all flows of type $\rho_1$ and $\rho_2$ are now known allows us to deduce the signs of type $\rho_{12}$ flows. Flows of type $\rho_{23}$ can be handled similarly using bordered triangles involving flows of type $\rho_2$ and $\rho_3$. Lastly since we know the signs of flows of type $\rho_{12}$ or $\rho_{23}$, we can build bordered triangles with a type $\rho_{12}$ and $\rho_3$ (or a type $\rho_1$ and $\rho_{23}$ to deduce the sign of any flow of type $\rho_{123}$.
\end{proof}

Note that we will omit the sign assignment from the notation when it is clear from context or unnecessary, and write either $c_i$ or $c_i^A$ depending on whether or not it is relevant that it is a type $A$ sign assignment being discussed.

\begin{prop}\label{prop:fourA}
There are four bordered sign assignments of type $A$ and power $n$ for any given signed matched pointed circle and any choice of $\mathbb{Z}/2$ grading, up to gauge equivalence.
\end{prop}

\begin{proof}
In \cite{hfz}, it was shown that all closed sign assignments of a given power are gauge equivalent. Therefore, all internal sign assignments (of the same power) are also gauge equivalent.

Thus it remains to discuss the eight cases that arise from the tuple of three choices $(c_1, c_2, c_3)$ on bordered flows. Let us assume that we fix a consistent internal sign assignment, which can be done by applying the aforementioned gauge equivalence. Consider the function $u:\GnA \to {\pm1}$ defined by $u(\x)=1$ if $s(\x)=1$ and $u(\x)=-1$ otherwise. Then we wish to show this map is a gauge equivalence between the two sign assignments $S$ and $S'$ given by any tuple of choices and its negation, e.g. $(+,-,+)$ and $(-,+,-)$. 

First, note that for any internal flow $\phi: \x \to \y$ we have $s(\x)=s(\y)$, and so $u(\x)S(\phi)u(\y)=S'(\phi)$. Then observe that all flows of types $\rho_1, \rho_2, \rho_3,$ and $\rho_{123}$ move between generators on different idempotents, and for any one of these flows $\phi$, we have $S(\phi)=-S'(\phi)$ because of our choice of tuples and property \ref{itm:a2}. However since for any given $\phi:\x \to \y$ in this set of flows we have $s(\x)=-s(\y)$, we know $u(\x)u(\y)=-1$. Therefore $u(\x)S(\phi)u(\y)=-S(\phi)=S'(\phi)$. Now let $\phi_1, \phi_2,$ and $\phi_{12}$ respectively be flows of type $\rho_1, \rho_2,$ and $\rho_{12}$ that form a bordered triangle. Then if $\phi_{12}$ is a flow from $\x$ to $\y$, we know $u(\x)u(\y)=1$ (because $\x$ and $\y$ are on the same idempotent) and thus $u(\x)S(\phi_{12})u(\y)=S(\phi_{12})=S(\phi_1)S(\phi_2)=(-S'(\phi_1))(-S'(\phi_2))=S'(\phi_{12})$. A similar argument holds for flows of type $\rho_{23}$.   Thus $u$ is a gauge equivalence between $S$ and $S'$. 

This leaves four classes to examine, which can be classified by what they do on flows of type $\rho_{12}$ and $\rho_{23}$. Recalling property \ref{itm:a2} and that $c_1,c_2$ and $c_3$ are the signs for flows of type $\rho_1,\rho_2$ and $\rho_3$ with specific starting generators, we can see that the flow of type $\rho_{12}$ with starting generator $\x=(\id,\mathbf{1},1)$ has sign $c_1c_2$ and similarly the flow of type $\rho_{23}$ with starting generator $\x=(\id,\mathbf{1},2)$ has sign $c_2c_3$. We can thus reframe our presentation of $S$ as the tuple of choices $(c_1,c_2,c_3)$ to instead be given by the pair $(c_1c_2,c_2c_3)$, which is indeed invariant under negation of all three $c_i's$. Then our four potential classes can then be represented by the four tuples $\{(+,+),(+,-),(-,+),(-,-)\}$, and we will denote these classes respectively by $S_{++},S_{+-},S_{-+},$ and $S_{--}$.

We wish to show that these four classes are all distinct under gauge equivalence, i.e. no two of them are gauge equivalent. Consider the two classes $S_{++}$ and $S_{-+}$. As representatives choose two sign assignments $S_1$ and $S_2$ that have the same internal sign assignment (again by applying the gauge equivalence defined in \cite{hfz} if necessary) and are generated by the sign  tuples $(c_1(S_1), c_2(S_1), c_3(S_1)) = (+,+,+)$ and $(c_1(S_2),c_2(S_2),c_3(S_2))=(+,-,-)$ respectively. Let $\phi_1$ and $\phi_2$ be bordered flows of type $\rho_{12}$ whose starting generators differ by a transposition $(i \, j)$. Then there exist a pair of internal rectangles $R_{ij},R_{ij}'$ that form a bordered square with $\phi_1$ and $\phi_2$. Then we also know there exists bigons $b_1$ and $b_2$ that form an internal square with $R_{ij}$ and $R_{ij}'$ such that $b_1$ and $b_2$ has the same starting and ending generators as $\phi_1$ and $\phi_2$ respectively. 

Suppose there is a gauge equivalence $u$ between $S_1$ and $S_2$. Let $\x$ and $\y$ be the starting and ending generators of $\phi_1$. Then since $S_1(\phi_1)=-S_2(\phi_1)$, we must have that $u(\x)u(\y)=-1$. However, since $S_1$ and $S_2$ agree on all internal flows, we also know that $S_1(b_1)=S_2(b_1)$ and thus we must have that $u(\x)u(\y)=1$. This is a contradiction and so $S_1$ and $S_2$ are not equivalent. Similar arguments hold for each other pair of classes among $\{S_{++}, S_{+-}, S_{-+}, S_{--}\}$, and thus none of these four classes are equivalent.
\end{proof}
 
 When needed, we will think of the class of bordered sign assignment $\STA$ of type $A$ as being given by the tuple $(c_1(\STA)c_2(\STA), c_2(\STA)c_3(\STA))$ as this tells us both the general class and the specific behavior.

\subsection{The type $A$ structure $\cfahat$ with signs}
\label{ssec:a-def}

 \begin{defn}
 Let $\HD = (\Sigma, \alpha,\beta, z)$ be a nice Heegaard diagram for a bordered 3-manifold $Y$ with torus boundary. Fix an ordering and orientation on  the $\alpha$-circles, $\beta$-circles, and $\alpha$-arcs. Let $P$ be the sign sequence for which $\bdy \HD$ can be identified  with $\ZZ_P$ (so that each $\alpha$-arc is oriented from a $-$ to a $+$). The generators of the diagram $\HD$ naturally define formal right bordered generators of power $|\beta|$, and the empty bigons and rectangles on $\HD$ (counted by the multiplication maps for $\cfahat(\HD)$) specify formal right  flows of power $|\beta|$.  We further fix a choice $C \in \{+,-\}^2$, and let $\STA$ be a representative of the class $S_C$.

We define a right type $A$ structure $\cfahat_{C,P}(\HD;\Z)$ over $\alg_P$ as follows. Let $\cfahat_{C,P}(\HD;\Z)$ be the $\Z$-module freely generated by the generators $\mathcal G(\HD)$ of the Heegaard diagram. We give $\cfahat_{C,P}(\HD;\Z)$ the structure of a right $\mathcal I_P$-module by setting

\[\x\cdot \iota_t = \begin{cases}
\x \quad \textrm{ if } t = o(\x) \\
0 \quad \textrm{otherwise.}
\end{cases}
\]

A type $A$ structure on $\cfahat_{C,P}(\HD;\Z)$ is  a family of maps
\[\mza_{n+1}\colon \cfahat_{C,P}(\HD;\Z)\otimes_{\mathcal I_P} \alg_P\otimes_{\mathcal I_P} \cdots \otimes_{\mathcal I_P} \alg_P\to \cfahat_{C,P}(\HD;\Z) \]
satisfying Equation~\ref{eqn:Arelns}.

We define these maps by counting the same regions on $\HD$ as for  $\cfahat(\HD)$ (i.e.\ empty bigons and rectangles), but this time with sign. Namely, for $\x\in \mathcal G(\HD)$ and $a\in \alg_P$ of form $a = \rho_I$, we define
\begin{align*} 
\mza_1(\x)&= \sum_{ \y \in \mathbb{T}_\alpha \cap \mathbb{T}_\beta} \ \sum_{\phi \in \mathrm{Flows}_{\mathrm{int}}(\x,\y)} \ \STA(F(\phi)) \cdot \y \\
  \mza_2(\x, a)&= \sum_{ \y \in \mathbb{T}_\alpha \cap \mathbb{T}_\beta} \ \sum_{\phi \in \mathrm{Flows}_{a}(\x,\y)} \ \STA(F(\phi)) \cdot \y \\
    \mza_2(\x, \bf{I})&= \x\\
    \mza_{n+1} &= 0 \quad \textrm{ if $n>1$,}
  \end{align*}
  where $\bf{I}$ is the unit $\iota_1+\iota_2\in \alg_P$.
Here,  $\mathrm{Flows}_{\mathrm{int}}(\x,\y)$ and $\mathrm{Flows}_{a}(\x,\y)$ are the sets of internal and bordered flows, respectively,  between $\x$ and $\y$ on $\HD$, and $F(\phi)$ is the formal flow in $\FnA^R$ that corresponds to $\phi$. Note that if $\phi$ contributes to $\mza_2(\x,a)$, then $F(\phi)$ is a formal flow of type $a$.
 \end{defn}
 
 \begin{thm}
  The data $(\cfahat_{C,P}(\HD;\Z), (\mza_1, \mza_2))$ forms a type $A$ structure over $\alg_P$.
 \end{thm}

 \begin{proof}
Since the  higher multiplication maps are identically zero, we only need to verify these three relations 
 \begin{align*}
 (\mza_1)^2(\x) &= 0\\
\mza_1(\mza_2(\x,a)) &= \mza_2(\mza_1(\x),a) + (-1)^{\mid \x\mid}\mza_2(\x, d_{\alg_P}(a))\\
\mza_2(\x, a\cdot b) &= \mza_2(\mza_2(\x,a), b).
 \end{align*}
 The first relation follows from the fact that the sign assignment $\STA$ satisfies \ref{itm:s3a}. 
For the torus algebra, the differential operator $d_{\alg_P}$ is trivial, so the second relation simplifies to $\mza_1(\mza_2(\x,a)) = \mza_2(\mza_1(\x),a)$; this simplified relation follows from  \ref{itm:a1}. The third relation follows from \ref{itm:a2}.
\end{proof}

\begin{thm}\label{thm:indA}
Let $\HD$ be a bordered diagram for manifold with torus boundary. Fix a choice $C \in \{+,-\}^2$ and a sign sequence $P$. Then the structure $\cfahat_{C,P}(\HD;\Z)$ does not depend on the ordering and orientation of the curves, nor on the choice of representative $\STA$ of the class $S_C$.
\end{thm}

\begin{proof}
First fix an ordering and orientation of the curves of the diagram $\HD$. Suppose $S$ and $S'$ are two different representatives of the class $S_C$, giving rise to maps $(m_1^S,m_2^S)$ and $(m_1^{S'},m_2^{S'})$ respectively. Since $S$ and $S'$ are representatives of the same class, then by Proposition~\ref{prop:fourA} there must exist some gauge equivalence $u$ between them. Then we define a linear map $H:\cfahat_{C,P}(\HD;\Z) \to \cfahat_{C,P}(\HD;\Z)$ on the underlying $\Z$-modules by $H(\x)=u(F(\x))\cdot \x$, where $F(\x)$ is the formal generator associated to the Heegaard diagram generator $\x$. Clearly$H$ induces an isomorphism between $(\cfahat_{C,P}(\HD;\Z),(m_1^S,m_2^S))$ and $(\cfahat_{C,P}(\HD;\Z),(m_1^{S'},m_2^{S'})$. Thus we have independence from the choice of representative.

Now suppose we have a fixed representative $S$ of $S_C$, and a fixed order of the curves. However, let us fix two different orientations. Since the sign sequence $P$ is fixed, these two orientations cannot differ on the $\alpha$-arcs. For simplicity, we will assume the orientations only differ on one curve, say $\alpha_1$. Then note that $\alpha_1$ corresponds to a specific curve in the set $\alphas$ used to define formal generators and flows; we will abuse notation and call this formal curve $\alpha_1$ too. We will notate the two orientations as $o$ and $o'$. 
 
 Define a map $h: \FnA \to \FnA$ on the set of formal flows by mapping a flow $\phi$ to the flow $\phi'$ which is graphically obtained from $\phi$ by flipping the orientation of $\alpha_1$. By an argument similar to Lemma~\ref{lem:existA}, the map $S \circ h$ is a sign assignment, and is also in the class $S_C$. Then we have the maps $(m_1^{S,o},m_2^{S,o})$ defined using the sign assignment $S$ and the orientation $o$ on the diagram, $(m_1^{S\circ h,o'},m_2^{S\circ h.o'})$ using $S \circ h$ and $o'$, and $(m_1^{S,o'},m_2^{S,o'})$ using $S$ and $o'$. We wish to show that $(m_1^{S,o},m_2^{S,o})$ and $(m_1^{S,o'},m_2^{S,o'})$ produce isomorphic structures. However, we notice that by construction $(m_1^{S\circ h,o'},m_2^{S\circ h.o'})$ and $(m_1^{S,o},m_2^{S,o})$ agree on all values, and thus give rise to isomorphic structures. Then since $S \circ h$ and $S$ are in the same class, $(m_1^{S,o'},m_2^{S,o'})$ and $(m_1^{S\circ h,o'},m_2^{S\circ h.o'})$ must also produce isomorphic structures by the argument used above. Thus we have shown independence of orientation.
 
 Finally, suppose we have fixed a representative $S$ of the class $S_C$ and an orientation on the curves. Let us fix two different orderings, noting once again the labels of the $\alpha$-arcs are fixed, and thus cannot differ between these two orderings. Then the permutation that maps one ordering to the other gives rise to a map $g: \FnA \to \FnA$ on the set of formal flows, and we can define a new sign assignment as the composition $S \circ g$. By an analogous argument to the above case of differing orientations, we conclude the independence of ordering.
\end{proof}


\section{Type $D$ sign assignments}
\label{sec:cfd}

\subsection{The existence of a type $D$ sign assignment}
\label{ssec:d-exist}

We begin with an important observation that while formal flows of type $\rho_{12}$ and $\rho_{23}$ do appear as legitimate holomorphic curves in the formal picture, they will never appear in a type $D$ structure derived from a nice diagram. Furthermore, they are identically zero under the type $D$ structure map for the torus algebra, and will never pair with a flow in a type $A$ structure when gluing two diagrams together (as both the starting and ending generator occupy the same idempotent and thus the same $\alpha$ curve in a closed diagram). Since they thus play no role in either the closed structures or the type $D$ structure itself (when working with nice diagrams), we disregard them in the construction and classification of type $D$ bordered sign assignments.

 Our first  step in defining a bordered sign assignment  $\STD$ of  type $D$ is to define $\STD$ on all formal internal flows. For $\s\in \{1,2\}$, let $\GnA^{\s, L}$ be the set of formal left bordered generators of form $(\sigma, \e, \s)$.  Define two injections $f_{\s}:\GnA^{\s, L}\to \gens_{n+1}$ as follows.

 If $\x = (\sigma, \e, \s)$, define $f_{\s}(\sigma)$
 to be the permutation that fixes $1$ and maps $i$ to $\sigma(i-1)+1$ for $i = 2, \ldots, n+1$. Define $f_{\s}(\e)=(1,\epsilon_1,\dots,\epsilon_n)$, i.e $f_{\s}(\e_i)=1$ if $i=1$ and $f_{\s}(\e)=\epsilon_{i-1}$ otherwise. Lastly, define $f_{\s}(\x) = (f_{\s}(\sigma), f_{\s}(\e))$.  

Graphically, for $\x\in \GnA^{\s, L}$ we think of $f_\s(\x)$ as the formal generator obtained from $\x$ by replacing the label $\alpha_\s^a$ with $\alpha_2$, and replacing each remaining label $\alpha_i$ with $\alpha_{i+1}$, for $i=2,\dots,n$. Each $\beta$ index $\sigma(i)$ is replaced with $\sigma(i)+1$, for $i=1, \dots, n$. Finally, we add one more cross with arcs labeled $\alpha_1$ and $\beta_1$, and orient the arcs so their intersection is positive (with the convention that the $\alpha$-arc comes first). See for example Figure~\ref{fig:f_gen_d}.

\begin{figure}[h]
\centering
  \labellist
  	\pinlabel \textcolor{red}{$\alpha_1^a$} at 160 735
	\pinlabel \textcolor{blue}{$\beta_{\sigma(1)}$} at 220 735
	\pinlabel \textcolor{red}{$\alpha_2$} at 275 735
	\pinlabel \textcolor{blue}{$\beta_{\sigma(2)}$} at 328 735
	\pinlabel $\dots$ at 352 775
	\pinlabel \textcolor{red}{$\alpha_{n-1}$} at 377 735
	\pinlabel \textcolor{blue}{$\beta_{\sigma(n-1)}$} at 450 735
	\pinlabel \textcolor{red}{$\alpha_n$} at 490 735
	\pinlabel \textcolor{blue}{$\beta_{\sigma(n)}$} at 550 735
	\pinlabel \textcolor{red}{$\alpha_1$} at 53 625
	\pinlabel \textcolor{blue}{$\beta_{1}$} at 108 625
	\pinlabel \textcolor{red}{$\alpha_2$} at 160 625
	\pinlabel \textcolor{blue}{$\beta_{\sigma(1)+1}$} at 220 625
	\pinlabel $\dots$ at 352 665
	\pinlabel \textcolor{red}{$\alpha_{3}$} at 275 625
	\pinlabel \textcolor{blue}{$\beta_{\sigma(2)+1}$} at 328 625
	\pinlabel \textcolor{red}{$\alpha_n$} at 377 625
	\pinlabel \textcolor{blue}{$\beta_{\sigma(n-1)+1}$} at 435 625
	\pinlabel \textcolor{red}{$\alpha_{n+1}$} at 500 625
	\pinlabel \textcolor{blue}{$\beta_{\sigma(n)+1}$} at 550 625
   \endlabellist
 \includegraphics[scale=.7]{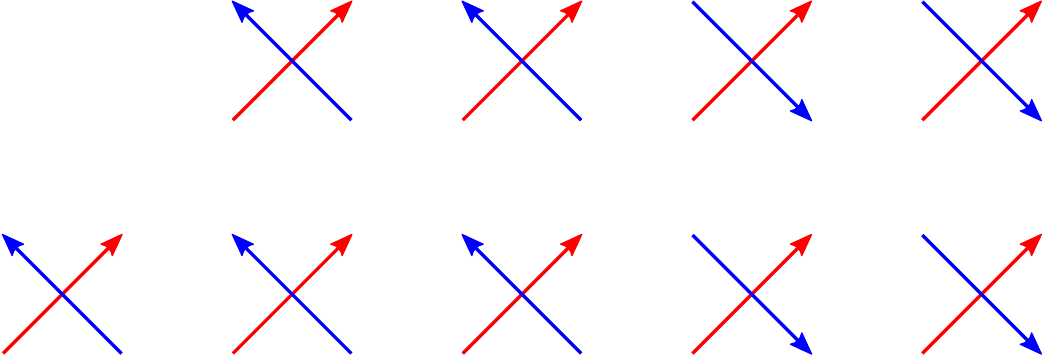} 
      \vskip .2 cm
       \caption{Top: A formal bordered generator $\x\in \GnA^{1, L}$. Bottom: The formal generator $f_1(\x)\in \gens_{n+1}$. (Here $n=4$.)}\label{fig:f_gen_d}
\end{figure}

 Next, we extend the maps $f_\s$ to internal flows. 
Note that if  $\phi$ is an internal flow from $\x = (\sigma_1, \e_1, \s_1)$ to $\y= (\sigma_2, \e_2, \s_2)$, then $\s_1 = \s_2$, so simply write $\s$ for $\s_1$ or $\s_2$. We define $f(\phi)$ as the formal flow from $f_{\s}(\x)$ to $f_{\s}(\y)$ that is obtained from $\phi$ graphically as follows. Relabel $\alpha_{\s}$ to $\alpha_2$, shift the remaining $\alpha$ and $\beta$ labels as above, and add a cross with arcs labeled $\alpha_{1}$ and $\beta_{1}$ and oriented so their intersection is positive. See for example Figure~\ref{fig:f_flow_d}.

\begin{figure}[h]
\centering
  \labellist
  	\pinlabel \textcolor{red}{$\alpha_1^a$} at 130 705
	\pinlabel \textcolor{blue}{$\beta_{\sigma(1)}$} at 193 705
	\pinlabel \textcolor{red}{$\alpha_2$} at 222 705
	\pinlabel \textcolor{blue}{$\beta_{\sigma(2)}$} at 279 705
	\pinlabel \textcolor{red}{$\alpha_{3}$} at 318 705
	\pinlabel \textcolor{blue}{$\beta_{\sigma(3)}$} at 373 705
	\pinlabel \textcolor{red}{$\alpha_4$} at 408 705
	\pinlabel \textcolor{blue}{$\beta_{\sigma(4)}$} at 455 705
	\pinlabel \textcolor{red}{$\alpha_1$} at 53 595
	\pinlabel \textcolor{blue}{$\beta_{1}$} at 108 595
	\pinlabel \textcolor{red}{$\alpha_2$} at 130 595
	\pinlabel \textcolor{blue}{$\beta_{\sigma(1)+1}$} at 190 595
	\pinlabel \textcolor{red}{$\alpha_{3}$} at 225 595
	\pinlabel \textcolor{blue}{$\beta_{\sigma(2)+1}$} at 279 595
	\pinlabel \textcolor{red}{$\alpha_4$} at 318 595
	\pinlabel \textcolor{blue}{$\beta_{\sigma(3)+1}$} at 373 595
	\pinlabel \textcolor{red}{$\alpha_{5}$} at 408 595
	\pinlabel \textcolor{blue}{$\beta_{\sigma(4)+1}$} at 455 595
   \endlabellist
 \includegraphics[scale=.7]{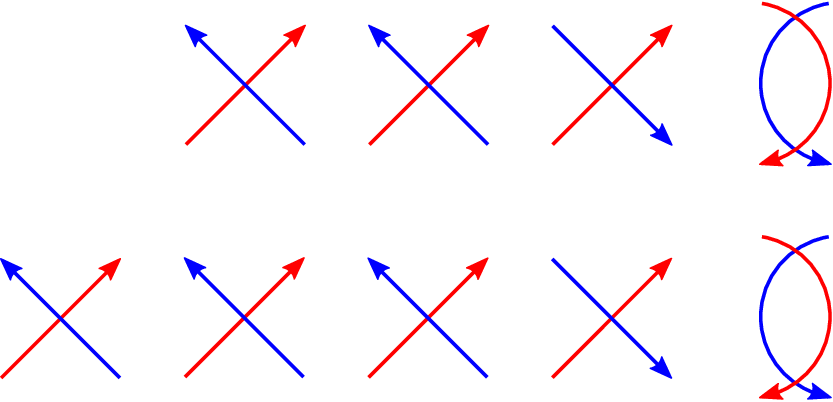} 
      \vskip .2 cm
       \caption{Top: A formal internal flow $\phi\in\FnA^L$. Bottom: The formal flow $f(\phi)\in \mathcal{F}_{n+1}$. (Here $n=4$.)}\label{fig:f_flow_d}
\end{figure}
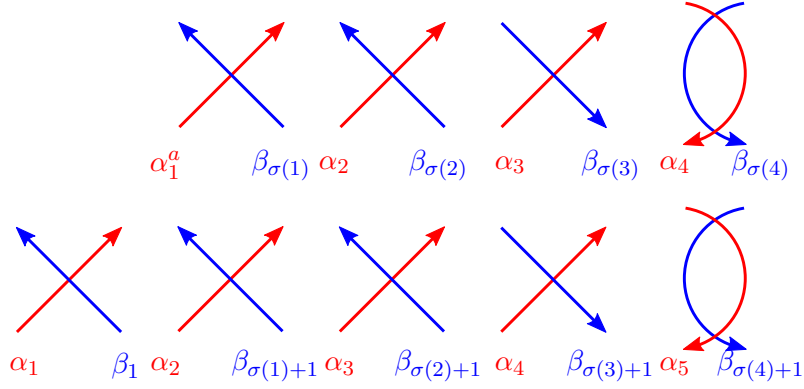

\begin{defn}\label{def:d-formula}
Let $\STD$ be the map $\FnA^L \to \{\pm 1\}$ defined as follows.
\begin{enumerate}
\item if $\phi: \x \to \y$ is an internal flow, then $\STD(\phi)=S(f(\phi))$, where $S$ is the sign assignment defined in \cite{hfz};
\item if $\psi: \x \to \y$ is a bordered flow, then $\STD(\psi)=\sign(\sigma_{\x})  \Pi_i \e_i(\x)$.
\end{enumerate}
\end{defn}

\begin{thm}
The function $\STD$ is a bordered sign assignment of type $D$.
\end{thm}

\begin{prf}
We first show that $\STD$ satisfies properties \ref{itm:s1d} through \ref{itm:s3d}. Suppose $(\phi_1,\phi_2)$ is an $\alpha$-type boundary degeneration. Then $(f(\phi_1),f(\phi_2))$ is an $\alpha$-degeneration in $\mathcal{F}_{n+1}$, so we have $\STD(\phi_1)\STD(\phi_2)=S(f(\phi_1))S(f(\phi_2))=1$. Thus \ref{itm:s1d} holds. Analogous arguments show that \ref{itm:s2d} and \ref{itm:s3d} hold as well.
  
 It remains to show that $\STD$ satisfies \ref{itm:d1}. Suppose the pairs $(\phi_1:\x \to \x',\phi_2:\x' \to \y' )$ and $(\phi_3:\x \to \y,\phi_4:\y \to \y')$ form a bordered square. Without loss of generality, assume $\phi_1$ and $\phi_4$ are bordered flows of type $\rho$. If $\e_1(\x)=\e_1(\x')$, then $gr(\rho)=0$. In this case $f(\phi_2)$ and $f(\phi_3)$ are identically the same flow, and thus $\STD(\phi_2)\STD(\phi_3)=1$. On the other hand, if $\e_1(\x)=-\e_1(\x')$, then $gr(\rho)=1$. In this scenario, it can be shown via arguments analogous to those made in the proof of Theorem~\ref{thm:sta_exist} that $S(f(\phi_2))S(f(\phi_3))=-1=\STD(\phi_2)\STD(\phi_3)$.  In both cases, we  have that $\STD(\phi_2)\STD(\phi_3)=(-1)^{gr(\rho)}$. Now, the internal flow $\phi_3:\x \to \y$ is either a bigon or a rectangle. If it is a bigon with moving coordinate $i$, then $\epsilon(\x)$ and $\epsilon(\y)$ differ only at the $i^{th}$ place, and $\sigma_{\x}=\sigma_{\y}$. If it is a rectangle with moving coordinate $i$ and $j$, then $\sign(\sigma_{\x})=-\sign(\sigma_{\y})$, and $\epsilon(\x)$ and $\epsilon(\y)$ either differ at both or neither of the moving coordinates. In either case, we have that $\sign(\sigma_{\x})  \Pi_i \e_i(\x)\sign(\sigma_{\y})  \Pi_i \e_i(\y)=-1$. Thus, we have that $\STD(\phi_1)\STD(\phi_4)=-1$, and therefore $\Pi \STD(\phi_i)=(-1)^{gr(\rho)+1}$. Thus, \ref{itm:d1} holds.
\end{prf}

\subsection{Equivalence classes of type $D$ sign assignments}
\label{ssec:d-unique}
Since we do not have a condition for type $D$ sign assignments similar to \ref{itm:a2}, there is no way to relate flows of different $\rho$ type. Furthermore, in a given bordered square each bordered flow will be of the same type and thus said type will always occur exactly two times. As such, for a given type $D$ bordered sign assignment $S^D$ and for each type $\rho_i$ of bordered flows, modifying the sign assignment by a multiplicative factor $c_i \in \{\pm 1\}$ applied to all flows of that type results in a function that still satisfies \ref{itm:s1d} through \ref{itm:d1}. With this in mind, let $c_i(S^D)$ for $i=1,2, 3, 123$ be the sign of the flow of type $\rho_i$ with starting generator $\x=\{\id, \mathbf{1}, \s\}$ (where $\s \in \{1,2\}$ as appropriate). Now we can classify bordered sign assignments of type $D$ according to a choice of internal sign assignment for each idempotent and a tuple of choices $(c_1(S^D),c_2(S^D),c_3(S^D),c_{123}(S^D)) \in \{\pm1\}^4$.  

\begin{prop}\label{prop:fourB}
There are four bordered sign assignments of type $D$ and power $n$ for any given signed matched pointed circle and any choice of $\mathbb{Z}/2$ grading, up to gauge equivalence.
\end{prop}

\begin{proof}
Since we can always fix a consistent choice of internal sign assignments as in the type $A$ case, we are left with sixteen possible classes of sign assignments according to our four choices $(c_1,c_2,c_3,c_{123})$. However, for a given flow $\phi_1: \x \to \y$ of type $\rho_1$, there is another flow $\phi_2: \x \to \y$ of type $\rho_3$. Because these two flows have the same starting and ending generators, any gauge equivalence must either change the sign of both or neither. Thus, the product $c_1c_3$ is invariant under gauge equivalence. We can build a similar pair of flows of type $\rho_2$ and $\rho_{123}$ to show that the product $c_2c_{123}$ is also invariant. The reason we can build such flows for these pairs and no other pair (such as $\rho_1$ and $\rho_2$) is because each pair agrees in terms of the relative parity of the sign profile of their starting and ending generators at the first coordinate, i.e. if $\epsilon_1(\x)\neq\epsilon_1(\y)$ differ for some flow $\phi:\x \to \y$ of type $\rho_1$, then $\epsilon_1(\x')\neq\epsilon_1(\y')$ for any other flow $\phi':\x' \to \y'$ of type $\rho_1$ or $\rho_3$. This is what allows us to build the pair of flows that share starting and ending generators. Then similar to type $A$ structures, we have exactly four equivalence classes that are classified by the pair $(c_1c_3, c_2c_{123}) \in \{\pm 1\}^2$. 
\end{proof}

We omit the sign assignment from the notation when it is clear from context or unnecessary, and will write either $c_i$ or $c_i^D$ depending on whether or not it is important to note that the sign assignment was of type $D$.

\subsection{The type $D$ structure $\cfdhat$.}
\label{ssec:d-def}
 \begin{defn}
 Let $\HD = (\Sigma, \alpha,\beta, z)$ be a nice Heegaard diagram for a bordered 3-manifold $Y$ with torus boundary. Fix an ordering and orientation on  the $\alpha$-circles, $\beta$-circles, and $\alpha$-arcs. Let $P$ be the sign sequence for which $\bdy \HD$ can be identified  with $\ZZ_P$ (so that each $\alpha$-arc is oriented from a $+$ to a $-$). The generators of the diagram $\HD$ naturally define formal left bordered generators of power $|\beta|$, and the empty bigons and rectangles on $\HD$ (counted by the structure map $\delta^1$ for $\cfdhat(\HD)$) specify formal left flows of power $|\beta|$. We further fix a choice $C \in \{+,-\}^2$, and let $\STD$ be a representative of the class $S_C$.

We define a left type $D$ structure $\cfdhat_{C,P}(\HD;\Z)$ over $\alg_P$ as follows. Let $\cfdhat_{C,P}(\HD;\Z)$ be the $\Z$-module freely generated by the generators $\mathcal G(\HD)$ of the diagram. We give $\cfdhat_{C,P}(\HD;\Z)$ the structure of a left $\mathcal I_P$-module by setting

\[\iota_t \cdot \x= \begin{cases}
\x \quad \textrm{ if } t = \overline{o(\x)} \\
0 \quad \textrm{otherwise.}
\end{cases}
\]

 We define a map 
 \[\delta^1\colon \cfdhat_{C,P}(\HD;\Z)\to \alg_P\otimes_{\mathcal I_P} \cfdhat_{C,P}(\HD;\Z)\]
  by counting the same regions on $\HD$ as for  $\cfdhat(\HD)$ (i.e.\ empty bigons and rectangles), but this time with sign. Namely, define
\begin{align*} 
\delta^1(\x)&= \sum_{ \y \in \mathbb{T}_\alpha \cap \mathbb{T}_\beta} \ \sum_{\phi \in \mathrm{Flows}_{\mathrm{int}}(\x,\y)} \ \STD(F(\phi)) \cdot \iota_\x \otimes \y \\
		&+ \sum_{ \y \in \mathbb{T}_\alpha \cap \mathbb{T}_\beta} \ \sum_{a \in \alg_P} \ \sum_{\phi \in \mathrm{Flows}_{a}(\x,\y)} \ \STD(F(\phi)) \cdot a \otimes \y 
  \end{align*}
 where $\mathrm{Flows}_{\mathrm{int}}(\x,\y)$ and $\mathrm{Flows}_{a}(\x,\y)$ are the sets of internal and bordered flows of type $a$, respectively,  between $\x$ and $\y$ on $\HD$, $F(\phi)$ is the formal flow in $\FnA^R$ that corresponds to $\phi$, and $\iota_\x$ is the idempotent occupied by $\x$. \end{defn}
 
 \begin{thm}
 The  map $\delta^1$ satisfies the type $D$ relation 
\[0=(\mu_2\otimes \id_{\cfdhat_{C,P}(\HD;\Z)}) \circ (\id_{\alg_P} \otimes \delta^1) \circ \delta^1 + (\mu_1 \otimes \id_{\cfdhat_{C,P}(\HD;\Z)}) \circ \delta^1. \]
 \end{thm}

 \begin{proof}
In the torus algebra, the differential operator $\mu_1$ is trivial on all elements, so the relation simplifies to $0=(\mu_2 \otimes \id_{\cfdhat_{C,P}(\HD;\Z)}) \circ (\id_{\alg_P} \otimes \delta^1) \circ \delta^1$. Now applying $\delta^1$ twice to the same generator $\x$ and tracking the outputs will result in a collection of internal and bordered squares. Then property \ref{itm:s3d} ensures that each internal square results in zero contribution, and property \ref{itm:d1} does the same for bordered squares. Since we explicitly constructed our sign assignment $\STD$ to satisfiy these properties, we are done.
\end{proof}

\begin{thm}\label{thm:indD}
Let $\HD$ be a bordered diagram for manifold with torus boundary. Fix a choice $C \in \{+,-\}^2$ and a sign sequence $P$. Then the structure $\cfdhat_{C,P}(\HD;\Z)$ does not depend on the ordering and orientation of the curves, nor on the choice of representative $\STD$ of the class $S_C$.
\end{thm}

\begin{proof}
The proof of this theorem can be done three pieces; fixed order and orientation with differing representative, fixed representative and ordering with differing orientation, and fixed representative and orientation with differing orders. Each piece can be proven using an identical argument from the corresponding section of the proof of Theorem~\ref{thm:indA}. Thus we are done.
\end{proof}

We are now ready to prove Theorem~\ref{thm:ind}.

\begin{proof}
Theorems~\ref{thm:indA} and \ref{thm:indD} separately prove the independence of type $A$ structures and type $D$ structures from their relevant choices. Thus we are done.
\end{proof}


\section{Pairing}
\label{sec:pairing}

In order to discuss pairing bordered sign assignments and the relation to closed ones, we find it convenient to introduce an additional set of objects to be thought of as formal flows. 
\begin{defn}\label{def:id-flow}
Given a bordered generator $\x$,  the \emph{formal identity flow $\phi_{\x}$} is defined graphically in the same way as $\x$.  
\end{defn}
We think of $\phi_{\x}$ as a flow from $\x$ to $\x$ which has no moving coordinates, and whose domain is $\x$. 
We extend the sets of flows between bordered generators by defining
\[\overline{\FnA^{\bullet}} \coloneqq \FnA^{\bullet}\cup \{\phi_{\x} \mid \x\in \GnA^{\bullet}\} \qquad \textrm{ for } \bullet\in \{L,R\}.\]
We also extend bordered sign assignments to sign functions on these larger sets of flows. Specifically, if $S: \FnA^{\bullet}\to \{\pm 1\}$ is a bordered sign assignment, we define  a new sign function 
\[\overline{S}: \overline{\FnA^{\bullet}}\to \{\pm 1\}\]
that maps all identity flows to $1$ and agrees with $S$ on all other flows. 

We also  introduce a couple more notions of compatibility that allow us to build the ``formal" analogue to gluing Heegaard diagrams. For generators, compatibility is defined as follows.

\begin{defn}\label{def:union-gen}
Let $P, P'\in \{-,+\}^{2}$. We say that a  pair of formal generators  $(\x, \y)\in \GmA^R \times  \GnAp^L$ is \emph{compatible} if $P=P'$ and $\s(\x) \neq \s(\y)$. 

Given compatible formal generators $\x_1\in \GmA^R$ and $\x_2\in \GnA^L$,  their \emph{union}   $\x_1 \cup \x_2 \in \mathcal{G}_{m+n}$ is the formal generator obtained diagrammatically as follows. Take the diagrammatic representations for $\x_1$ and $\x_2$. Replace each label $\alpha_i$ with $\alpha_{n+1+i}$ and each label $\beta_i$ with $\beta_{n+i}$ in the diagrammatic representation of $\x_2$, and replace each label $\alpha_i^a$ with $\alpha_{n+i-1}$ for $i=1,2$. Finally, take the union of the two resulting diagrams. 

We denote the set of pairs of compatible generators in $\GmA^R \times  \GnA^L$ by $\gmnpp$, or simply by $\gmn$ when the choice $P$ is clear from the context or unimportant. 
\end{defn}

For flows, compatibility is defined similarly.

\begin{defn}\label{def:union-flow}
Let $P, P'\in \{-,+\}^{2}$. 
We say that a pair of flows  $(\phi_1, \phi_2) \in \overline{\FmA^R} \times \overline{\FnAp^L}$ is \emph{compatible} if $P=P'$ and either
\begin{itemize}
\item[-] one flow is internal, and the other is an identity flow, or
\item[-] both flows are bordered and of the same type.
\end{itemize}

Given compatible formal flows $\phi_1\in \overline{\FmA^R}$ and $\phi_2\in \overline{\FnA^L}$,  their \emph{union} $\phi_1\cup \phi_2 \in \mathcal{F}_{m+n}$ is the formal flow obtained diagrammatically as follows. 
Take the diagrammatic representations for $\phi_1$ and $\phi_2$. Replace labels as for generators (see Definition~\ref{def:union-gen}). Take the union of the resulting diagrams, glued along the pair of unlabeled edges if there are such. See Figure~\ref{fig:paired_flow} for an example. 

\begin{figure}[h]
\centering
\labellist
	\pinlabel \textcolor{red}{$\alpha_1$} at 32 587
	\pinlabel \textcolor{red}{$\alpha_2$} at 92 587
	\pinlabel \textcolor{red}{$\alpha_3$} at 152 587
	\pinlabel \textcolor{red}{$\alpha_1^a$} at 235 587
	\pinlabel \textcolor{red}{$\alpha_2^a$} at 235 645
	\pinlabel \textcolor{red}{$\alpha_1^a$} at 333 587
	\pinlabel \textcolor{red}{$\alpha_2^a$} at 333 645
	\pinlabel \textcolor{red}{$\alpha_2$} at 398 587
	\pinlabel \textcolor{red}{$\alpha_3$} at 458 587
	\pinlabel \textcolor{red}{$\alpha_1$} at 60 510
	\pinlabel \textcolor{red}{$\alpha_2$} at 120 510
	\pinlabel \textcolor{red}{$\alpha_3$} at 180 510
	\pinlabel \textcolor{red}{$\alpha_4$} at 272 510
	\pinlabel \textcolor{red}{$\alpha_5$} at 272 566
	\pinlabel \textcolor{red}{$\alpha_6$} at 380 510
	\pinlabel \textcolor{red}{$\alpha_7$} at 440 510
	\pinlabel \textcolor{blue}{$\beta_3$} at 67 587
	\pinlabel \textcolor{blue}{$\beta_2$} at 127 587
	\pinlabel \textcolor{blue}{$\beta_4$} at 185 587
	\pinlabel \textcolor{blue}{$\beta_1$} at 212 615
	\pinlabel \textcolor{blue}{$\beta_1$} at 374 615
	\pinlabel \textcolor{blue}{$\beta_3$} at 432 587
	\pinlabel \textcolor{blue}{$\beta_2$} at 492 587
	\pinlabel \textcolor{blue}{$\beta_3$} at 95 510
	\pinlabel \textcolor{blue}{$\beta_2$} at 155 510
	\pinlabel \textcolor{blue}{$\beta_4$} at 212 510
	\pinlabel \textcolor{blue}{$\beta_1$} at 239 541
	\pinlabel \textcolor{blue}{$\beta_5$} at 345 541
	\pinlabel \textcolor{blue}{$\beta_7$} at 412 510
	\pinlabel \textcolor{blue}{$\beta_6$} at 473 510
	\pinlabel $\rho_3$ at 274 615
	\pinlabel $\rho_3$ at 311 615
\endlabellist
 \vspace{.2 cm}
 \includegraphics[scale=.9]{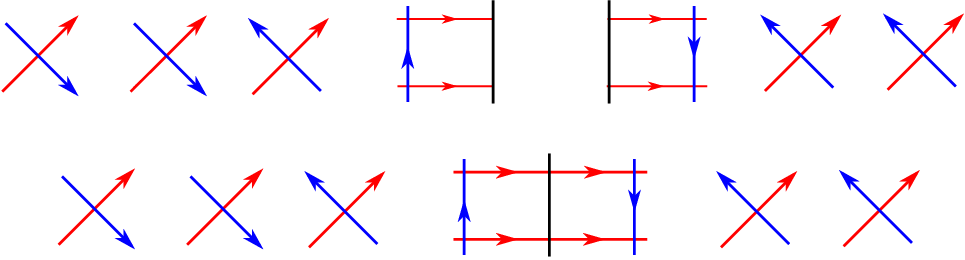}
 \vspace{.3 cm}
 	\caption{Top: A formal right bordered flow in $\overline{\FmA^R}$ and  a formal left bordered flow in $\overline{\FnA^L}$ which are compatible. Bottom: The union of the two bordered flows. Here $m=4$, $n=3$, and $P=(++)$.}\label{fig:paired_flow}
\end{figure}
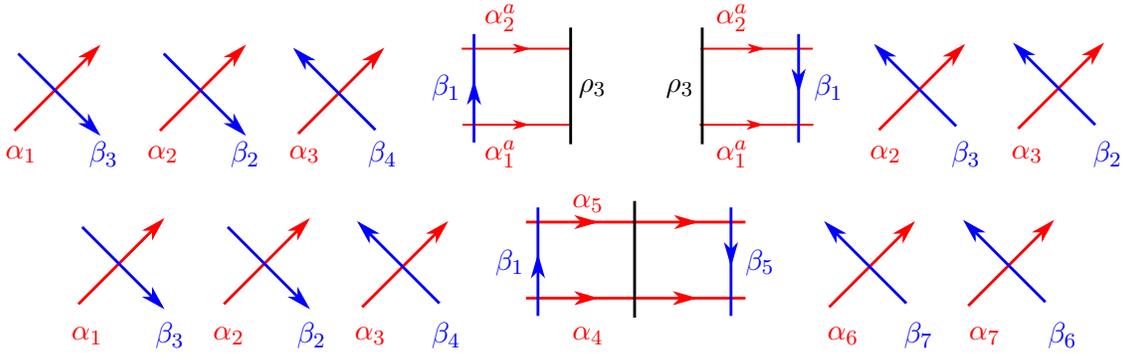

We denote the set of pairs of compatible flows in $\overline{\FmA^R} \times \overline{\FnA^L}$ by $\fmnpp$, or simply by $\fmn$ when the choice $P$ is clear from the context or unimportant. 
\end{defn}

Note that if $\phi_i$ is a flow from $\x_i$ to $\y_i$ for $i=1,2$, and $\phi_1$ and $\phi_2$ are compatible, then the starting generators are compatible too, and so are the ending generators. In that case,  $\phi_1\cup \phi_2$ is a flow from $\x_1 \cup \x_2$ to $\y_1\cup \y_2$.

Also note that the map $\gmnpp\to \mathcal{G}_{m+n}$ given by $(\x_1,\x_2)\mapsto \x_1\cup \x_2$ is an injection, so we'll sometimes view $\gmnpp$ as a subset of $\mathcal{G}_{m+n}$ and use the notation  $(\x_1,\x_2)$ and $\x_1\cup \x_2$ interchangeably. Similarly, we can view $\fmnpp$ as a subset of $\mathcal{F}_{m+n}$ and use the notation  $(\phi_1,\phi_2)$ and $\phi_1\cup \phi_2$ interchangeably.

\begin{defn}\label{def:compat}
Let $P, P'\in \{-,+\}^{2}$. 
A right type $A$ bordered sign assignment $\STA$ of power $m$ and a left type $D$ bordered sign assignment $\STDp$ of power $n$ are \emph{compatible} if $P=P'$ and the map 
\begin{align*}
\overline{\STA} \otimes \overline{\STD}  \colon  \fmnpp & \to \{\pm 1\}\\
(\phi_1, \phi_2)& \mapsto \overline{\STA}(\phi_1)  \overline{\STD}(\phi_2) \cdot (-1)^{|\x||\phi_2|}
\end{align*}
satisfies \ref{itm:s1}, \ref{itm:s2}, and \ref{itm:s3} on its domain. \footnote{The choice of $\otimes$ in the notation  is inspired by the fact that this sign function is exactly what one sees when tensoring a type $A$ and a type $D$ structure whose structure maps are defined using the two bordered sign assignments; see the proof of Theorem~\ref{thm:pair}.} Note that $|\phi_2|=0$ if $\phi_2$ is an identity flow, and  $|\phi_2|=1$ otherwise.
\end{defn}

We next show that given any  type $A$ sign assignment $\STA$, there exists a compatible type $D$ sign assignment $\STD$. In our discussion, we will sometimes use the following definition.

\begin{defn}
Let $(\phi_1,\phi_2)$ and $(\phi_3,\phi_4)$ be two pairs of right (resp.\ left) formal flows that form a bordered square such that $\phi_1$ and $\phi_4$ are the bordered flows. Let $\psi_1$ be a left (resp.\ right) bordered flow that is compatible with $\phi_1$ and $\phi_4$. Then the collection of closed flows obtained by taking the unions of $\phi_1$ and $\phi_4$ with $\psi_1$, and the unions of $\phi_2$ and $\phi_3$ with the appropriate identity flows forms a closed square. We refer to the set $\{(\phi_1,\phi_2),(\phi_3,\phi_4), \psi_1\}$ as a \emph{paired square} of type $A$ (resp.\ type $D$). 
\end{defn}

\begin{figure}[h]
\centering
\labellist
	\pinlabel $\x$ at 232 536
	\pinlabel $\y$ at 232 582
	\pinlabel $\gs$ at 180 536
	\pinlabel $\gt$ at 180 582
	\pinlabel $\x\y$ at 184 502
	\pinlabel $\gs\gt$ at 233 505
	\pinlabel $\ga$ at 322 582
	\pinlabel $\gb$ at 322 536
	\pinlabel $\x$ at 452 542
	\pinlabel $\y$ at 452 502
	\pinlabel $\gs$ at 514 542
	\pinlabel $\gt$ at 514 505
	\pinlabel $\x\y$ at 519 582
	\pinlabel $\gs\gt$ at 452 585
	\pinlabel $\ga$ at 377 505
	\pinlabel $\gb$ at 377 542
	\endlabellist
 \includegraphics[scale=.9]{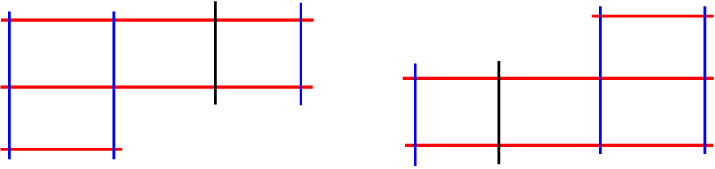}
 	\caption{Left: An example of a paired square of type $A$. Right: An example of a paired square of type $D$.}\label{fig:paired_sq}
\end{figure}
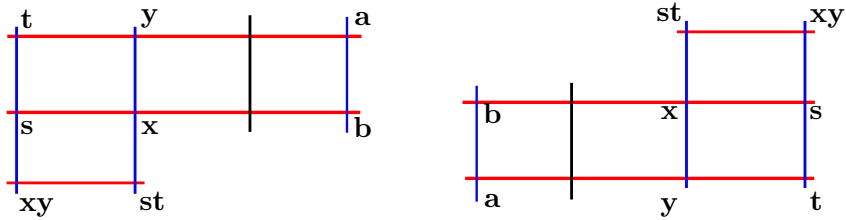

\begin{thm}\label{thm:compatexist}
Let $\STA$ be a given type $A$ bordered sign assignment of power $m$. Then for any $m>0$, there exists a type $D$ bordered sign assignment $\STD$ of power $n$ that is compatible with $\STA$.
\end{thm}

\begin{proof}
We begin by observing that by construction, any type $D$ bordered sign assignment $\STD$ (with the correct signed pointed matched circle) when paired with $\STA$ will result in a function $\overline{\STA}  \otimes \overline{\STD}$ that satisfies properties \ref{itm:s2} and \ref{itm:s3}. To see \ref{itm:s2}, first note that any disk-like $\beta$-degeneration is composed of two bigons. As such, it must result as a pair of unions of left internal flows with the appropriate right identity flow $\phi_1 \cup \psi_\ga$ and $\phi_2 \cup \psi_\ga$, or vice versa. In the first case, 
\[\left(\overline{\STA}  \otimes\overline{\STD}\right)(\phi_1 \cup \psi_\ga)\left(\overline{\STA}  \otimes \overline{\STD}\right)(\phi_1 \cup \psi_\ga)=\STA(\phi_1)\STA(\phi_2)=-1\]
 by property \ref{itm:s2a}. In the second case, we instead have 
 \[\left(\overline{\STA}  \otimes \overline{\STD}\right)(\psi_\ga \cup \phi_1)\left(\overline{\STA}  \otimes \overline{\STD}\right)(\psi_\ga \cup \phi_2)=(-1)^{2|\ga|}\STD(\phi_1)\STD(\phi_2)=-1\]
  by property \ref{itm:s2d}. In either case, property \ref{itm:s2} holds. 

Now let $\beta_p,\beta_q$ be the labels on the boundary components of an annular $\beta$-degeneration, assuming without loss of generality that $p<q$. If $p,q\leq n$, then we can construct it as the unions two left internal flows with a right identity flow and use the above argument. Similarly, if $p,q \geq n+1$ we can construct is as the unions two right internal flows with a left identity flow. It remains to deal with the case that $p \leq n$ and $q \geq n+1$. Then the annulus must be constructed as two closed rectangles, each comprised as the union of two bordered flows $\phi_1 \cup \phi_2$ of some type $\rho_i$ and $\phi_3 \cup \phi_4$ of some type $\rho_j$. See Figure~\ref{fig:annulus}. Because these are closed rectangles, there must be two different $\alpha$'s occupied, and since the rectangles are comprised of bordered flows they must specifically be $\alpha_m$ and $\alpha_{m+1}$. Thus one pair of bordered flows has starting generators that occupy $\alpha_1^a$ and the other must have starting generators that occupy $\alpha_2^a$. Since type $\rho_2$ flows are only the type with starting generators on $\alpha_2^a$, we may assume that $\phi_1$ and $\phi_2$ are of type $\rho_2$. However, if we orient the $\alpha$ arcs as according to the signed pointed matched circle, we immediately see that $\phi_3$ and $\phi_4$ cannot possibly be of any valid $\rho$ type. Therefore there are no such $\beta$-degenerations. Since this was the last remaining case, property \ref{itm:s2} holds.

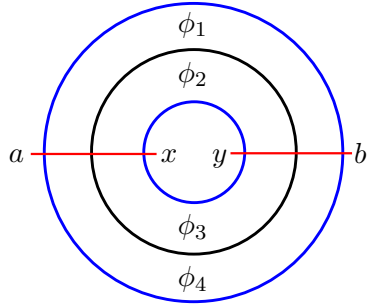
\begin{figure}[h]
\centering
\labellist
	\pinlabel $a$ at 142 624
	\pinlabel $b$ at 401 626
	\pinlabel $x$ at 257 624
	\pinlabel $y$ at 295 622
	\pinlabel $\phi_1$ at 275 723
	\pinlabel $\phi_2$ at 275 685
	\pinlabel $\phi_3$ at 275 570
	\pinlabel $\phi_4$ at 275 531
	\endlabellist
 \includegraphics[scale=.5]{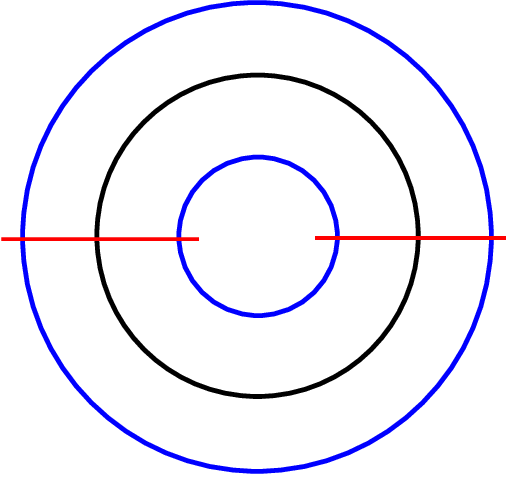}
 	\caption{An annular $\beta$-degeneration that involves both right and left bordered flows. Note that while we cannot draw these domains on a Heegaard diagram, they are nonetheless all valid flows and form such a degeneration on the formal level. See the proof of Theorem~\ref{thm:compatexist} for a discussion on labels and orientations.}\label{fig:annulus}
\end{figure}

For property \ref{itm:s3}, we note that there are five cases:
\begin{itemize}
\item[-] the unions of four left internal flows with four right identity flows,
\item[-] the unions of four right internal flows with four left identity flows,
\item[-] the union of a right internal flow with a left identity flow and the union of a left internal flow with a right identity flow, noting that such a pair of unions when performed in either order form a full square in the closed case,
\item[-] a paired square of type $A$,
\item[-] a paired square of type $D$.
\end{itemize}
 {\bf{Case 1.}} Let $\phi_1 \cup \psi_\ga, \phi_2 \cup \psi_\ga, \phi_3 \cup \psi_\ga,$ and $\phi_4 \cup \psi_\ga$ be the four unions. Then
 \begin{align*}
\prod_{i=1}^4\left(\overline{\STA}  \otimes \overline{\STD}\right)(\phi_i \cup \psi_\ga) &= \prod_{i=1}^4 \STA(\phi_i) \\
&= -1,
 \end{align*}
 where the first equality comes from the definition of $\overline{\STA} \otimes \overline{\STD}$, and the second from property \ref{itm:s3a}. Thus property \ref{itm:s3} holds for $\overline{\STA} \otimes \overline{\STD}$. \\
 {\bf{Case 2.}} Let $\psi_\ga \cup \phi_1, \psi_\ga \cup \phi_2, \psi_\ga \cup \phi_3,$ and $\psi_\ga \cup \phi_4$ be the four unions. Then
 \begin{align*}
\prod_{i=1}^4\left(\overline{\STA}  \otimes \overline{\STD}\right)(\psi_\ga \cup \phi_i) &=(-1)^{2|\ga|}\prod_{i=1}^4 \STD(\phi_i) \\
&= -1,
 \end{align*}
 where the first equality comes from the definition of $\overline{\STA} \otimes\overline{\STD}$, and the second from property \ref{itm:s3d}. Thus property \ref{itm:s3} holds for $\overline{\STA} \otimes \overline{\STD}$. \\
 {\bf{Case 3.}} Suppose we have $\phi_1:\w \to \x$ as the left internal flow, and $\phi_2: \y \to \z$ as the right internal flow. Then the full square would give us 
\begin{align*}
\left(\overline{\STA}  \otimes \overline{\STD}\right)(\phi_1 \cup \phi_\y)  \left(\overline{\STA}  \otimes \overline{\STD}\right)(\phi_\x \cup \phi_2) &   \left(\overline{\STA} \otimes \overline{\STD}\right)(\phi_\w \cup \phi_2) \\
\cdot \left(\overline{\STA}  \otimes \overline{\STD}\right)(\phi_1 \cup \phi_\z) &=\STA(\phi_1)^2\STD(\phi_2)^2(-1)^{|\w|+|\x|}.
\end{align*}
Since there is an internal flow from $\w$ to $\x$, we know that $(-1)^{|\w|+|\x|}=-1$, and thus property \ref{itm:s3} holds. \\
{\bf{Case 4.}} Let $\phi_1:\x \to \y$ and $\phi_4: \gs \to \gt$ be the left bordered flows, $\phi_2:\y \to \gt$ and $\phi_3:\x \to \gs$ be the left internal flows, and $\psi:\ga \to \gb$ be the right bordered flow. See the left side of Figure~\ref{fig:paired_sq} for a concrete example of such a paired square. Then we have
\begin{align*}
\left(\overline{\STA} \otimes \overline{\STD}\right)(\phi_1 \cup \psi_1)  \left(\overline{\STA} \otimes \overline{\STD}\right)(\phi_2 \cup \phi_\gb) & \left(\overline{\STA} \otimes \overline{\STD}\right)(\phi_3 \cup \phi_\ga) \\
\cdot \left(\overline{\STA} \otimes \overline{\STD}\right)(\phi_4 \cup \psi_1)&=(-1)^{|\x| + |\s|} \STD(\psi_1)^2 \prod_{i=1}^4 \STA(\phi_i)\\
&=(-1)^{|\x|+|\gs|} \\
&=-1.
\end{align*}
The first equality comes from the definition of $\overline{\STA} \otimes \overline{\STD}$, and the second from property \ref{itm:a1}. For the last equality, since there is an internal flow from $\x$ to $\s$, their grading differs by one and thus $(-1)^{|\x| + |\s|}=-1$. Then \ref{itm:s3} holds for paired squares of type $A$. \\
{\bf{Case 5.}} Let $\phi_1:\x \to \y$ and $\phi_4: \gs \to \gt$ be the right bordered flows of type $\rho$, $\phi_2:\y \to \gt$ and $\phi_3:\x \to \gs$ be the right internal flows, and $\psi:\ga \to \gb$ be the left bordered flow. See the right side of FIgure~\ref{fig:paired_sq} for a concrete example of such a paired square. Then we have
\begin{align*}
\left(\overline{\STA} \otimes \overline{\STD}\right)(\psi_1 \cup \phi_1)\left(\overline{\STA} \otimes \overline{\STD}\right)(\phi_\gb \cup \phi_2) & \left(\overline{\STA} \otimes \overline{\STD}\right)(\phi_\ga \cup \phi_3) \\
\cdot \left(\overline{\STA} \otimes \overline{\STD}\right)(\psi_1 \cup \phi_4) &=(-1)^{|\ga| + |\gb|} \STA(\psi_1)^2 \Pi_{i=1}^4 \STD(\phi_i)\\
&=(-1)^{|\ga| + |\gb|}\cdot(-1)^{|\rho|+1}\\
&=(-1)^{2|\rho|+1}\\
&=-1
\end{align*}
The first equality comes from the definition of $\overline{\STA} \otimes \overline{\STD}$, and the second from property \ref{itm:d1}.  For the third equality, we know that $\ga$ and $\gb$ are connected by a bordered flow of type $\rho$, so we must have that $|\ga| + |\gb|=|\rho|$. Therefore \ref{itm:s3} also holds for paired squares of type $D$.

By the above, the only place where an issue can arise in finding a compatible $\STD$ is property \ref{itm:s1}. As with $\beta$-degenerations, any disk-like $\alpha$-degeneration must be comprised of two closed flows such that each is the union of an internal flows and an identity flow, and thus covered by property \ref{itm:s1a} or \ref{itm:s1d}. That leaves annular $\alpha$-degenerations. 
First, any annular degeneration where $\alpha_p,\alpha_q$ are such that both $p,q \leq n$ or $p,q \geq n+1$ must be comprised of internal flows just as the disk-like degenerations were, and as such would be covered by property \ref{itm:s1a} or \ref{itm:s1d} once again. Thus let us assume $p \leq n$ and $q \geq n+1$. As in the $\beta$-annulus case, we must have two closed rectangles, each comprised of the union of a left and right bordered flow. Since these are closed rectangles, we must have two distinct $\alpha$'s occupied, and since they are unions of bordered flows these $\alpha$'s must be $\alpha_m$ and $\alpha_{m+1}$. Then one pair of bordered flows must have starting generators who occupy $\alpha_1^a$ and ending generators who occupy $\alpha_2^a$, and the other pair of flows will have the opposite. Since flows of type $\rho_2$ are the only bordered flows with starting generators on $\alpha_2^a$, it must be one of the two types of bordered flows on the annulus. By orienting the $\alpha$-arcs (see Figure~\ref{fig:annulus}) according to the choice of signed pointed matched circle for $\STA$, we see that the other flow type must $\rho_{123}$. Furthermore, as discussed in Propositions \ref{prop:fourA} and \ref{prop:fourB}, one can always predict how $\STA$ and $\STD$ behave under changes in the permutation and sign profiles of the starting generator. Thus we may assume without loss of generality that $p=n$, $q=n+1$ and both are oriented such that their intersection with $\alpha_1^a$ is positive (with the convention, as always, that the $\alpha$ comes first). 

\begin{figure}[h]
\centering
\labellist
	\pinlabel $\x_1$ at 152 558
	\pinlabel $\y_1$ at 250 558
	\pinlabel $\x_2$ at 357 558
	\pinlabel $\y_2$ at 450 558
	\pinlabel $\phi_1$ at 240 470
	\pinlabel $\phi_2$ at 240 620
	\pinlabel $\phi_3$ at 365 470
	\pinlabel $\phi_4$ at 365 620
	\endlabellist
 \includegraphics[scale=.5]{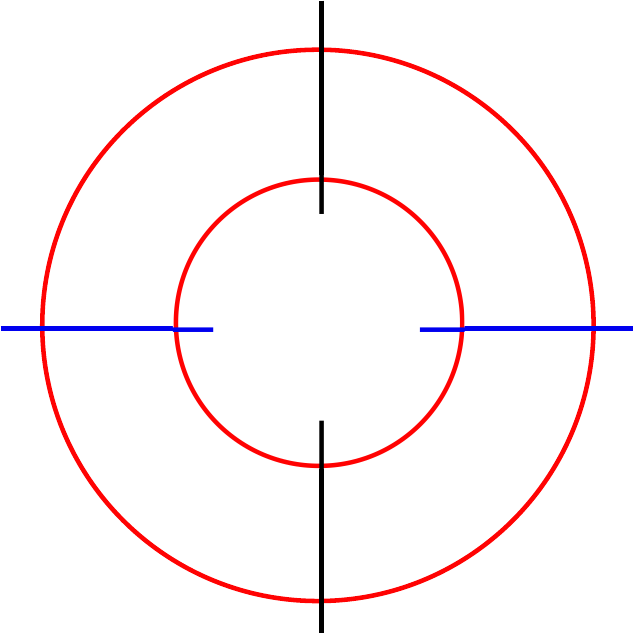}
 	\caption{An annular $\alpha$-degeneration that involves both right and left bordered flows. Note that while we cannot draw these domains on a Heegaard diagram, they are nonetheless all valid flows and form such a degeneration on the formal level. See the proof of Theorem~\ref{thm:compatexist} for a discussion on labels and orientations.}\label{fig:annulus}
\end{figure}
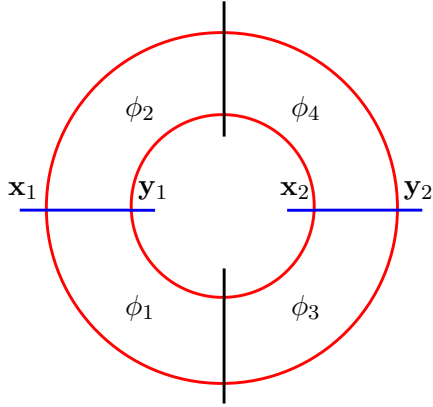

Then we have four flows; a right bordered flow $\phi_1: \x_1 \to \y_1$ of type $\rho_{123}$, a right bordered flow $\phi_2: \y_1 \to \x_1$ of type $\rho_2$, a left bordered flow $\phi_3: \x_2 \to \y_2$ of type $\rho_{123}$, and a left bordered flow $\phi_4: \y_2 \to \x_2$ of type $\rho_2$. Again, as noted above, since the effects of changes in sign profile and permutation are well understood, we may further assume without loss of generality that all generators have identity permutations and identity sign profiles away from the annulus. Recalling our classification of bordered sign assignments, this implies that $\STA(\phi_2)=c_2^A,\STD(\phi_3)=c_2^D,$ and $\STD(\phi_4)=c_{123}^D$.  

However, $\STA(\phi_1)$ is a bit more involved. Its value can be derived as the product of signs of three flows $\psi_1, \psi_2,$ and $\psi_3$, that are of type $\rho_1, \rho_2$ and $\rho_3$ respectively, using a series of two bordered triangles. The sign of $\psi_1$ is $c_1^A$ by construction, but $\psi_3$ differs from the flow used to define $c_3^A$ (which we will call $\psi_3^c$) in the orientation of $\beta_m$, and $\psi_2$ may also differ from the flow used to define $c_2^A$ (which we will call $\psi_2^c$) depending on the choice of signed pointed matched circle. However, $\psi_3$ can be related to $\psi_3^c$ by a bordered square, meaning the sign of $\psi_3$ will differ from $c_3^A$ by a factor dependent on the internal sign assignment. Similarly, if $\psi_2$ and $\psi_2^c$ differ, they can be related by a bordered square and the sign $\psi_2$ will differ from $c_2^A$ by another factor dependent on the internal sign assignment. Since the internal sign assignment is fixed, we will simply call the of product of these two factors $c_{\mathrm{int}}^A$ and remark that it is a known quantity in the following equations. Given all of this, we can now say that $\STA(\phi_1)=c_1^Ac_2^Ac_3^Ac_{\mathrm{int}}^A$.

We need the product of the signs of these flows once paired to be 
\begin{align*}
\left(\overline{\STA} \otimes \overline{\STD}\right)(\phi_1 \cup \phi_3)\left(\overline{\STA} \otimes \overline{\STD}\right)(\phi_2 \cup \phi_4) &= (-1)^{\mid\x_1\mid + \mid\y_1\mid}c_2^A(c_1^Ac_2^Ac_3^Ac_{\mathrm{int}}^A)c_2^Dc_{123}^D\\ &=1.
\end{align*}
We know that $|\x_1| + |\y_1|=|\rho_2|=|\rho_{123}|$, which is in turn determined by the choice of signed matched pointed circle, and the value of each $c_i^A$ was fixed at the beginning in our choice of class for $\STA$. 
In order for compatibility to hold true, we must have that $c_2^Dc_{123}^D=(-1)^{\mid\rho_2\mid}c_1^Ac_3^Ac_{\mathrm{int}}^A$. Recalling that a type $D$ sign assignment is given by the tuple $(c_1^Dc_3^D,c_2^Dc_{123}^D)$, we see that there clearly must be a sign assignment that satisfies this equality.
\end{proof}

We next show that a pair of compatible bordered sign assignments induces a unique closed sign assignment.

\begin{thm}\label{thm:ext}
Suppose a right type $A$ bordered sign assignment $\STA$ of power $m$ and a left type $D$ bordered sign assignment $\STD$ of power $n$ are compatible. Then the sign function ${\STA\otimes \STD} \colon \fmnpp \to \{\pm 1\}$ can be uniquely extended to a (closed) sign assignment of power $m+n$.
\end{thm}

Outline of the proof of Theorem~\ref{thm:ext}:
\begin{enumerate}
\item We begin by considering a certain subset $\plthin$ of $\fmnpp$. We show that there is a unique (up to gauge equivalence) map $S' \colon \plthin \to \{\pm 1\}$ satisfying the properties  from Definition~\ref{def:s} on its domain.  Thus, $S'$ must be the restriction of  $\overline{\STA}\otimes \overline{\STD}\colon \fmnpp\to \{\pm 1\}$ to $\plthin$. In a series of steps, we then show that there is a unique extension of $S'$ to a function on all formal flows $\fmnp$ in $\fmnpp$ whose starting and ending generators have sign profile $\bf{1}$.
\item We show that there is a unique sign assignment on all bigons in $\fmnpp$.
\item We show that a sign assignment on all bigons and all flows in $\fmnp$ extends uniquely to a sign assignment on $\fmnpp$.
\item Given a sign assignment $S$ on $\mathcal{F}_{m+n}$, its restriction to $\fmnpp$ is a sign assignment. By the above, it must agree with $\overline{\STA}\otimes \overline{\STD}$. By uniqueness of closed sign assignments, we've shown that $\overline{\STA}\otimes \overline{\STD}$ extends to a sign assignment of power $m+n$.
\end{enumerate}

We will first assume that $P = (++)$. At the end of Section~\ref{ssec:pair-epsilon-any}, we argue that we get the same results for the other three sign sequences.

 \subsection{Restricting to  generators in $\gmn$ with the identity sign profile}\label{ssec:pair-epsilon-1}
 
Some of the proofs in this section generalize arguments from \cite[Section 4]{most}, so we review some relevant terminology here.  

Let $G_{\ell}$ be an $\ell \times \ell$ toroidal grid diagram drawn in the plane as in \cite{most}, with  horizontal curves oriented left to right and labeled $\alpha_1, \ldots, \alpha_\ell$ from bottom to top, and vertical curves oriented upward and labeled $\beta_1, \ldots, \beta_\ell$ from left to right. See Figure~\ref{fig:grid-mn}. 
We will think of the generators of $G_\ell$ as the formal generators in $\mathcal G_\ell$ with sign profile ${\bf 1}$. We will think of the set $\mathrm{Rect}_\ell$ of rectangles on $G_{\ell}$ as a set of formal rectangles in $\mathcal F_{\ell}$. 
Let  $\mathrm{Rect}^{\circ}_\ell$ be the set of empty rectangles on $G_{\ell}$. We will say a rectangle on $G_\ell$ is \emph{planar}, if it does not intersect the interior of the rightmost vertical annulus, or the interior of the topmost horizontal annulus. As in \cite{most}, we say a rectangle is \emph{thin} if has width one (a thin rectangle may have any height). Following the notation from \cite{most}, we denote the sets of thin rectangles and planar thin rectangles by $\mathrm{tRect}$ and $\mathrm{tRect}^{\ast}$, respectively (as the index $\ell$ is usually clear, it is not included in the notation).
We will denote by  $\Gamma_\ell$ the Cayley graph with vertices $V(\Gamma_\ell)$ the generators of $G_{\ell}$ and edges $E(\Gamma_\ell)$ corresponding to the planar thin rectangles on $G_\ell$ (that is, there is an edge between $\x$ and $\y$ if and only if there is a planar thin rectangle from $\x$ to $\y$ or one from $\y$ to $\x$). This was denoted simply $\Gamma$ in \cite{most}.

Let 
\[\mathcal G_{m,n}' =\{ \x\cup \y \in \mathcal G_{m+n} \mid  (\x, \y)\in \GmA^R \times  \GnA^L, \e(\x) = {\bf 1}, \e(\y) = {\bf 1}\}.\] With the above identification in mind, observe that $\mathcal G_{m,n}'$  is exactly the generators of $G_{m+n}$ supported in $\mathrm{BL}_m \cup \mathrm{TR}_n$, where $\mathrm{BL}_m$ is the bottom-left subgrid formed by $\alpha_1, \ldots, \alpha_{m+1}$ and $\beta_1, \ldots, \beta_m$, and $\mathrm{TR}_n$ is the top-right subgrid formed by $\alpha_m, \ldots, \alpha_{m+n}$ and $\beta_{m+1}, \ldots, \beta_{m+n}$. Let $\fmnp\subset \fmn\subset \mathcal F_{m+n}$ be the set of rectangles on $G_{m+n}$ with starting and ending generator both in $\mathcal G_{m,n}'$. Note that this set is precisely the set of all flows with starting and ending generators in $\mathcal G_{m,n}'$.
Let $\fempty \coloneqq \fmnp\cap \mathrm{Rect}^{\circ}_{m+n}$ be the set of empty rectangles on $G_{m+n}$ with starting and ending generator both in $\mathcal G_{m,n}'$. 
Let $\thin\subset \mathcal F_{m+n}$ and  $\plthin\subset \mathcal F_{m+n}$ be the sets of thin and planar thin rectangles, respectively, on $G_{m+n}$ with starting and ending generator both in $\mathcal G_{m,n}'$. Observe that $\plthin$ consists of the thin rectangles whose domain is fully supported in $\mathrm{BL}_m$ or $\mathrm{TR}_n$, together with the thin rectangles whose domain is the $1\times 1$ square $D_m$ given by intersection of the annulus cobounded by $\beta_m$ and $\beta_{m+1}$ and the annulus cobounded by $\alpha_{m}$ and $\alpha_{m+1}$.
See Figure~\ref{fig:grid-mn} for an illustration of $\mathrm{BL}_m$, $\mathrm{TR}_n$, and $D_m$.

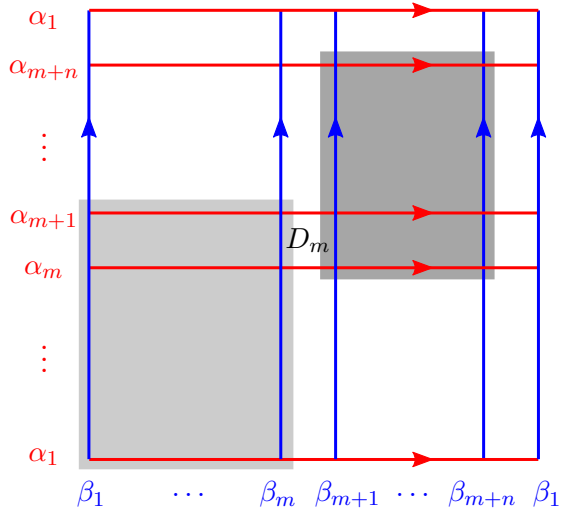
\begin{figure}[h]
\centering
  \labellist
  	\pinlabel \textcolor{red}{$\alpha_1$} at -14 5
	\pinlabel $\rotatebox{90}{\textcolor{red}{\dots}}$ at -14 45
	\pinlabel \textcolor{red}{$\alpha_m$} at -14 80
	\pinlabel \textcolor{red}{$\alpha_{m+1}$} at -14 100
	\pinlabel $\rotatebox{90}{\textcolor{red}{\dots}}$ at -14 130
	\pinlabel \textcolor{red}{$\alpha_{m+n}$} at -14 160
	\pinlabel \textcolor{red}{$\alpha_1$} at -14 180
	\pinlabel $D_m$ at 92 92
	\pinlabel \textcolor{blue}{$\beta_1$} at 5 -10
	\pinlabel \textcolor{blue}{$\dots$} at 45 -10
	\pinlabel \textcolor{blue}{$\beta_m$} at 80 -10
	\pinlabel \textcolor{blue}{$\beta_{m+1}$} at 108 -10
	\pinlabel \textcolor{blue}{$\dots$} at 135 -10
	\pinlabel \textcolor{blue}{$\beta_{m+n}$} at 162 -10
	\pinlabel \textcolor{blue}{$\beta_1$} at 188 -10
   \endlabellist
 \includegraphics[scale=0.95]{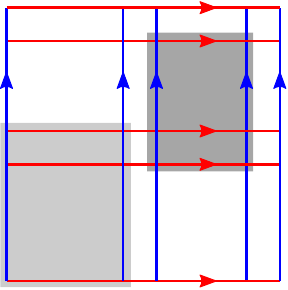}
 \vspace{.2cm}
 	\caption{The toroidal grid diagram $G_{m+n}$. The subgrids $\mathrm{BL}_m$ and $\mathrm{TR}_n$ are shaded in light and dark grey, respectively. The annulus cobounded by $\beta_m$ and $\beta_{m+1}$ and the annulus cobounded by $\alpha_{m}$ and $\alpha_{m+1}$ intersect in the square $D_m$.}\label{fig:grid-mn}
\end{figure}

\begin{prop}\label{prop:planar-thin}
There is a unique (up to gauge equivalence) map $S' \colon \plthin\to \{\pm 1\}$ satisfying  \ref{itm:s3}  from Definition~\ref{def:s}. 
\end{prop}
\begin{proof}
By  \cite[Proposition 4.7]{most}, there is a function  $S$, unique up to a gauge transformation,  from $\mathrm{tRect}^{\ast}_{m+n}$ to $\{\pm 1\}$ that satisfies the following two properties
\begin{itemize}
\item[\mylabel{itm:sq}{(Sq)}]  If the edges in $E(\Gamma_{m+n})$ corresponding to rectangles $r_1, \ldots, r_4$ form a square, then 
\[S(r_1)\cdots S(r_4) = -1.\]
\item[\mylabel{itm:hex}{(Hex)}] If the edges in $E(\Gamma_{m+n})$ corresponding to rectangles $r_1, \ldots, r_6$ form a hexagon, then \[S(r_1)\cdots S(r_6) = 1.\]
\end{itemize}
 Since $\plthin$ is a subset of $\mathrm{tRect}^{\ast}_{m+n}$, it follows that there is a function from $\plthin$ to $\{\pm 1\}$ that satisfies \ref{itm:sq} and \ref{itm:hex}. (We currently make no claim of uniqueness.) 
By \cite[Lemma 4.6]{most}, any sign assignment on $\mathrm{Rect}^{\circ}_{m+n}$ restricts to such a function on  $\mathrm{tRect}^{\ast}_{m+n}$. The argument in the proof of \cite[Lemma 4.6]{most}  implies that this is also true if we restrict our flows to only a subset of the generators of $G_{m+n}$. In particular, given a sign assignment on $\fmnp$, the fact that it satisfies \ref{itm:s1} implies that it satisfies \ref{itm:sq}, whereas given any hexagon, two applications of \ref{itm:s1} imply \ref{itm:hex}. Therefore,  $\overline{\STA}\cdot \overline{\STD} \vert_{\plthin}$ satisfies \ref{itm:sq} and \ref{itm:hex}.

We now show uniqueness, i.e.\ that any other  function from $\plthin$ to $\{\pm 1\}$ that satisfies \ref{itm:sq} and \ref{itm:hex} is gauge equivalent to  $\overline{\STA}\cdot \overline{\STD} \vert_{\plthin}$.

Consider the Cayley graph $\Gamma_{m+n}$ with vertices $V(\Gamma_{m+n})$ the generators of $G_{m+n}$ and edges $E(\Gamma_{m+n})$ corresponding to planar thin rectangles. Let $W_{m+n}$ be the $2$-complex obtained by attaching $2$-cells to $\Gamma_{m+n}$ corresponding to squares and hexagons; this was denoted $W$ in \cite[Proposition 4.7]{most}.

 Let $\Gamma_{m,n}$ be the induced subgraph of $\Gamma_{m+n}$ on $V(\Gamma_{m,n}) = \mathcal G_{m+n}'$. Observe that $E(\Gamma_{m,n}) = \mathrm{t}{\mathcal F_{m+n}^{\ast}}'$. Attach $2$-cells to $\Gamma_{m,n}$ corresponding to squares and hexagons, as in  \cite[Proposition 4.7]{most}, to obtain a $2$-complex $W_{m,n}$. Note that $W_{m,n}$ is simply the induced subcomplex of $W_{m+n}$ with $0$-skeleton $V(\Gamma_{m,n})$. By the argument at the end of the proof of  \cite[Proposition 4.7]{most}, if we can show that $W_{m,n}$ is connected and simply connected, we are done.

Partition the set $\mathcal G_{m,n}' $ as $\mathcal G_{m,n}^1\sqcup \mathcal G_{m,n}^2$ where $\mathcal G_{m,n}^i =\{ \x\cup \y \in \mathcal G_{m,n}' \mid   \s(\x) = i\}$. On the grid $G_{m+n}$, we can think of  $\mathcal G_{m,n}^1$ (resp.\ $\mathcal G_{m,n}^2$) as the generators supported in $\mathrm{BL}_m \cup \mathrm{TR}_n$ whose $\alpha_m$ component is on some $\beta_j$ with $j\leq m$ (resp.\ $j\geq m+1$). Let $\Gamma^i$ be the induced subgraph of $\Gamma_{m,n}$ with vertex set $\mathcal G_{m,n}^i$. Similarly, let $W^i$ be the induced subcomplex of $W_{m,n}$ with $0$-skeleton $\mathcal G_{m,n}^i$. Clearly $\Gamma^1$ is isomorphic to $\Gamma_m$, and $\Gamma^2$ to $\Gamma_n$. Thus, by \cite[Proposition 4.7]{most}, $W^1$ is connected and simply-connected, and so is $W^2$.

Let $\x\in \mathcal G_{m,n}'$ be the generator $\{\alpha_i\cap \beta_i \mid 1\leq i\leq m+n\}\subset G_{m+n}$. The planar thin rectangle starting at $\x$ and with domain the  $1\times 1$ square $D_m$ defined earlier demonstrates that $W_{m,n}$ is connected. To see this, we note that it corresponds to an edge connecting $W^1$ and $W^2$, which are connected subcomplexes whose $0$-skeleta together comprise the $0$-skeleton for $W_{m,n}$. 

It remains to show that $W_{m,n}$ is simply connected. Suppose $C$ is a cycle in the $1$-skeleton of $W_{m,n}$. If $C$ is fully contained in $W^1$ or in $W^2$, we've already shown it's contractible. So assume $C$ has some vertices in $W^1$ and some in $W^2$.

 It suffices to show  that a walk $v_0, v_1, \ldots, v_\ell$ with $v_0 = v_\ell = \x\in W^1$ is homotopic rel boundary to a walk in $W^1$. Assume the walk contains a vertex in $W^2$, otherwise we're done. We can think  of vertices as permutations in $S_{m+n}$, by identifying a vertex $v = \{\alpha_i\cap \beta_{\sigma(i)} \mid 1\leq i\leq m+n\}$ with the permutation $\sigma$. For example, $v_0 = \id_{S_{m+n}}$. 
Similarly, we can then think of edges as multiplication to the left by a transposition as follows. If $e$ is an edge between $v$ and $w$ corresponding to a planar thin rectangle (from $v$ to $w$, or from $w$ to $v$) in the annulus cobounded by $\beta_k$ and $\beta_{k+1}$, then $v = \tau_k w$, and equivalently $w = \tau_k v$, where $\tau_k$ is the transposition $(k, k+1)$. Then we can write the sequence of edges from $v_0$ to $v_l$ in the walk as the sequence  $\tau_{i_1}, \tau_{i_2}, \ldots , \tau_{i_\ell}$. Observe that we have $v_k = \tau_{i_1}\tau_{i_2}\cdots \tau_{i_k}$ for  $1\leq k\leq \ell$, and since $v_0 = v_\ell = \id_{S_{m+n}}$, we also have $\tau_{i_1}\tau_{i_2}\cdots \tau_{i_\ell} = \id_{S_{m+n}}$. 
 
 Let $s>0$ be the smallest index such that $v_s\in W^2$. This exists by the assumption that the walk contains a vertex in $W^2$. Let $t$ be the smallest index such that $t>s$ and $v_t\in W^1$. This exists since the walk ends back in $W^1$.  
 Since $v_{s-1}$ is in $W^1$ and $v_s =\tau_m v_{s-1}$ is in $W_2$, it must be that the edge between $v_{s-1}$ and $v_s$ corresponds to a planar thin rectangle with domain $D_m$. So $v_{s-1}(m) = m$ and $v_{s-1}(m+1) = m+1$, which implies $v_s(m)= m+1$ and $v_s(m+1) = m$. Analogously, $v_t(m) = m$ and $v_t(m+1) = m+1$. By definition of $s$ and $t$, the sequence $\tau_{i_{s+1}}, \tau_{i_{s+2}}, \ldots, \tau_{i_{t-1}}$ does not contain the transposition $\tau_m$. Since distant transpositions commute (i.e.\ $\tau_i\tau_j = \tau_j\tau_i$ whenever $|i-j|>1$), we can rewrite the product  $\sigma = \tau_{i_{s+1}} \tau_{i_{s+2}} \cdots \tau_{i_{t-1}}$ as $\sigma_1\sigma_2$, where $\sigma_1$ is a product of transpositions in $\{\tau_1, \ldots, \tau_{m-1}\}$, and $\sigma_2$ is a product of transpositions in $\{\tau_{m+1}, \ldots, \tau_{m+n}\}$. This corresponds to a homotopy in $W_{m,n}$ that uses $2$-cells corresponding to squares of disjoint planar thin rectangles. Further, since $v_t = \tau_m\sigma v_s$, we have $\sigma  = \tau_mv_tv_s^{-1}$, so $\sigma(m) = m$ and $\sigma(m+1) = m+1$. It is easy to see that then $\sigma_1(m) = m$. Then we can rewrite $\sigma_1$ as a product $\sigma'$ of transpositions  $\{\tau_1, \ldots, \tau_{m-2}\}$ using only the relations  $\tau_i\tau_j = \tau_j\tau_i$ whenever $|i-j|>1$, and $\tau_i\tau_{i-1}\tau_i = \tau_{i-1}\tau_i\tau_{i-1}$, for $i,j\leq m$. These relations correspond to squares and hexagons, so this corresponds to a homotopy from the path starting at $v_s$ and representing the sequence of transpositions in the product $\sigma_1$ to some other path starting at $v_s$ and only using  transpositions in $\{\tau_1, \ldots, \tau_{m-2}\}$. Analogously, the path starting at $\sigma_1v_s$ and representing the sequence of transpositions in the product $\sigma_2$ is homotopic to a path that only uses transpositions in $\{\tau_{m+2}, \ldots, \tau_{m+n}\}$. Composing all these homotopies, we have a homotopy from the path $v_{s-1}, v_s, \ldots, v_{t_1}, v_t$ to a path whose sequence of edges starts and ends with $\tau_m$, and never uses transpositions adjacent to $\tau_m$. Since distant transpositions commute, we can move the terminal $\tau_m$ to the front, yielding a homotopy in $W_{m,n}$ to a path from $v_{s-1}$ to $v_t$ whose first two edges are $\tau_m$, and whose remaining edges do not include $\tau_m$. The relation $\tau_m^2 = 1$ yields a homotopy to a path from $v_{s-1}$ to $v_t$ that does not contain an edge labeled $\tau_m$. We have found a homotopy from the initial finite walk to finite walk with strictly fewer edges labeled $\tau_m$. Repeating this process finitely many times, we obtain a homotopy to a walk which starts and ends at $\id_{S_{m+n}}$ and has no edges labeled $\tau_m$. This new closed walk is then entirely in $W^1$, and is thus contractible. Thus, $W_{m,n}$ is simply connected. 
\end{proof}

\begin{prop}\label{prop:all-thin}
Given  any function $S_0: \plthin \to \{\pm 1\}$  satisfying  \ref{itm:s3}, there is a unique extension to a function $S:\plempty\to \{\pm 1\}$ which satisfies   \ref{itm:s3}. 
\end{prop}

\begin{proof}
By \cite{most}, we know there exists a sign assignment on  $\mathrm{Rect}^{\circ}_{m+n}$, so by restricting, we know there exists a sign assignment on  $\plempty$. It remains to prove uniqueness. 

Suppose $S_1$ and $S_2$ are two sign assignments on $\plempty$. By Proposition~\ref{prop:planar-thin}, the restrictions $S_1\vert _{\plthin}$ and $S_2\vert_{\plthin}$ are related by a gauge transformation. Let  $f\colon \mathcal G_{m,n}' \to \{\pm 1\}$  be such a gauge transformation, i.e.\ a function such that for any rectangle $r\colon \x\to \y$ in $\plthin$ we have $S_2(r) = f(\x)f(\y) S_1(r)$. Define a new function 
\[\mathcal B \colon \plempty \to \{\pm 1\}\]
by 
\[\mathcal B(r)  = f(\x)f(\y)S_1(r)S_2(r),\]
where $\x$ and $\y$ are the starting and ending generators for $r$, respectively; c.f.\ \cite[Proof of Theorem 4.2]{most}. 
It is easy to deduce from \ref{itm:s1}--\ref{itm:s3} and Proposition~\ref{prop:planar-thin}  that $\mathcal B$ satisfies the following properties on its domain: 

\begin{enumerate}
\item[\mylabel{itm:bpt}{(B-1)}] $\mathcal B (r) = 1$ for any planar thin rectangle $r\in \plthin$.
\item[\mylabel{itm:bsq}{(B-2)}] $\mathcal B(r_1)\cdot \mathcal B(r_2) = \mathcal B(r_3)\cdot \mathcal B(r_4)$ whenever $(r_1, r_2)$ and $(r_3, r_4)$ form a square.
\end{enumerate}

We will show that $\mathcal B$ is one on all rectangles in $\plempty$.

We use induction on the width of the rectangle. 
 Suppose $r: \x\to \y$   is a planar empty rectangle.
 If $r$ is thin, we have already shown that $\mathcal B(r)=1$. Suppose $r$ has width $k>1$, and  assume that $\mathcal B$ is one on all planar rectangles of width $k-1$.

Say the two moving coordinates of $r$ are on $\beta_i$ and $\beta_{i+k}$, and consider the vertical strip in $G_{m+n}$ whose left and right edges are $\beta_i$ and $\beta_{i+k}$, respectively. Since $r$ is empty and each $\beta$-curve contains a point in $\x$, the complement of the domain of $r$ in this vertical strip contains $k-1>0$ points of $\x$. 

If any of these points are below the domain of $r$, pick the highest such point $p$. If $p$ is on $\beta_j$ for $j  > m$, then  
we get the following pairs of empty planar rectangles $(r, r_1)$ and $(r_2, r_3)$ that form a square, with all four rectangles in $\plempty$. Let $x_1$ and $y_1$ be the $\beta_i$-coordinates of $\x$ and $\y$, respectively. Similarly,  let $x_2$ and $y_2$ be the $\beta_{i+k}$-coordinates of $\x$ and $\y$. Then $r_1$ is the rectangle starting at $\y$ with moving coordinates $p$ and $y_2$, and $r_2$ and $r_3$ are uniquely defined by the property that $r_2$ starts at $\x$ and the juxtaposition $r_2\ast r_3$ has the same domain as $r\ast r_1$. See Figure~\ref{fig:thick-sq}.
\begin{figure}[h]
\centering
\labellist
	\pinlabel $x_1$ at -5 28
	\pinlabel $y_1$ at -5 82
	\pinlabel $y_2$ at 110 28
	\pinlabel $x_2$ at 110 82
	\pinlabel $p$ at 45 -4
	\endlabellist
 \includegraphics[scale=.9]{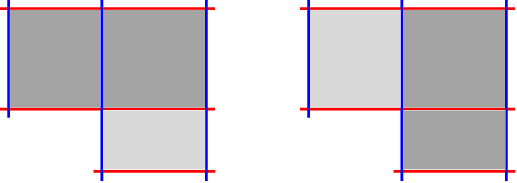}
 	\caption{Left: The juxtaposition $r\ast r_1$; the domain for $r$ is dark grey, and the domain for $r_1$ is light grey. Right: The juxtaposition $r_2\ast r_3$ whose domain is the same as for $r\ast r_1$.}\label{fig:thick-sq}
\end{figure}
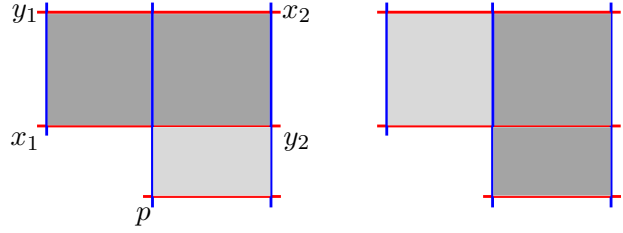
 By assumption, since $r_1$, $r_2$, and $r_3$ are all empty and all have width less than $k$, we have that $\mathcal B(r_1) = \mathcal B(r_2) = \mathcal B(r_3)$, so by \ref{itm:bsq} we have $\mathcal B(r)=1$. The case when $p$ is on $\beta_j$ for $j  \leq m$ is analogous. 
 
If all points of $\x$ in the interior of  the vertical strip cobounded by $\beta_i$ and $\beta_{i+k}$ are above the domain of $r$, then pick the lowest such point $p$. The argument in this case is similar.

Since $\mathcal B (r) = 1$ for any planar empty rectangle  $r$, it follows that $S_1$ and $S_2$ are gauge equivalent. 
 \end{proof}

Next, we show that  a sign assignment  $S:\plempty\to \{\pm 1\}$ can be uniquely extended to all rectangles in $\fmnp$, including nonempty and/or nonplanar ones. We show this in two steps: first for nonempty planar ones, then for nonplanar ones. Recall that by \cite[Proposition 4.4]{hfz}, there is, up to gauge equivalence, a unique sign assignment on $\mathrm{Rect}_l$ for any $l$. 

The arguments used to prove \cite[Proposition 4.4]{hfz} work verbatim for the subset $\fmnp\subset \mathrm{Rect}_{m+n}$. We briefly explain this below.
 
 In \cite{hfz}, the \emph{complexity} $K(r)$ of a rectangle $r\colon \x\to \y$ was defined to be the number of points $p\in \x$ supported in the interior of the domain of $r$. Given a sign assignment $S$ on planar empty rectangles in $\fmnp$, we extend it inductively to a sign assignment on all rectangles in $\fmnp$, as follows. Suppose $r\in\fmnp$ is a planar rectangle with $K(r)>0$. Pick a point $p$ of $\x$ in the interior of $r$, and consider the (unique) juxtaposition of rectangles $r_1\ast r_2\ast r_3 = r$ such that $p$ is a moving coordinate of $r_1$ and $r_3$ and a non-moving coordinate on the right edge of $r_2$, as in Figure~\ref{fig:complexity}.
 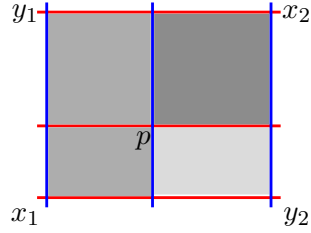
\begin{figure}[h]
\centering
\labellist
	\pinlabel $x_1$ at -5 -4
	\pinlabel $y_1$ at -5 82
	\pinlabel $y_2$ at 110 -4
	\pinlabel $x_2$ at 110 82
	\pinlabel $p$ at 45 28
	\endlabellist
 \includegraphics[scale=.9]{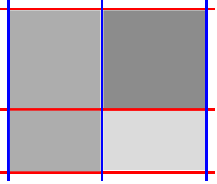}
 	\caption{A rectangle  $r\colon \x\to \y$ of complexity $K(r)>1$; the points $x_1$ and $x_2$ are the moving coordinates of $\x$, the points $y_1$ and $y_2$ are the moving coordinates of $\y$, and $p$ is in $\x\cap \y$. The rectangle $r$ can be decomposed into three rectangles $r_1$, $r_2$, and $r_3$ (in dark, medium, and light grey, respectively), each of lower complexity than $r$.}\label{fig:complexity}
\end{figure}
Define $S(r) = S(r_1)S(r_2)S(r_3)$. 

\begin{lem}\label{lem:complexity}
The sign $S(r)$ defined above is independent of the choice of decomposition of $r$. Furthermore, this function $S$ satisfies  \ref{itm:s3} on the set of planar rectangles in $\fmnp$.
\end{lem}
\begin{proof}
The proof follows from \cite{hfz} verbatim, after the observation that all arguments from \cite{hfz} applied to  planar rectangles in $\fmnp$ only use planar rectangles in $\fmnp$.
\end{proof}
   
Now suppose $r\colon \x \to \y$ is rectangle in $\fmnp$ that is not planar. If $r$ intersects the interior of the topmost horizontal annulus of the toroidal grid, but is disjoint from the interior of the rightmost vertical annulus. Then there is unique rectangle $r'\colon \y\to \x$ which is planar and together with $r$ forms a $\beta$-degeneration. The sign for $r'$ has already been defined, and we define $S(r)= -S(r')$. Similarly, if $r$  intersects the interior of the rightmost vertical  annulus, but not the interior of the topmost horizontal annulus, then there is unique rectangle $r'\colon \y\to \x$ which is planar and together with $r$ forms an $\alpha$-degeneration. Define $S(r)= S(r')$. Last, if  $r$ intersects the interiors of both the topmost horizontal annulus and the rightmost vertical annulus, then there is a unique other rectangle $r'\colon \x \to \y$ which is planar and related to $r$ by a sequence of a $\beta$-degeneration and an $\alpha$-degeneration. Define $S(r)= -S(r')$.

It is clear that the extension of the sign function $S$ above is well defined, since the planar rectangles used in the definition were unique.

\begin{lem}\label{lem:nonplanar}
The sign function $S \colon \fmnp\to \{\pm 1\}$ defined above  satisfies \ref{itm:s1}, \ref{itm:s2}, and  \ref{itm:s3}.
\end{lem}
\begin{proof}
Again the proof follows from \cite{hfz} verbatim, after the observation that all arguments from \cite{hfz} applied to rectangles in $\fmnp$ only use rectangles in $\fmnp$.
\end{proof}

Thus, we have shown that the function $S$ defined above is a sign assignment. Further, it is unique up to gauge equivalence. 

\begin{prop}\label{prop:fmnp}
There is a unique (up to gauge equivalence) sign assignment $S \colon \fmnp\to \{\pm 1\}$.
\end{prop}
\begin{proof}
Propositions~\ref{prop:planar-thin} and \ref{prop:all-thin} show that there is a unique (up to gauge equivalence) sign assignment on the flows in $\plempty$. The extension we define to $\fmnp$ is a sign assignment by  Lemmas~\ref{lem:complexity} and \ref{lem:nonplanar}. This extension was uniquely determined by the axioms of sign assignment, which completes the proof.
\end{proof}

\subsection{Restricting to bigons in $\fmn$}\label{ssec:pair-bigons}

 Observe that $\gmn$ is the subset of $\mathcal G_{m+n}$ consisting of generators with any sign profile and with permutation $\sigma$ such that $\sigma(i)\leq m$ when $i\leq m+1$ and $\sigma(i)>m$ when $i\geq m$. By \cite[Proposition 4.3]{hfz}, we know that for a fixed permutation, there is, up to gauge equivalence, a unique sign assignment over the set of generators with that permutation. Thus, there is, up to gauge equivalence, a unique sign assignment on the set of bigons over $\gmn$, i.e.\ on the set of bigons in $\fmn$.

\subsection{Extending to all flows in  $\fmn$}\label{ssec:pair-epsilon-any}

Note that $\alpha$- and $\beta$-degenerations consist of either two bigons or two rectangles. Further, a square of flows that contains a bigon has at least one generator with non-identity sign profile, so, restricting to $\fmnp$ and bigons, a square consists either only of bigons, or only of rectangles in $\fmnp$. Thus, a sign assignment on $\fmnp$ and a sign assignment on bigons collectively form a sign assignment on the union of the two sets. In other words, we have shown the existence and uniqueness of a sign assignment $S$ for flows in $\fmn$ that connect a pair of generators with  sign profile ${\bf 1}$, or a pair of generators with the same permutation.

  It remains to show such a  sign assignment $S$ can be extended to rectangles in $\fmn$ whose starting and ending generators do not both have identity sign profile.

  By \cite[Lemma 4.10]{hfz}, a sign assignment defined over all bigons and a fixed rectangle $r$ connecting two generators with sign profile ${\bf 1}$ can be uniquely extended to the set $\theta(r)$ consisting of all rectangles related to $r$ by a simple flip. 
  
  \begin{lem}\label{lem:theta}
   If $r$ is a rectangle in $\fmn$, and $r'$ is obtained from $r$ by a sequence of simple flips, then $r'$ is also in $\fmn$.
  \end{lem}
  
\begin{proof}
First observe that simple flips do not change the permutation of the two generators, so the starting and ending generator of $r'$ are both in $\gmn$. Further, simple flips do not change the labels and orientations of the domain. This implies that if $r$ is the union of an internal flow and an identity flow, then so is $r'$, and if  $r$ is the union of two bordered flows of a given type $\rho_i$, then $r'$ is the union of two bordered flows of the same type. 
\end{proof}
 
Further,    by \cite[Lemma 4.14]{hfz}, a sign assignment $S$ defined over all bigons and a fixed rectangle $r$ can be uniquely extended to a function on the set $\omega(r)$ consisting of all rectangles obtained from $r$ by changing orientations of (any number of) the edges of $r$  so that this extension satisfies \ref{itm:s3}. Moreover, the discussion in the proof of  \cite[Lemma 4.13]{hfz} shows that if $r'$ is obtained from $r$ by a single change of orientation of an edge, then the value $S(r)$ determines $S(r')$.

  \begin{lem}\label{lem:omega}
   If $r$ is a rectangle in $\fmn$, and $r'$ is a rectangle in  $\omega(r)\cap \fmn$, then $r'$ can be obtained from $r$ by a sequence of changes of orientations on  edges of $r$  so that the rectangle resulting from each change is still in  $\omega(r)\cap \fmn$. 
  \end{lem}
  
\begin{proof}
First observe that changing orientations of edges does not change the the permutation of the two generators, so the starting and ending generator of any rectangle in $\omega(r)$ are both in $\gmn$. Next, it is clear that if $r$ is the union of an internal flow and an identity flow, then each rectangle in $\omega(r)$ is in $\fmn$.  Last, consider the case when $r$ is the union of two bordered flows.

 If $r$ is the union of two bordered flows of type $\rho_2$, then  $\omega(r)\cap \fmn$ consists of four rectangles (each of which is also the union of two bordered flows of type $\rho_2$), obtained from $r$ by taking the four possible pairs of orientations of the $\beta$-edges; it is easy to see that changing either or both of the orientations of the $\alpha$-edges results in a rectangle in $\mathcal F_{m+n}$ that is not the union of a left flow and a right flow.

 If $r$ is the union of two bordered flows of type $\rho_1$, $\rho_3$, or $\rho_{123}$, then  $\omega(r)\cap \fmn$ consists of twelve rectangles -- for each possible  orientation of the $\beta$-edges, three of the four possible orientations of the $\alpha$-edges result in  rectangles $r_1$, $r_3$, and $r_{123}$  in $\omega(r)\cap \fmn$ such that each $r_i$ is  the union of two bordered flows of type  $\rho_i$, whereas the fourth results in a rectangle in $\omega(r)\setminus \fmn$ (i.e.\ one that is not the union of two bordered flows). The rectangle $r_{123}$ can be obtained from   $r_1$ or from $r_3$ by changing the orientation of $\alpha_{m+1}$ or $\alpha_m$, respectively, and thus from any other rectangle in $\omega(r)\cap \fmn$ by also changing orientations on the $\beta$-edges.
\end{proof}

 By the above observations,  we have that \cite[Lemma 4.15]{hfz} has the following analogue for our restricted sets of flows.

  \begin{lem}\label{lem:theta-omega}
  Let $S$ be a sign assignment defined over all bigons in $\fmn$ and over some fixed  rectangle $r$ in $\fmn$. Then $S$ can be uniquely extended to a function over the set   $\cup \{\omega(r_1)\cap \fmn \mid r_1\in \theta(r)\}$ so that the extension satisfies \ref{itm:s3}.
  \end{lem}

\begin{proof}
This is immediate from  Lemmas~\ref{lem:theta} and \ref{lem:omega}, combined with the arguments in \cite[Lemmas 4.10 and 4.14]{hfz}.
\end{proof}

Now we are ready to show there is a unique sign assignment on $\fmn$. Start with a sign assignment $S$ on all rectangles in $\fmnp$ and all bigons in $\fmn$. Use Lemma~\ref{lem:theta-omega} repeatedly for every rectangle in $\fmnp$, to extend $S$ to a sign function on $\fmn$. Note again that this is the unique extension that satisfies \ref{itm:s3} on each set  $\{\cup \{\omega(r_1)\cap \fmn \mid r_1\in \theta(r)\}\mid r\in \fmnp\}$.  By \cite[Lemmas 4.16 and 4.17]{hfz}, this extension is indeed a sign assignment. 

The argument when $P\neq (++)$ is analogous. We simply change the orientation of $\alpha_m$ and/or $\alpha_{m+1}$ on $G_{m+n}$ so that 
\[(\alpha_m\cap \beta_m, \alpha_{m+1}\cap \beta_m) = P\] and observe that the operation of changing sign profile in a fixed set of entries gives a bijection on generators which results in a bijection on bigons and on rectangles. Thus, all the arguments in this section apply, after composing with this bijection.

\subsection{Compatible bordered sign assignments pair to closed sign assignments}\label{ssec:pair-all}

We are finally ready to prove that bordered sign assignments on formal flows pair nicely. 

\begin{proof}[Proof of Theorem~\ref{thm:ext}]
Given a sign assignment $S$ on $\mathcal{F}_{m+n}$, its restriction to $\fmnpp$ is a sign assignment. By the above, it must agree with $\overline{\STA}\otimes \overline{\STD}$. By uniqueness of closed sign assignments, we've shown that $\overline{\STA}\otimes \overline{\STD}$ extends to a sign assignment of power $m+n$.
\end{proof}

We are finally ready to prove Theorem~\ref{thm:pair}.

\begin{proof}[Proof of Theorem~\ref{thm:pair}]
For each of  $\HD_1 =(\Sigma_1,\alphas_1,\betas_1,z)$ and $\HD_2=(\Sigma_2,\alphas_2,\betas_2,z)$, fix an ordering and orientation on  the $\alpha$-circles, $\beta$-circles, and $\alpha$-arcs, so that $\bdy\HD_1 = -\bdy\HD_2$ can be identified with the same signed pointed matched circle  $\ZZ_P$. Let $\STA$ be a bordered sign assignment of type $A$, compatible with $P$, and of power $m = |\betas_1|$. Let $\STD$ be a bordered sign assignment of type $D$ and of power $n = |\betas_2|$ compatible with $\STA$. 

Omitting the above choices from the notation, let  $\cfahat(\HD_1; \Z)$ and $\cfdhat(\HD_2; \Z)$ be the resulting type $A$ and type $D$ structure, respectively. Writing $F(\x)$ for the formal generator corresponding to a Heegaard diagram generator $\x$, we see that the differential on $\cfahat(\HD_1; \Z)\boxtimes \cfdhat(\HD_2; \Z)$ is then
\begin{align*}
\bdy^{\boxtimes}(\x_1\otimes \x_2) =& \sum_{ \y_1 \in \mathbb{T}_{\alpha_1} \cap \mathbb{T}_{\beta_1}} \ \sum_{\phi \in \mathrm{Flows}_{\mathrm{int}}(\x_1,\y_1)} \ \STA(F(\phi))  \y_1 \otimes \x_2\\
& + (-1)^{|\x_1|}\sum_{\substack{ \y_1 \in \mathbb{T}_{\alpha_1} \cap \mathbb{T}_{\beta_1} \\ \y_2 \in \mathbb{T}_{\alpha_2} \cap \mathbb{T}_{\beta_2}}} \ \sum_{a \in \alg_P}  \ \sum_{\substack{\phi_1 \in \mathrm{Flows}_{a}(\x_1,\y_1) \\ \phi_2 \in \mathrm{Flows}_{a}(\x_2,\y_2)}} \ \STA(F(\phi_1))   \y_1\otimes \STD(F(\phi_2)) \y_2 \\
& + (-1)^{|\x_1|} \sum_{ \y_2 \in \mathbb{T}_{\alpha_2} \cap \mathbb{T}_{\beta_2}} \ \sum_{\phi \in \mathrm{Flows}_{\mathrm{int}}(\x_2,\y_2)} \ \STD(F(\phi))  \x_1 \otimes \y_2 \\
=& \sum_{ \y_1 \in \mathbb{T}_{\alpha_1} \cap \mathbb{T}_{\beta_1}} \ \sum_{\phi \in \mathrm{Flows}_{\mathrm{int}}(\x_1,\y_1)} \ (\overline{\STA}\otimes \overline{\STD})(F(\phi)\cup\phi_{F(\x_2)})  \y_1 \otimes \x_2\\
& + \sum_{\substack{ \y_1 \in \mathbb{T}_{\alpha_1} \cap \mathbb{T}_{\beta_1} \\ \y_2 \in \mathbb{T}_{\alpha_2} \cap \mathbb{T}_{\beta_2}}} \ \sum_{a \in \alg_P}  \ \sum_{\substack{\phi_1 \in \mathrm{Flows}_{a}(\x_1,\y_1) \\ \phi_2 \in \mathrm{Flows}_{a}(\x_2,\y_2)}} \ (\overline{\STA}\otimes \overline{\STD})(F(\phi_1)\cup F(\phi_2))   \y_1\otimes  \y_2 \\
& +  \sum_{ \y_2 \in \mathbb{T}_{\alpha_2} \cap \mathbb{T}_{\beta_2}} \ \sum_{\phi \in \mathrm{Flows}_{\mathrm{int}}(\x_2,\y_2)} \ (\overline{\STA}\otimes \overline{\STD})(\phi_{F(\x_1)}\cup F(\phi))  \x_1 \otimes \y_2. 
\end{align*}

Consider the ordering and orientation on the curves on $\HD$ induced by that on the two bordered diagrams, as follows. Curves on $\HD_1$ come before curves on $\HD_2$. The circles obtained by gluing $\alpha$-arcs come after the $\alpha$-circles on $\HD_1$ and before the $\alpha$-circles on $\HD_2$, with the circle corresponding to $\iota_1$ coming first. 
The generators of the diagram $\HD$ naturally define formal generators of power $m+n$, and the empty bigons and rectangles on $\HD$ specify formal flows of power $m+n$. 
 With this identification and with Theorem~\ref{thm:ext} in mind, we now see that 
$\bdy^{\boxtimes}$ is precisely the differential on $\widetilde{\mathrm{CF}}(\HD; \Z)$ with respect to the sign assignment whose restriction to $\fmnpp$ is given by  $\overline{\STA}\otimes \overline{\STD}$.
\end{proof}


\section{A surgery exact triangle discussion}
\label{sec:tri}

Recall the following surgery exact triangle in Heegaard Floer homology \cite{osz14}. Given a framed knot $K$ in a 3-manifold $Y$, consider the 3-manifolds $Y_{-1}$, $Y_0$, and $Y_{\infty}$ obtained by $-1$, $0$, and $\infty$ surgery on $K$. There is an exact triangle 
\[\cdots \to \hfhat(Y_{\infty}; \Z)\to \hfhat(Y_{-1}; \Z)\to \hfhat(Y_0; \Z)\to \cdots\]
Since it is unknown whether the combinatorially defined Heegaard Floer homology $\hft$ over $\Z$ agrees with the analytically defined theory, the surgery exact triangle from \cite{osz14} does not automatically imply a surgery exact triangle for the integral version of 
$\hft$.

In order to discuss a bordered approach to proving Conjecture~\ref{conj:tri}, we need to lay out a couple more algebraic definitions, following \cite[Section 12]{osz-bord2}.

\begin{defn}
Let $\mathcal M$ and $\mathcal N$ be type $D$ structures over $\alg$ with structure maps $\delta^1_{\mathcal M}$ and $\delta^1_{\mathcal N}$, respectively. A \emph{morphism} from $\mathcal M$ to $\mathcal N$ is a map
\[f\colon M \to A \otimes N.\]
The \emph{differential} of a morphism $f$ is the morphism 
\[df\coloneqq (\mu_1 \otimes \id_N) \circ f + (\mu_2 \otimes \id_N) \circ (\id_A \otimes \delta^1_{\mathcal N}) \circ  f - (\mu_2 \otimes \id_N) \circ (\id_A \otimes  f) \circ  \delta^1_{\mathcal M}.\]
A \emph{homomorphism} is a map $f\colon M\to A\otimes N$ with $df = 0$.

The \emph{identity morphism} $\id_{M}\colon N\to A\otimes N$ is the map defined by $\id_M(x) = I_A\otimes x$, where $I_A$ is the unit of the DGA $\alg$.

Given morphisms $f$ from $M$ to $N$ and $g$ from $N$ to $P$, their \emph{composition}  is the morphism
\[g\circ f \coloneqq (\mu_2 \otimes \id_P) \circ (\id_A \otimes g)\circ f. \]
\end{defn}

Given a morphism $f$ from $\mathcal M$ to $\mathcal N$, we write $\delta^1_{\mathcal N}\circ f$ for $(\mu_2 \otimes \id_N) \circ (\id_A \otimes \delta^1_{\mathcal N}) \circ  f$ and $f\circ \delta^1_{\mathcal M}$ for $(\mu_2 \otimes \id_N) \circ (\id_A \otimes  f) \circ  \delta^1_{\mathcal M}$. For an algebra with no differential then, $f$ is a homomorphism exactly when $\delta^1_{\mathcal N}\circ f = f\circ \delta^1_{\mathcal M}$.

\begin{defn}
Two homomorphisms 
\[f, g\colon  M \to A \otimes N.\]
are \emph{homotopic}, denoted $f\sim g$, if there exists a morphism $h\colon M\to N[1]$ such that 
\[dh = f-g.\]

Two type $D$ structures $M$ and $N$ over $\alg$ are \emph{homotopy equivalent} if there exist homomorphisms $f\colon M\to A\otimes N$ and $g\colon N\to A\otimes M$ such that $g\circ f\sim \id_M$ and $f\circ g\sim \id_N$. 
\end{defn}

The authors of \cite[Section 11.2]{bfh2} used bordered Floer homology to build a surgery exact triangle over $\F_2$. To do so, they construct bordered diagrams for three tori filled at slopes $\infty$,  $-1$, and $0$ respectively. They then build a short exact sequence between the type $D$ structures arising form these diagrams using homomorphisms $\overline{\phi}$ and $\overline{\psi}$.  By \cite[Proposition 2.36]{bfh2}, after tensoring with a type $A$ structure this short exact sequence induces a long exact sequence on homology.

\begin{figure}[h]
\centering
\labellist
	\pinlabel \textcolor{red}{$\alpha_1^a$} at 45 90
	\pinlabel \textcolor{red}{$\alpha_2^a$} at 90 40
	\pinlabel \textcolor{blue}{$\beta_1$} at 60 40
	\pinlabel $0$ at 78 3
	\pinlabel $1$ at 78 81
	\pinlabel $2$ at 3 81
	\pinlabel $3$ at 3 3
	\pinlabel $\HD_{\infty}:$ at -30 42
	\pinlabel $\HD_{-1}:$ at 125 42
	\pinlabel $\HD_{0}:$ at 280 42
	\pinlabel $r$ at 36 -5
	\pinlabel $s$ at -7 35
	\pinlabel $t$ at -7 50
	\pinlabel $a$ at 148 42
	\pinlabel $b$ at 197 -5
	\pinlabel $n$ at 300 48
	\pinlabel $p$ at 340 -5
	\pinlabel $q$ at 360 -5
	\pinlabel $\phi_1$ at 72 44
	\pinlabel $\phi_2$ at 14 70
	\pinlabel $\phi_3$ at 14 15
	\pinlabel $\phi_5$ at 220 70
	\pinlabel $\phi_4$ at 167 15
	\pinlabel $\phi_6$ at 320 70
	\pinlabel $\phi_7$ at 375 70
	\pinlabel $\phi_8$ at 350 13
\endlabellist
\hspace{1cm}
 \includegraphics[scale=1]{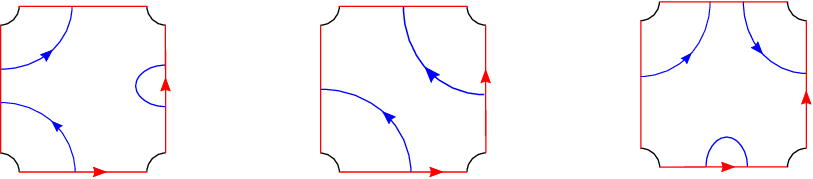}
 \vspace{.2 cm}
 	\caption{The bordered diagrams $\HD_\infty, \HD_{-1}, \HD_0$ for the three  tori filled at slopes $\infty$, $-1$, and $0$. The $\alpha$-arcs have been oriented consistently across the three diagrams. Further, the $\beta$-circles are oriented, and flows have been labeled.}\label{fig:surg-tri}
\end{figure}
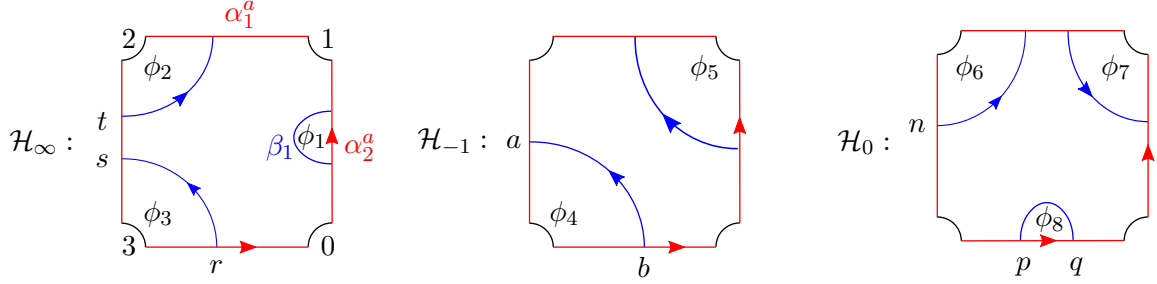

Suppose we are to mimic the approach in \cite[Section 11.2]{bfh2} to prove Conjecture~\ref{conj:tri}. In order to define type $D$ structures over $\Z$, we need to work with nice Heegaard diagrams. To obtain nice diagrams, we perform isotopies on the $\beta$ circles of two of the diagrams from  \cite[Section 11.2]{bfh2}; the resulting diagrams  $\HD_\infty$, $\HD_{-1}$, and $\HD_0$ are presented in Figure~\ref{fig:surg-tri}. We then fix a consistent orientation of the $\alpha$-curves across the three diagrams so that each has the same associated signed pointed matched circle. Finally, we orient the $\beta$-curves and choose a type $D$ sign assignment $S$ compatible with these choices. Let $\phi_1, \ldots, \phi_8$ be the flows on these three diagrams, as labeled on Figure~\ref{fig:surg-tri}. Note that $\phi_2$ and $\phi_6$ have the same associated formal flow, as do $\phi_3$ and $\phi_4$. Below, we will abuse notation and write $S(\phi_i)$ to mean the sign of the formal flow associated to $\phi_i$. In this notation, we have $S(\phi_2) = S(\phi_6)$ and $S(\phi_3) = S(\phi_4)$. With the above choices, we obtain the following type $D$ structures. 

\begin{itemize}
\item[] $\cfdhat(\HD_{\infty};\Z)\colon \hspace{.4cm} \delta^1(r)=S(\phi_2)\rho_2 \otimes t$  \hspace{.4cm} $\delta^1(s)=S(\phi_3)\rho_3 \otimes r +S(\phi_1) \otimes t$ \hspace{.4cm} $\delta^1(t)=0$
\vspace{.2cm}
\item[] $\cfdhat(\HD_{-1};\Z)\colon \hspace{.43cm}  \delta^1(a)=S(\phi_4)\rho_3 \otimes b + S(\phi_5)\rho_1 \otimes b$ \hspace{.4cm} $\delta^1(b)=0$ 
\vspace{.2cm}
\item[] $\cfdhat(\HD_0;\Z)\colon \hspace{.5cm}  \delta^1(n)=S(\phi_7)\rho_1 \otimes q$ \hspace{.4cm} $\delta^1(p)=S(\phi_6)\rho_2 \otimes n +S(\phi_8) \otimes q$ \hspace{.4cm} $\delta^1(q)=0$
\end{itemize}
We show that the maps used in \cite[Section 11.2]{bfh2} do not extend over $\Z$ to a short exact sequence between these structures.

We observe that $\cfdhat(\HD_{\infty};\Z)$ is homotopy equivalent to the type $D$ structure $N_{\infty}$ with a single generator $r$ and structure map $\delta^1(r)=-S(\phi_1)S(\phi_2)S(\phi_3)\rho_{23}\otimes r$.    Similarly, $\cfdhat(\HD_0;\Z)$  is homotopy equivalent to the type $D$ structure $N_0$ with a single generator $n$ and structure map $\delta^1(n)=-S(\phi_6)S(\phi_7)S(\phi_8)\rho_{12}\otimes n$. Last, let $N_{-1}$ denote $\cfdhat(\HD_{-1};\Z)$.
   We introduce  shortcut notation for the signs in the structure maps for $N_{\infty}$, $N_{-1}$, and $N_{0}$. Namely, let $S_1 = -S(\phi_1)S(\phi_2)S(\phi_3)$, $S_5 = -S(\phi_6)S(\phi_7)S(\phi_8)$, $S_2 = S(\phi_4)$, and $S_6 = S(\phi_5)$.

   Note that there exist formal flows $\phi_9$ and $\phi_{10}$ which form $\beta$-degenerations with the formal flows $F(\phi_1)$ and $F(\phi_8)$, respectively. Furthermore, the formal flows $F(\phi_5)$, $F(\phi_7)$, $\phi_9$, and $\phi_{10}$ form a type $D$ bordered square, with bordered flows of type $\rho_1$. This implies that 
   \begin{align*}
   S_1S_2S_5S_6 &=S(\phi_1)S(\phi_2)S(\phi_3)S(\phi_4)S(\phi_5)S(\phi_6)S(\phi_7)S(\phi_8)\\
   &= S(\phi_1)S(\phi_5)S(\phi_7)S(\phi_8)\\
   &= S(\phi_5)S(\phi_7)S(\phi_9)S(\phi_{10})\\
   &= (-1)^{\mid \rho_1 \mid+1}\\
   &= (-1)^{\mid \rho_2 \mid}.
   \end{align*}

We would like to build a short exact sequence of type $D$ homomorphisms $0 \to N_{\infty} \to N_{-1} \to N_{0} \to 0$, which would then induce a surgery exact triangle. We check if this is possible for maps obtained by introducing signs to the homomorphisms $\overline{\phi}$ and $\overline{\psi}$ used in \cite[Section 11.2]{bfh2}. 

Consider the maps $\phi: N_{\infty} \to N_{-1}$ given by $\phi(r)= S_4 \rho_2 \otimes a + S_3 b$ and $\psi:N_{-1} \to N_0$ given by $\psi(a)=S_8 n$ and $\psi(b)=S_7 \rho_2 n$, where $S_3$, $S_4$, $S_7$, and $S_8$ are in $\{\pm 1\}$. 

Suppose that $\phi$ is a $D$ homomorphism. Then we must have
\begin{align*} 
0 & =(\delta^1 \circ \phi - \phi \circ \delta^1)(r)\\
 &=(\mu_2 \otimes \id)\circ (\id \otimes \delta)(S_4 \rho_2 \otimes a + S_3 \otimes b)-(\mu_2 \otimes \id)\circ (\id \otimes \phi)(S_1\rho_{23}\otimes r) \\
						       &=(\mu_2 \otimes \id)((-1)^{\mid \rho_2 \mid}(S_4 \rho_2 \otimes((S_2 \rho_3 + S_6 \rho_1) \otimes b)-(\mu_2 \otimes \id)(S_1\rho_{23} \otimes (S_4 \rho_2 \otimes a + S_3 \otimes b) \\
						       &=(-1)^{\mid \rho_2 \mid}S_2S_4 \rho_{23} \otimes b - S_1S_3\rho_{23} \otimes b
\end{align*}
which implies $S_1S_2S_3S_4=(-1)^{\mid \rho_2 \mid}$. Now suppose $\psi$ is a $D$ homomorphism. Then we must have
\begin{align*} 
0&=(\delta^1 \circ \psi - \psi \circ \delta^1)(a)\\
 &=(\mu_2 \otimes \id)\circ (\id \otimes \delta)(S_8 \otimes n)-(\mu_2 \otimes \id)\circ (\id \otimes \phi)((S_2 \rho_3 + S_6 \rho_1) \otimes b \\
						       &=(\mu_2 \otimes \id)(S_8 \otimes S_5 \rho_{12}\otimes n)-(\mu_2 \otimes \id)(S_2 \rho_3 + S_6 \rho_1) \otimes S_7 \rho_2 \otimes n\\
						       &=S_5S_8 \rho_{12} \otimes n - S_6S_7\rho_{12} \otimes n
\end{align*}
which implies $S_5S_6S_7S_8=1$.  Combining these two sign relations, we see that we must have 
\begin{equation}\label{eqn:1}
\displaystyle\prod_{i=1}^8 S_i = (-1)^{\mid \rho_2 \mid}.
\end{equation}

On the other hand, if these two maps are to fit into a short exact sequence, we must have that $S_3S_4S_7S_8=-1$. Since $S_1S_2S_5S_6 =  (-1)^{\mid \rho_2 \mid}$, we must also have 
\begin{equation}\label{eqn:2}
\displaystyle\prod_{i=1}^8 S_i = (-1)^{\mid \rho_2 \mid +1}.
\end{equation} 
Equations~\ref{eqn:1} and \ref{eqn:2} give a contradiction, so $\phi$ and $\psi$ cannot simultaneously be type $D$ homomorphisms and satisfy exactness.

Thus, one would need to dig deeper for a proof of Conjecture~\ref{conj:tri}.

\bibliographystyle{alpha}
\bibliography{master}

\end{document}